\pdfoutput=1
\documentclass[11pt,a4paper,reqno]{amsart}
\usepackage[foot]{amsaddr}
\usepackage[utf8]{inputenc}

\usepackage[left=1in,right=1in,top=1in,bottom=1in]{geometry}

\usepackage{bbm}
\usepackage{amssymb}
\usepackage{amsmath}
\usepackage{mathtools}
\usepackage{amsthm}
\usepackage{array,multirow,graphicx}
\usepackage[new]{old-arrows}

\usepackage{stmaryrd}
\usepackage{quiver}

\usepackage{enumitem}
\usepackage{tabularx}

\usepackage[hidelinks]{hyperref}
\usepackage[english,noabbrev,sort,nameinlink]{cleveref}

\setlist[itemize]{noitemsep, topsep=0pt, leftmargin=1cm}

\newlist{enumalpha}{enumerate}{1}
\setlist[enumalpha, 1]{label=(\alph*)}
\newlist{enumroman}{enumerate}{1}
\setlist[enumroman, 1]{label=(\roman*)}

\newtheorem{theorem}{Theorem}
\newtheorem{proposition}[theorem]{Proposition}
\newtheorem{lemma}[theorem]{Lemma}
\newtheorem{corollary}[theorem]{Corollary}
\newtheorem{definition}[theorem]{Definition}
\theoremstyle{remark}
\newtheorem{example}[theorem]{Example}
\newtheorem{remark}[theorem]{Remark}
\numberwithin{theorem}{section}
\numberwithin{equation}{section}
\crefname{subsection}{subsection}{subsections}
\crefname{subsubsection}{paragraph}{paragraphs}
\newcommand{\npar}{\medskip\noindent\textbf}

\newcommand{\customref}[2]{\hyperref[#2]{#1}}
\newcommand\iref[2]{\customref{\Cref*{#1}~\ref*{#2}}{#2}}

\newcommand{\Z}{\mathbb Z}
\newcommand{\Q}{\mathbb Q}
\newcommand{\R}{\mathbb R}
\newcommand{\N}{\mathbb N}
\newcommand{\F}{\mathbb F}
\newcommand{\J}{\mathbb J}
\newcommand{\cF}{\mathcal F}
\newcommand{\cD}{\mathcal D}
\newcommand{\cR}{\mathcal R}
\newcommand{\cO}{\mathcal O}
\newcommand{\cG}{\mathcal G}
\newcommand{\cI}{\mathcal I}
\newcommand{\fF}{\mathfrak F}
\newcommand{\fg}{\mathfrak g}
\newcommand{\fh}{\mathfrak h}
\newcommand{\fn}{\mathfrak n}
\newcommand{\fm}{\mathfrak m}

\newcommand{\nbar}{\underline n}
\newcommand{\abar}{\underline a}

\newcommand{\ab}{\textnormal{ab}}
\newcommand{\alg}{\textnormal{alg}}
\newcommand{\Aut}{\mathrm{Aut}}

\newcommand{\Gal}{\mathrm{Gal}}
\newcommand{\GL}{\mathrm{GL}}
\DeclareMathOperator{\Hom}{Hom}
\newcommand{\perf}{\textnormal{perf}}
\DeclareMathOperator{\rank}{rank}
\newcommand{\sep}{\textnormal{sep}}

\newcommand{\Stab}{\mathrm{Stab}}
\DeclareMathOperator{\Spec}{Spec}
\DeclareMathOperator{\Tr}{Tr}
\DeclareMathOperator{\Ver}{Ver}

\newcommand{\dblquot}{\mskip-0.3\thinmuskip/\!\!/\!}
\newcommand{\Orb}[3]{#1\otimes#2\dblquot_{#3}}

\newcommand{\dd}{\mathrm{d}}
\DeclareMathOperator{\Ext}{\acute{E}tExt}
\DeclareMathOperator{\lastjump}{lastjump}
\DeclareMathOperator{\orb}{orb}
\newcommand{\bwp}{\boldsymbol{\wp}}
\DeclareMathOperator{\pr}{pr}

\newcommand{\andd}{\quad\quad\textnormal{and}\quad\quad}
\newcommand{\rprod}{\mathchoice{\prod_P\!'\ }{\prod_P'}{\prod_P'}{\prod_P'}}
\newcommand{\notp}{\N\setminus p\N}
\newcommand{\Onotp}{\{0\}\cup\notp}
\newcommand{\suchthat}[2]{\left\{#1\ \middle\vert\ #2\right\}}
\newcommand{\cardsuchthat}[2]{\left|\suchthat{#1}{#2}\right|}
\newcommand{\simto}{\overset\sim\to}
\newcommand{\llpar}{(\!(}
\newcommand{\rrpar}{)\!)}
\newcommand{\longtwoheadrightarrow}[1]
     {\foreach \n in {1,...,#1}{\relbar\joinrel\!}\twoheadrightarrow}

\title{Counting two-step nilpotent wildly ramified extensions of function fields}

\author{Fabian Gundlach}
\email{fabian.gundlach@uni-paderborn.de}
\author{Béranger Seguin}
\email{math@beranger-seguin.fr}
\address{Universität Paderborn, Fakultät EIM, Institut für Mathematik, Warburger Str. 100, 33098 Paderborn, Germany.}

\begin{document}

\begin{abstract}
  We study the asymptotic distribution of wildly ramified extensions of function fields in characteristic~$p > 2$, focusing on (certain) $p$-groups of nilpotency class at most $2$.
  Rather than the discriminant, we count extensions according to an invariant describing the last jump in the ramification filtration at each place.
  We prove a local--global principle relating the distribution of extensions over global function fields to their distribution over local fields, leading to an asymptotic formula for the number of extensions with a given global last-jump invariant.
  A key ingredient is Abrashkin's nilpotent Artin–-Schreier theory, which lets us parametrize extensions and obtain bounds on the ramification of local extensions by estimating the number of solutions to certain polynomial equations over finite fields.
\end{abstract}

\maketitle

{
  \setcounter{tocdepth}{1}
  \hypersetup{linkcolor=black}
  \tableofcontents{}
}

\section{Introduction}
\label{sn:introduction}

For the whole article, we fix a prime number~$p$.
If $F$ is a field, we denote by $\Gamma_F \coloneqq \Gal(F^\sep|F)$ its absolute Galois group.

\subsection{Context}

The asymptotic distribution of field extensions (usually counted by discriminant) is an actively studied topic.
Over a number field $F$, the case of abelian extensions was solved in \cite{wright} using the description of $\Gal(F^\ab|F)$ given by class field theory, and the case of extensions of degree $\leq 5$ was solved in \cite{davenport,bharga-quart,bharga-quint,bsw} using explicit parametrizations.
Significant progress has also been made for nilpotent extensions in \cite{klueners-malle-nilpotent-Galois-extensions,koymans-pagano-nilpotent}.
Although the general problem remains wide open, precise conjectures predict the expected distribution of class groups and Galois groups \cite{cohenlenstra,malle1}.
Over function fields, similar conjectures have been made when considering only tamely ramified extensions.
Results consistent with these expectations were obtained recently in \cite{EVW,ETW} by counting $\F_q$-points of moduli spaces of tamely ramified covers of the line (Hurwitz spaces).

Meanwhile, the distribution of wildly ramified extensions of function fields of characteristic~$p$ (both local and global) is much more mysterious, and there is not even a conjecture (see however the ``main speculation'' of \cite{darda-yasuda-wild}).
For abelian $p$-extensions, asymptotics have been described in \cite{lagemann1,lagemann2,kluners-muller,potthast,gundab}.
For non-abelian $p$-extensions, very little is known: only the local distribution of extensions whose Galois group is a certain generalization of Heisenberg groups (different from the generalized Heisenberg groups which we consider in \Cref{sn:more-groups}) has been described in \cite{muller-thesis}.
Constructions of moduli spaces for wildly ramified covers of curves, as in the articles \cite{harbater,bertin-mezard,fried-mezard,pries,zhang,moduli-torsors,danghippold}, highlight how different these spaces are from usual Hurwitz spaces, making it unlikely that the strategy of \cite{ETW} can be straightforwardly adapted.

The goal of this article is to study the distribution of extensions of function fields of characteristic~$p > 2$ for the $p$-groups arising as central extensions of abelian groups (i.e., $p$-groups of nilpotency class $\leq 2$).

\subsection{The last jump}
\label{subsn:last-jump}

In this paper, we do not count using the discriminant, but we use an invariant obtained by describing the last jump in the ramification filtration at each place, in the upper numbering.
Let $G$ be a finite group (seen as a discrete topological group).
If $\fF$ is a local field and $K$ is an (étale) $G$-extension of~$\fF$, corresponding to a continuous group homomorphism $\gamma : \Gamma_{\fF} \to G$ (see \Cref{subsn:extensions-and-cohomology}), we define an invariant $\lastjump(K|\fF)$ as follows:
\[
  \lastjump(K|\fF)
  \coloneqq
  \inf\suchthat{
      v\in\R_{\geq0}
  }{
    \gamma\big(
      \Gamma_{\fF}^{v}
    \big)
    = 1
  },
\]
where $\Gamma_{\fF}^{v}$ denotes the $v$-th ramification subgroup of the absolute Galois group of $\fF$, in the upper numbering (see \cite[Chap.~IV, \S 3]{serrecl}).
When the choice of the local field $\fF$ is implied by the context, we define $\lastjump(K) \coloneqq \lastjump(K|\fF)$.

Now, let $K$ be an (étale) $G$-extension of a global field~$F$.
At each place $P$ of $F$, we define $\lastjump_P(K) \coloneqq \lastjump(K \otimes_F F_P \mid F_P)$, where $F_P$ is the completion of $F$ at $P$.
These local invariants are then assembled into the following global invariant, mimicking the way invariants like the discriminant (or rather its degree) behave:
\begin{equation}
\label{eq:def-global-lastjump}
  \lastjump(K)
  \coloneqq
  \sum_P
    \deg(P)\cdot\lastjump_P(K).
\end{equation}
In its principle, this invariant is more closely related to the ``product of the ramified primes'' (used for example in \cite{wood-local}) than to the discriminant, but it adds weights to the primes depending on how wild the ramification at each prime is.
Note however that our invariant does not distinguish unramified primes from tamely ramified primes.

When $G$ is a $p$-group and $F$ has characteristic~$p$, tame ramification is impossible, so that $\lastjump(K)$ (resp.~$\lastjump_P(K)$) vanishes if and only if $K$ is an unramified extension (resp.~$P$ is unramified in $F$).

The invariants $\lastjump_P(K)$ and $\lastjump(K)$ are rational numbers whose denominators divide $|G|$.
When $G$ is an abelian group, the last jump is always an integer by the Hasse--Arf theorem, and it coincides with the exponent of the ``Artin--Schreier conductor'', an invariant for which asymptotics were given in \cite{gundab}.
As discussed there, counting by Artin--Schreier conductor gives simpler, more uniform results than counting by discriminant.

\subsection{Main results}

Assume now that $p > 2$, and let $G$ be a nontrivial finite $p$-group of nilpotency class at most~$2$, i.e., such that $[G,G] \subseteq Z(G)$.
As we explain in \Cref{lem:p-torsion-is-group}, the $p$-torsion elements of $G$ form a group, which we denote by~$G[p]$.
Let $F \coloneqq \F_q(T)$ be a rational (global) function field of characteristic~$p$, and let $\Ext(G,F)$ be the set of isomorphism classes of (étale) $G$-extensions of $F$ (see \Cref{subsn:extensions-and-cohomology}).
Our first main result is the following exact local--global principle:

\begin{theorem}[cf.\ \Cref{thm:local--global-principle}]
  \label{thm:intro-local--global}
  For every place $P$ of $F=\F_q(T)$, let $N_P\in\Q_{\geq0}$.
  Assume that $N_P=0$ for all but finitely many places.
  Then,
  \[
    \sum_{\substack{
      K \in \Ext(G,F):\\
      \forall P,\ \lastjump_P(K) = N_P
    }}
      \frac{1}{|\Aut(K)|}
    \quad=\quad
    \prod_P
      \ 
      \sum_{\substack{
        K_P \in \Ext(G,F_P):\\
        \lastjump(K_P) = N_P
      }}
        \frac{1}{|\Aut(K_P)|}.
  \]
\end{theorem}

This local--global principle allows us to determine the asymptotics of $G$-extensions of~$\F_q(T)$ using estimates for the number of extensions of the local fields $\F_{q^d}\llpar T\rrpar$ for $d \geq 1$.
The following global asymptotics are the main results of this article.

\begin{theorem}[cf.\ \Cref{thm:proof-counting}]
  \label{thm:intro-counting}
  Let
  \[
    r \coloneqq \log_p|G[p]|
    \andd
    M \coloneqq
    \begin{cases}
      1 & \textnormal{if } G[p]\textnormal{ is abelian},\\
      1+p^{-1} & \textnormal{otherwise}.\\
    \end{cases}
  \]
  If $G[p]$ is non-abelian, assume that $|G[p]|\leq p^{p-1}$.
  Assume moreover that $q$ is a large enough power of~$p$ (depending on the group $G$).
  Then, there is a function
  $
    C:\Q/M\Z\rightarrow\R_{\geq0}
  $
  with $C(0) \neq 0$, such that for rational $N\rightarrow\infty$, we have
  \[
    \sum_{\substack{
      K\in\Ext(G,F):\\
      \lastjump(K)=N
    }}
      \frac1{|\Aut(K)|}
    \quad=\quad
    C\!\left(N\bmod M\right)\cdot
    q^{\frac{r+1}{M}\cdot N}
    +
    o\!\left(
      q^{\frac{r+1}{M}\cdot N}
    \right).
  \]
\end{theorem}

The hypothesis that $q$ is large is needed due to a technical limitation of our local counting methods.
An explicit lower bound on $q$ can be deduced from the proof in \Cref{subsn:proof-main-thm}, but we presume that the conclusion may hold for all $q$.
We show that the conclusion indeed holds for all $q$ if $G$ has exponent~$p$, or more generally if $G$ satisfies the hypothesis of \Cref{prop:better-bound}.

For concrete groups $G$, it can be possible to overcome the restriction that $|G[p]|\leq p^{p-1}$.
As an illustration, we carry out the necessary computations for the generalized Heisenberg groups $G = H_k(\F_p)$ defined in \Cref{def:heisenberg-group}:

\begin{theorem}[cf.\ \Cref{thm:heisenberg-count}]
  \label{thm:intro-heisenberg}
  There are explicit constants $A \in \Q_{>0}, B \in \N_{>0}, M \in \Q_{>0}$ (cf.~\Cref{subsn:heisenberg-global-asymptotics}) and a function $C:\Q/M\Z\rightarrow\R_{\geq0}$ with $C(0)\neq0$, such that for rational $N\rightarrow\infty$, we have
  \[
    \sum_{\substack{
      K\in\Ext(G,F):\\
      \lastjump(K)=N
    }}
      \frac1{|\Aut(K)|}
    \quad=\quad
    C\!\left(N\bmod M\right)\cdot
    q^{AN}
    N^{B-1}
    +
    o\!\left(
      q^{AN}
      N^{B-1}
    \right).
  \]
  Moreover, if $p \geq 5$, then $B = 1$ (cf.~\Cref{prop:B-is-often-1}).
\end{theorem}

\begin{remark}
  In principle, asymptotics for the number of field extensions (excluding non-simple étale algebras) could be obtained by inclusion--exclusion over subgroups of $G$.
  Note also the following criterion for surjectivity, which follows from \cite[Lemma~5.9]{cgt}: when $G$ is a nilpotent group, a homomorphism $\Gamma_F \to G$ is surjective if and only if the induced homomorphism $\Gamma_F^{\ab} \to G^{\ab}$ is surjective.
  (This fact was also used in \cite{koymans-pagano-nilpotent}.)
\end{remark}

\subsection{Strategy and outline of the paper}
We now summarize the content of each section.
This outline also serves as an explanation of our general strategy.

A key tool in our proofs is Abrashkin's \emph{nilpotent Artin--Schreier theory} from \cite{abrashkin-ramification-filtration-3} (simpler forms of this theory for $p$-groups of exponent $p$ and for $p$-groups of nilpotency class $\leq 2$, were respectively introduced in \cite{abrashkin-ramification-filtration} and in \cite{abrashkin-ramification-filtration-2}).
\Cref{sn:preliminaries} is essentially a reformulation of this theory.
We review a general principle for the parametrization of extensions (\Cref{thm:param-galrep}), and we explain how to apply this principle using Witt vectors and Lie algebras (\Cref{thm:parametrization-abelian}, \Cref{thm:parametrization}).

In \Cref{sn:local}, we assume that the base field is a local function field of characteristic~$p$.
We refine the parametrization by describing an approximate fundamental domain, making the parametrization finite-to-one (\Cref{thm:local-fundamental-domain}), and we use Abrashkin's description of the ramification filtration from \cite{abrashkin-ramification-filtration-3} to characterize extensions with a given last jump (\Cref{def:property-Jv}, \Cref{thm:lastjump-Jv}).

In \Cref{sn:local-counting}, we analyze the equations obtained in the previous section in order to obtain bounds on the number of their solutions, and hence on the number of local extensions with bounded last jump (\Cref{thm:local-counting}).
Locally, the bounds we obtain are rough: we have precise estimates of the number of extensions for small values of the last jump, but only upper bounds for large values.

In \Cref{sn:global}, we assume that the base field is a rational global function field of characteristic~$p$.
We prove an exact local--global principle for the last jump (\Cref{thm:local--global-principle}, which is \Cref{thm:intro-local--global}), as well as a general analytic lemma allowing one to deduce global asymptotics from local estimates (\Cref{lem:analytic-lemma}).
These two tools, combined with the results of the previous section, let us prove our main counting theorem (\Cref{thm:proof-counting}, which is \Cref{thm:intro-counting}) in \Cref{subsn:proof-main-thm}.

Handling the case where the $p$-torsion subgroup~$G[p]$ is non-abelian of size $\geq p^p$ requires a more careful analysis of the equations.
In \Cref{sn:more-groups}, we consider the infinite family of (non-abelian) Heisenberg groups $H_k(\F_p)$ (of exponent~$p$), for which the hypotheses of \Cref{thm:intro-counting} need not be satisfied, and we obtain global asymptotics for the number of $H_k(\F_p)$-extensions of~$\F_q(T)$ (see \Cref{thm:heisenberg-count}, which is \Cref{thm:intro-heisenberg}).
This example illustrates what is needed to solve the problem for more general $p$-groups of exponent~$p$: a key step is to estimate the number of elements commuting with their Frobenius in a certain Lie algebra.

\subsection{Possible improvements}

Refinements could come from better understanding the geometry of the varieties defined by the equations of \Cref{def:property-Jv} (cf.~\Cref{rk:mod-space}) in order to improve our bounds on their number of $\F_q$-points (for example, using the Grothendieck--Lefschetz trace formula as in \cite{ETW} and \cite{langweiltordu}).
Another interesting question, which could be related to the geometric point of view on the problem, is whether the generating functions associated to our counting problems are rational functions, as such a phenomenon was observed in \cite{gundab} for abelian $p$-extensions.

\medskip

Although we consider only groups of nilpotency class $\leq 2$, Abrashkin's theory applies in principle to all $p$-groups of nilpotency class $< p$.
However, practical difficulties arise when using it for counting purposes due to the complexity of the description of the ramification filtration.
The equivalent formulations given by Abrashkin in \cite{abrashkin-differential} (in terms of the canonical connection on $\varphi$-modules from \cite[2.2.4]{fontaine}) and in \cite{abrashkin-deformations-2} might be better suited for counting.
Moreover, we have no reason to think that our exact local--global principle (\Cref{thm:local--global-principle}) holds for groups of higher nilpotency class.
Note also that, in characteristic $2$, Abrashkin's theory cannot be used to deal with any non-abelian group.

It should be noted that parametrizing extensions is not the major difficulty (see \Cref{subsn:parametrization}).
For instance, $\GL_n(\F_p)$-extensions are parametrized by étale $\varphi$-modules of dimension~$n$, and over local function fields the explicit (``group-theoretic'') description of the absolute Galois group given in \cite{koch-absgal}  (without its ramification filtration!) offers in theory easy access to all extensions.
The main challenge lies in obtaining a sufficiently good description of the ramification filtration of these extensions (e.g., an expression for the discriminant/last jump/$\dots$ in terms of the parametrization) to make counting possible.
In that regard, the results of \cite{imai-wild-ramification-groups} are interesting, as they give an example where the ramification filtration is described even when the nilpotency class equals~$p$, extending slightly beyond the scope of Abrashkin's theory.
For groups of nilpotency class $2$, the case $p = 2$ is also considered in \cite{abrashkin-grothendieck}.

\medskip

It is natural to ask whether other invariants satisfy the exact local--global principle from \Cref{thm:intro-local--global}.
For the discriminant, we doubt this --- at least, our method of proof does not apply.
Nevertheless, it might be feasible to prove an approximate local--global principle sufficient for a statement analogous to \Cref{thm:intro-counting}.

\medskip

In our main results, the base field is always a rational (global) function field $\F_q(T)$, with trivial class group.
For non-rational base fields $F \neq \F_q(T)$, the exact local--global principle can famously fail even when only considering unramified abelian extensions.
This failure can be quantified for abelian $p$-groups $G$ using Selmer groups (see \cite{lagemann1,lagemann2,potthast}) and it seems plausible that the same methods can be used to generalize our main counting results (\Cref{thm:intro-counting,thm:intro-heisenberg}) to non-rational base fields.

\subsection{Terminology and notation}

If $X$ is a set, we denote its cardinality by $|X|$.
When $x$ is an element of a set $X$ on which a group acts, we denote by $[x]$ the orbit of $x$, usually without specifying which group is acting when the context makes it clear.

Throughout the article, $\sigma$ always denotes the absolute Frobenius endomorphism $x \mapsto x^p$ of a commutative ring~$R$ of characteristic~$p$, and we also call $\sigma$ the endomorphism induced by $\sigma$ for every object constructed functorially from $R$.

In this article, Galois cohomology sets $H^i(\Gamma_F, G)$ are defined using \emph{continuous} group cohomology (absolute Galois groups are equipped with the profinite topology, and $G$ is a topological group on which $\Gamma_F$ acts continuously).
These cohomology sets are themselves groups if $G$ is abelian.
If $G$ is non-abelian, they are defined only if $i \in \{0,1\}$, and are pointed sets with no natural group structure (cf.~\cite[Chap.~VII, Annexe]{serrecl}).

Many notations are introduced in the text.
These notations, together with references to their definition and with a short description, are listed in \hyperref[notation-chart]{an appendix} (p.~\pageref{notation-chart}).

\subsection{Acknowledgments}

This work was supported by the Deutsche Forschungsgemeinschaft (DFG, German Research Foundation) --- Project-ID 491392403 --- TRR 358 (project A4).
The authors are grateful to Xujia Chen, Kiran Kedlaya, Jürgen Klüners and Nicolas Potthast for helpful discussions, and moreover to Victor Abrashkin, Jordan Ellenberg, Carlo Pagano, and Takehiko Yasuda for their comments on an earlier draft.

\section{Preliminaries}
\label{sn:preliminaries}

In this section, we review known results concerning the following topics:
\begin{itemize}
  \item
    the parametrization of extensions in characteristic~$p$ (\Cref{subsn:extensions-and-cohomology,subsn:parametrization});
  \item
    Perfect closures of rings in characteristic $p$ (\Cref{subsn:perfect-closure});
  \item
    Witt vectors and their Galois cohomology (\Cref{subsn:witt-vectors});
  \item
    the Lazard correspondence, which relates each finite $p$-group $G$ of nilpotency class $< p$ to a finite Lie $\Z_p$-algebra (\Cref{subsn:lie-algebras});
  \item
    the approach developed by Abrashkin under the name \emph{nilpotent Artin--Schreier theory}, which parametrizes $G$-extensions in characteristic~$p$ (\Cref{subsn:nilpotent-artinschreier}).
\end{itemize}
Our explanations loosely follow those given by Abrashkin in \cite[\S1]{abrashkin-ramification-filtration-3}, but with some differences (cf.~\Cref{rk:no-cohen-rings}).
The theory presented in \Cref{subsn:nilpotent-artinschreier} takes a simpler form when $G$ is a $p$-group of exponent~$p$, as perfect closures and Witt vectors are not required.
See \Cref{subsn:nilpas-expp} for a brief overview of the simplified theory.

\subsection{Extensions and cohomology classes}
\label{subsn:extensions-and-cohomology}

Let $G$ be a finite group and $F$ be a field.
By a \emph{$G$-extension} of $F$, we mean an étale $F$-algebra~$K$ together with an action of $G$ such that there is a $G$-equivariant $F^\sep$-algebra isomorphism between $K\otimes_F F^\sep$ and the ring of maps $f : G \to F^\sep$ on which $G$ acts via $(g.f)(h)=f(hg)$.
An \emph{isomorphism} between two $G$-extensions of $F$ is a $G$-equivariant isomorphism between the corresponding étale $F$-algebras.
We denote the set of isomorphism classes of $G$-extensions of~$F$ by~$\Ext(G,F)$, often confusing an element of $\Ext(G,F)$ with a representative $K$ of the corresponding isomorphism class, and denoting by $\Aut(K)$ the group of its $G$-equivariant $F$-algebra automorphisms.
We recall the well-known relationship between $\Ext(G,F)$ and the set $H^1(\Gamma_F, G)$ of $G$-conjugacy classes of continuous group homomorphisms $\gamma\in\Hom(\Gamma_F, G)$, seeing $G$ as a discrete topological group equipped with the trivial $\Gamma_F$-action:

\begin{lemma}
  \label{lem:etext-bij-cohomology}
  There is a bijection
  \[
    \Ext(G, F) \overset\sim\longleftrightarrow H^1(\Gamma_F, G)
  \]
  such that, if $K\in\Ext(G, F)$ corresponds to $[\gamma]\in H^1(\Gamma_F, G)$, then:
  \begin{enumalpha}
    \item
      For the action of $G$ on $\Hom(\Gamma_F, G)$ by conjugation, we have
      $
        \Aut(K) \simeq \Stab_G(\gamma)
      $.
    \item
      $K$ is a field if and only if $\gamma:\Gamma_F\rightarrow G$ is surjective.
    \item
      $K$ is the trivial $G$-extension of $F$ (the ring of maps $G \to F$) if and only if $\gamma = 1$.
  \end{enumalpha}
\end{lemma}

\begin{definition}[Twisting]
  \label{def:twisting}
  Let $N$ be a subgroup of the center of $G$.
  The \emph{twist} of $\gamma\in\Hom(\Gamma_F,G)$ by $\delta\in\Hom(\Gamma_F,N)$ is the point-wise product $\gamma\cdot\delta\in\Hom(\Gamma_F,G)$.
\end{definition}

\begin{remark}
  \label{rmk:twisting}
  Denoting the projection $G\twoheadrightarrow G/N$ by $\pi$, the twisting operation lets us define a bijection for any given $\gamma \in \Hom(\Gamma_F, G)$:
  {
    \setlength\arraycolsep{3pt}
    \[\begin{matrix}
      \Hom(\Gamma_F, N) &
      \overset\sim\longrightarrow &
      \{\gamma'\in\Hom(\Gamma_F, G) \mid \pi\circ\gamma' = \pi\circ\gamma\},
      \\[0.5em]
      \delta &
      \longmapsto &
      \gamma\cdot\delta.
    \end{matrix}\]
  }%
  Said differently, twisting defines a free action of the abelian group $\Hom(\Gamma_F, N)$ on the set $\Hom(\Gamma_F, G)$, and each orbit $[\gamma]$ is determined by the (well-defined) element $\pi \circ \gamma \in \Hom(\Gamma_F, G/N)$.
\end{remark}

\subsection{Parametrization of extensions.}
\label{subsn:parametrization}

We fix a field~$F$.
In this subsection, we describe a general principle for parametrizing elements of~$H^1(\Gamma_F, G)$.

Let $G_{F^\sep}$ be a topological group equipped with a continuous action of $\Gamma_F$ and with a $\Gamma_F$-equivariant group homomorphism $\sigma : G_{F^\sep} \to G_{F^\sep}$.
We denote by $G$ the subgroup of $G_{F^\sep}$ consisting of fixed points of $\sigma$, and by $G_F$ the closed subgroup of~$G_{F^\sep}$ consisting of $\Gamma_F$-invariant elements.
Note that, as the actions of $\sigma$ and $\Gamma_F$ commute, we have $\sigma(G_F) \subseteq G_F$.
We define a left action of $G_{F^\sep}$ on itself via the formula:
\[
  g.m \coloneqq \sigma(g) m g^{-1}.
\]
This action restricts to an action of $G_F$ on itself, whose set of orbits we denote by $G_F \dblquot_{G_F}$.
The \emph{multiplicative Artin--Schreier map} is the map $\bwp : G_{F^\sep} \to G_{F^\sep}$ defined by
\[
  \bwp(g) \coloneqq \sigma(g)g^{-1} = g.1.
\]
Note that $\bwp(g) = 1$ if and only if $g \in G$.
If $G_{F^\sep}$ is abelian, then $g.m = \wp(g)m$ for all $g,m\in G_{F^\sep}$, so that the set of orbits $G_F\dblquot_{G_F}$ is the quotient group $G_F/\wp(G_F)$.

\begin{proposition}
  \label{thm:param-galrep}
  Assume that the following properties hold:
  \begin{enumroman}
    \item
      \label{thm:param-galrep-i}
      $G \subseteq G_F$;
    \item
      \label{thm:param-galrep-ii}
      $G_F \subseteq \bwp(G_{F^\sep})$;
    \item
      \label{thm:param-galrep-iii}
      The map of pointed sets $H^1(\Gamma_F, G) \to H^1(\Gamma_F, G_{F^\sep})$ is trivial.
  \end{enumroman}
  Then, there is a bijection
  \[
    \begin{matrix}
      \orb : &
      H^1(\Gamma_F, G)&
      \overset\sim\longrightarrow &
      G_F \dblquot_{G_F}
      \\ &
      [\tau \mapsto g^{-1} \tau(g)] &
      \longmapsto &
      [\bwp(g)] &&
      \textnormal{for any } g \in \bwp^{-1}(G_F).
    \end{matrix}
  \]
  Moreover, if $\orb([\gamma]) = [m]$, then $\Stab_G(\gamma)=\Stab_{G_F}(m)$.
\end{proposition}

\begin{proof}
  As in \cite[Proposition~1]{arithmetic-invariant}, there is a bijection between the kernel of the map of pointed sets $H^1(\Gamma_F, \Stab_{G_{F^\sep}}(1)) \to H^1(\Gamma_F, G_{F^\sep})$ and the set $(G_F \cap \bwp(G_{F^\sep})) \dblquot_{G_F}$ of $G_F$-orbits of elements of $G_F \cap \bwp(G_{F^\sep})$.
  For the action of $G_{F^\sep}$ on itself, the stabilizer of~$1$ is~$G$ by definition, \ref{thm:param-galrep-i} implies that $\Gamma_F$ does act trivially on $G$, \ref{thm:param-galrep-ii} implies that the orbit $\bwp(G_{F^\sep})$ of~$1$ contains~$G_F$ (so that $(G_F \cap \bwp(G_{F^\sep})) \dblquot_{G_F} = G_F \dblquot_{G_F}$), and \ref{thm:param-galrep-iii} implies that $\ker(H^1(\Gamma_F, G) \to H^1(\Gamma_F, G_{F^\sep})) = H^1(\Gamma_F, G)$.
  We have the desired bijection, and its definition matches the one given here.
  The equality between stabilizers is easily verified.
\end{proof}

Assume that $F$ has characteristic $p$.
We illustrate the principle with fundamental examples:

\begin{example}
  \label{ex:alg-gp}
  Let $\cG$ be an algebraic group over~$\F_p$ and $G = \cG(\F_p)$.
  Then, the group~$\cG(F^\sep)$, equipped with its natural $\Gamma_F$-action and absolute Frobenius endomorphism $\sigma$, is a good candidate to apply \Cref{thm:param-galrep}, as it always satisfies condition \ref{thm:param-galrep-i}, and moreover $G_F = \cG(F)$.
  However, conditions \ref{thm:param-galrep-ii} and \ref{thm:param-galrep-iii} still need to be verified.
\end{example}

\begin{example}[Artin--Schreier theory]
  Consider the group $G_{F^\sep} = F^\sep$ (the case $\cG = \mathbb G_a$ of \Cref{ex:alg-gp}).
  The subgroup of $\Gamma_F$-invariant elements is $G_F = F$.
  The subgroup of elements fixed by the Frobenius homomorphism $\sigma(x)=x^p$ is $G = \F_p$.
  The map $\wp:F^\sep\rightarrow F^\sep$, $x\mapsto x^p-x$ is surjective, and we have $H^1(\Gamma_F, F^\sep) = 0$ by \cite[Chap.~X, \S~1, Prop.~1]{serrecl}.
  Hence, \Cref{thm:param-galrep} yields the well-known bijection $H^1(\Gamma_F,\Z/p\Z) \stackrel\sim\rightarrow F/\wp(F)$.
\end{example}

\begin{example}
  Let $G = \GL_n(\F_p)$ and $G_{F^\sep} = \GL_n(F^\sep)$ (the case $\cG = \GL_n$ of \Cref{ex:alg-gp}).
  The subgroup of $\Gamma_F$-invariant elements is $\GL_n(F)$.
  In this case, $H^1(\Gamma_F, G_{F^\sep})$ vanishes by a generalization of Hilbert's Theorem 90 (cf.~\cite[Chap.~X, \S 1, Prop.~3]{serrecl}), and the map~$\bwp$ is surjective on $F^\sep$-points as it comes from an étale morphism $(\GL_n)_{\F_p} \to (\GL_n)_{\F_p}$.
  We retrieve the theory of \emph{étale $\varphi$-modules} of dimension $n$ (cf.~\cite[Subsection~3.2]{fontaineouyang}, and notably their Remark~3.24).
  In particular, the case $n=1$ gives a special case of Kummer theory for the parametrization of $\Z/(q-1)\Z$-extensions (the case $\cG = \mathbb G_m$ of \Cref{ex:alg-gp}).
\end{example}

\subsection{Perfect closure}
\label{subsn:perfect-closure}

Let $R$ be an integral domain of characteristic~$p$, so that its Frobenius endomorphism $\sigma : R \to R$, $x \mapsto x^p$ is injective.
We define the \emph{perfect closure}~$R^\perf$ of $R$ as the following direct limit:
\[
  R^\perf \coloneqq \varinjlim (R \overset\sigma\to R \overset\sigma\to R \to \ldots).
\]
In other words, any element of $R^\perf$ is a formal $p^n$-th root of an element of $R$ for some $n \geq 0$.
Since $\sigma : R \to R$ is injective, the canonical map $R \to R^\perf$ is an injection, and we regard $R$ as a subring of $R^\perf$ via this map.
The absolute Frobenius endomorphism $\sigma$ of $R^\perf$ extends the Frobenius of $R$ and is an automorphism, so that $R^\perf$ is a perfect ring containing $R$.
Moreover, that construction is functorial in $R$.

Let now $F$ be a field of characteristic~$p$, and let $F^\sep$ be a separable closure of $F$.
The field $F^\alg \coloneqq (F^\sep)^\perf$ is an algebraic closure of $F$, and it is also a separable closure of $F^\perf$.
When we refer to the absolute Galois group $\Gamma_{F^\perf}$, we mean $\Gal(F^\alg|F^\perf)$.
A key property of the perfect closure is that its separable extensions correspond bijectively to those of~$F$:

\begin{lemma}
  \label{lem:galgp-perfclos}
  The restriction map $\Gamma_F \to \Gamma_{F^\perf}$ is an isomorphism of topological groups.
\end{lemma}

\begin{proof}
  The preimage of an automorphism $\tau$ of $F^\alg|F^\perf$ is the automorphism of $F^\sep|F$ sending~$\sqrt[p^k]{x}$ to $\sqrt[p^k]{\tau(x)}$ for all $x\in F$.
  We leave it to the reader to check that this inverse map is continuous.
\end{proof}

We let $\Gamma_F$ act on $F^\alg = (F^\perf)^\sep$ via the isomorphism $\Gamma_F \simeq \Gamma_{F^\perf}$.

For any finite group $G$, composition with the restriction map $\Gamma_F \to \Gamma_{F^\perf}$ induces a bijection between the pointed sets $H^1(\Gamma_{F^\perf}, G)$ and~$H^1(\Gamma_F, G)$, which by \Cref{lem:etext-bij-cohomology} means that there is a bijection $\Ext(G, F^\perf) \simeq \Ext(G, F)$.

We denote by $\wp$ the $\F_p$-linear endomorphism $x \mapsto \sigma(x) - x$.
For $G = \Z/p\Z$, the bijection $\Ext(\Z/p\Z, F^\perf) \simeq \Ext(\Z/p\Z, F)$ turns into a bijection $F^\perf/\wp(F^\perf) \simeq F/\wp(F)$ using Artin--Schreier theory.
This can also be observed directly:

\begin{lemma}
  \label{lem:artsch-perfclos}
  The map $F/\wp(F) \to F^\perf/\wp(F^\perf)$ induced by the inclusion $F \subseteq F^\perf$ is an isomorphism.
\end{lemma}

\begin{proof}
  The injectivity amounts to the inclusion $F \cap \wp(F^\perf) \subseteq \wp(F)$.
  Let $x \in F$.
  The equation $y^p - y = x$ is a separable equation in the variable $y$.
  Since the extension $F^\perf|F$ is purely inseparable, any solution $y$ in $F^\perf$ has to lie in $F$.
  Therefore, if $x \in \wp(F^\perf)$, then $x \in \wp(F)$.

  We now check surjectivity.
  Let $x \in F^\perf$.
  By definition of $F^\perf$, there is some $n \geq 0$ such that $\sigma^n(x) \in F$.
  By definition of $\wp$, we have $z \equiv \sigma(z) \mod \wp(F^\perf)$ for all $z \in F^\perf$, and in particular $x \equiv \sigma(x) \equiv \ldots \equiv \sigma^n(x)$.
\end{proof}

\subsection{Witt vectors}
\label{subsn:witt-vectors}

To deal with $p$-groups of exponent larger than~$p$, we use $p$-typical Witt vectors.
We have a functor
\[
  W:
  \{\textnormal{rings of characteristic~$p$}\}
  \longrightarrow
  \{\textnormal{$\Z_p$-algebras}\}
\]
mapping a ring $R$ to the corresponding ring of Witt vectors $W(R)$, whose elements can be represented as vectors of infinite length with coordinates in $R$.
Let $R$ be an integral domain of characteristic $p$.
The operation consisting of adding a leading zero to a Witt vector, shifting all other coordinates one place to the right, defines the ($\Z_p$-linear) Verschiebung operator $\Ver: W(R) \rightarrow W(R)$.
Moreover, the absolute Frobenius endomorphism $\sigma: x\mapsto x^p$ of $R$ induces a coordinatewise endomorphism $\sigma$ of $W(R)$, fixing exactly the elements of $W(\F_p) = \Z_p$.
We denote by $\wp$ the $\Z_p$-linear map $W(R) \to W(R)$, $x \mapsto \sigma(x) - x$, whose kernel is~$\Z_p$.
The endomorphism of $R$ given by multiplication by $p$ coincides with ${\Ver} \circ \sigma = \sigma \circ \Ver$.
The ring of Witt vectors of length~$n$ over $R$ is the $\Z/p^n\Z$-algebra
\[
  W_n(R) \coloneqq W(R) / \Ver^n(W(R)),
\]
which coincides with $W(R)/p^nW(R)$ when $R$ is perfect.
We have $W(R) = \varprojlim_n W_n(R)$.

We now fix a field $F$ of characteristic~$p$.
The action of $\Gamma_F$ on $F^\alg = (F^\sep)^\perf$ induces a coordinatewise action on $W(F^\alg)$, fixing exactly the elements of $W(F^\perf)$.

\begin{remark}[Artin--Schreier--Witt theory]
  The additive group $G_{F^\sep} \coloneq W_n(F^\sep)$ satisfies the hypotheses of \Cref{thm:param-galrep} with $G_F = W_n(F)$ and $G = W_n(\F_p) = \Z/p^n\Z$.
  Hence, we have a bijection $H^1(\Gamma_F, \Z/p^n\Z) \simto W_n(F)/\wp(W_n(F))$.
  (See for example \cite[Lemmas~9 and~10, and Proposition~11 in Section~4.10]{bosch-algebra}.)
\end{remark}

\begin{remark}
  \label{rmk:flat-module}
  The ring $W(F^\perf)$ is a torsion-free $\Z_p$-module, hence flat, so any short exact sequence of $\Z_p$-modules
  $
    0 \rightarrow \fn \rightarrow \fg \rightarrow \fg/\fn \rightarrow 0
  $
  induces an exact sequence of $W(F^\perf)$-modules
  \[
    0 \rightarrow \fn\otimes_{\Z_p} W(F^\perf) \rightarrow \fg\otimes_{\Z_p} W(F^\perf) \rightarrow (\fg/\fn)\otimes_{\Z_p} W(F^\perf) \rightarrow 0.
  \]
\end{remark}

\begin{lemma}
  \label{lem:witt-props}
  The following properties hold for any finite $\Z_p$-module $\fg$:
  \begin{enumroman}
    \item
      \label{lem:witt-props-0}
      The Artin--Schreier map $\wp:\fg\otimes W(F^\alg)\rightarrow \fg\otimes W(F^\alg)$, $g\mapsto\sigma(g)-g$, is surjective.
    \item
      \label{lem:witt-props-i}
      The natural map $\fg \otimes W(F^\perf) \to \fg \otimes W(F^\alg)$ is injective.
    \item
      \label{lem:witt-props-ii}
      $(\fg \otimes W(F^\alg))^{\Gamma_F} = \fg \otimes W(F^\perf)$.
    \item
      \label{lem:witt-props-iii}
      $(\fg \otimes W(F^\alg))^{\sigma}
      = \fg$.
    \item
      \label{lem:witt-props-iv}
      $H^1(\Gamma_F, \fg\otimes W(F^\alg)) = 0$.
  \end{enumroman}
\end{lemma}

\begin{proof}
  We can assume without loss of generality that $\fg = \Z/p^n\Z$ with $n \geq 1$, as every finite $\Z_p$-module is a direct sum of such factors.
  We have $\fg \otimes W(F^\perf) = W_n(F^\perf)$ and $\fg \otimes W(F^\alg) = W_n(F^\alg)$.
  Then:
  \begin{enumroman}
    \item
      Apply \cite[Lemma~9 in Section~4.10]{bosch-algebra} to $K = F^\perf$ (recall that $F^\alg = (F^\perf)^\sep$).
    \item
      Clear.
    \item
      Clear (the action of $\Gamma_F\simeq\Gamma_{F^\perf}$ on $W_n(F^\alg)$ is coordinatewise).
    \item
      Clear (the action of $\sigma$ on $W_n(F^\alg)$ is coordinatewise, and $W_n(\F_p) = \Z/p^n\Z$).
    \item
      This follows from \cite[Proposition~11 in Section~4.10]{bosch-algebra}, applied to $K = F^\perf$.
      \qedhere
  \end{enumroman}
\end{proof}

\Cref{lem:witt-props} lets us apply \Cref{thm:param-galrep} to obtain the following parametrization of $G$-extensions of $F$, where $G$ is any finite abelian $p$-group (corresponding to $(\fg,+)$ for a finite $\Z_p$-module~$\fg$):

\begin{corollary}
  \label{thm:parametrization-abelian}
  Let $\fg$ be a finite $\Z_p$-module, and let $\wp:\fg\otimes W(F^\perf)\to \fg\otimes W(F^\perf)$ be the (additive) Artin--Schreier map $g\mapsto \sigma(g)-g$.
  We have a
  bijection:
  \[
    H^1(\Gamma_F, (\fg, +)) \stackrel\sim\longrightarrow \fg\otimes W(F^\perf) / \wp(\fg\otimes W(F^\perf)).
  \]
\end{corollary}

The goal of the following subsections is to obtain a non-abelian version of \Cref{thm:parametrization-abelian} --- this will be \Cref{thm:parametrization}.

\subsection{Lie $\Z_p$-algebras}
\label{subsn:lie-algebras}

\subsubsection{Definitions}

A \emph{Lie $\Z_p$-algebra} is a $\Z_p$-module $\fg$ equipped with a ($\Z_p$-bilinear, alternating) Lie bracket $[-, -] : \fg^2 \to \fg$ satisfying the Jacobi identity $[[a,b],c]+[[b,c],a]+[[c,a],b]=0$.
Let~$\fg$ be a Lie $\Z_p$-algebra.
We say that $\fg$ is \emph{abelian} if its Lie bracket is identically zero.
An \emph{ideal} of $\fg$ is a submodule $\fn \subseteq \fg$ such that $[\fg, \fn] \subseteq \fn$.
We can form the quotient of $\fg$ by an ideal~$\fn$ to obtain a Lie algebra $\fg/\fn$.
The \emph{center} of $\fg$ is the ideal $Z(\fg)$ formed of elements $x$ such that $[\fg,x]=0$.
For elements $x_1,\dots,x_n \in \fg$, we use the notation
\[
  [x_1,\dots,x_n]
  \coloneqq
  \underbrace{
    [[\cdots[[
  }_{n-1}
    x_1,x_2],
    x_3],
    \dots,
  x_n].
\]
We say that $\fg$ is \emph{nilpotent} if there is an integer $n$ such that $[x_1, \ldots, x_{n+1}]$ vanishes for all $x_1, \ldots, x_{n+1} \in \fg$.
The smallest such $n$ is then the \emph{nilpotency class} of $\fg$.
For instance, the zero Lie algebra is the unique Lie algebra of class $0$, and the Lie algebras of class $1$ are exactly the nonzero abelian Lie algebras.
The center of a nonzero nilpotent Lie algebra $\fg$ is a nonzero abelian Lie algebra.
In particular, if $\fg\neq0$ is finite and nilpotent, then there is a Lie subalgebra $\fn \subseteq Z(\fg)$ with $\fn \simeq \Z/p\Z$.

Lie algebras of nilpotency class $\leq 2$ are those for which the Lie bracket is valued in the center, i.e., such that $[\fg, \fg] \subseteq Z(\fg)$.
If $\fg$ is a $\Z_p$-module and $\mathfrak z$ is a given submodule of~$\fg$, then equipping $\fg$ with a Lie bracket such that $\fg$ has nilpotency class $\leq 2$ and center $\mathfrak z$ amounts to giving a nondegenerate alternating $\Z_p$-bilinear map $(\fg/\mathfrak z) \oplus (\fg/\mathfrak z) \to \mathfrak z$ (the Jacobi identity is automatically satisfied).

\subsubsection{The Lazard correspondence (cf.\ \cite{lazard})}
\label{subsn:lazard-correspondence}

Let $\fg$ be a Lie $\Z_p$-algebra of nilpotency class $<p$.
We define a group law $\circ$ on $\fg$ via the \emph{truncated Baker--Campbell--Hausdorff formula}:
\[
  x \circ y
  \coloneqq
  x + y
  + \frac12 [x,y]
  + \frac1{12} [x,y,y] - \frac1{12} [x,y,x]
  + \ldots
\]
where the sum is including only the finitely many terms of the Baker--Campbell--Hausdorff formula (see e.g.~\cite[p.\,29]{serrelie}) which do not feature $p$-fold commutators, thus involving only denominators coprime to~$p$ and making the sum well-defined.
For instance, for Lie algebras of nilpotency class $\leq 2$ (with $p>2$), the formula simplifies to
$
  x \circ y
  =
  x + y
  + \frac12 [x,y]
$.
The operation transforming the Lie algebra $\fg$ into the group $(\fg, \circ)$ is the key construction in the \emph{Lazard correspondence}, which is an equivalence of categories:
\begin{align*}
  \{\textnormal{finite Lie $\Z_p$-algebras of nilpotency class $<p$}\} &
  \longleftrightarrow
  \{\textnormal{finite $p$-groups of nilpotency class $<p$}\}
  .
\end{align*}
In particular, every finite $p$-group $G$ of nilpotency class $<p$ is isomorphic to $(\fg, \circ)$ for some finite Lie $\Z_p$-algebra $\fg$, which then has the same nilpotency class.
This correspondence was introduced by Lazard in \cite{lazard} (see also \cite{efflaz} or \cite[Section 1.2]{abrashkin-ramification-filtration-3}), and it is somewhat analogous to the classical Lie correspondence between Lie algebras and Lie groups.

Via this correspondence, Lie subalgebras of $\fg$ correspond to subgroups of $(\fg, \circ)$.
Similarly, ideals of $\fg$ correspond to normal subgroups of $(\fg, \circ)$, and the quotients then correspond to each other: a short exact sequence $0 \to \fn \to \fg \to \fg/\fn \to 0$ of Lie algebras (i.e., a short exact sequence of $\Z_p$-modules in which every arrow is a Lie algebra homomorphism) induces a short exact sequence $1 \to (\fn, \circ) \to (\fg, \circ) \to (\fg/\fn, \circ) \to 1$ of groups.

As the Lie bracket of an element $x \in \fg$ with itself vanishes, the $n$-th power of $x$ as an element of the group $(\fg, \circ)$ is $n \cdot x$, for all $n \in \Z$.
In particular, the inverse of $x$ with respect to $\circ$ is simply its additive inverse $-x$.

If $[x,y]=0$, then $x \circ y = y \circ x = x + y$.
Conversely, if $x$ and $y$ commute in $(\fg,\circ)$, then the Campbell identity
\[
  x \circ y \circ (-x)
  =
  \sum_{n=0}^{p-1}
    \frac{(-1)^n}{n!}
    [y,\underbrace{x,\ldots,x}_n]
\]
implies that $[x,y]$ is an $n$-fold commutator for any $n \geq 2$ (by induction on $n$) and hence $[x,y]=0$ as $\fg$ is nilpotent.
In particular, the center $Z(\fg)$ of the Lie algebra $\fg$ coincides with the center of the group~$(\fg, \circ)$.
Note that the Lie algebra $\fg$ is abelian if and only if the group $(\fg, \circ)$ is abelian, in which case the laws $+$ and $\circ$ coincide.

If $\fg$ is a $\Z_p$-module, the set $\fg[p]$ of its $p$-torsion elements forms an $\F_p$-vector space.
When $\fg$ is a Lie $\Z_p$-algebra, the Lie $\F_p$-algebra $\fg[p]$ is an ideal of $\fg$: indeed, if $x \in \fg[p]$, then $p \cdot [x,y] = [p \cdot x, y] = 0$ for any $y \in \fg$.
Using the Lazard correspondence, this directly implies:

\begin{lemma}
  \label{lem:p-torsion-is-group}
  Let $G$ be a finite $p$-group of nilpotency class smaller than~$p$.
  Its $p$-torsion elements form a normal subgroup of $G$, which we denote by $G[p]$.
\end{lemma}

\begin{remark}
  The condition on the nilpotency class is crucial in \Cref{lem:p-torsion-is-group}:
  for example, the $2$-torsion elements of the $2$-group $D_4$ (of nilpotency class $2$) do not form a subgroup of $D_4$.
\end{remark}

\subsection{Nilpotent Artin--Schreier theory}
\label{subsn:nilpotent-artinschreier}

We fix a finite Lie $\Z_p$-algebra $\fg$ of nilpotency class smaller than~$p$, and we let $G \coloneqq (\fg, \circ)$.
(By \Cref{subsn:lazard-correspondence}, every finite $p$-group $G$ of nilpotency class $<p$ arises in this way.)
The tensor product $\fg \otimes_{\Z_p} W(F^\alg)$ inherits a Lie $\Z_p$-algebra structure from $\fg$ (and hence a group law $\circ$) by $W(F^\alg)$-linear extension.
In this subsection, we apply \Cref{thm:param-galrep} to the group $G_{F^\sep} \coloneq (\fg\otimes W(F^\alg), \circ)$ in order to prove \Cref{thm:parametrization}.

The actions of $\Gamma_F$ and of $\sigma$ on $\fg\otimes W(F^\alg)$ respect the group law $\circ$, so they are actions on the group~$G_{F^\sep}$, and they commute with each other.
By \Cref{lem:witt-props} (points~\ref{lem:witt-props-ii} and~\ref{lem:witt-props-iii}), the subgroup of $G_{F^\sep}$ fixed by $\Gamma_F$ is $G_F \coloneqq (\fg\otimes W(F^\perf),\circ)$ and the subgroup fixed by $\sigma$ is $G=(\fg,\circ)$.
Following \Cref{subsn:parametrization}, we then define a left action of the group $G_{F^\sep}$ on the set $\fg \otimes W(F^\alg)$ by $g.m \coloneqq \sigma(g) \circ m \circ (-g)$.
This action restricts to an action of $G_F = (\fg \otimes W(F^\perf), \circ)$ on $\fg \otimes W(F^\perf)$.
We write $\Orb{\fg}{W(F^\perf)}{W(F^\perf)}$ for the set of $(\fg\otimes W(F^\perf), \circ)$-orbits of $\fg\otimes W(F^\perf)$.
The multiplicative Artin--Schreier map $\bwp:\fg\otimes W(F^\alg)\rightarrow \fg\otimes W(F^\alg)$ is given by $g\mapsto \sigma(g)\circ(-g)$, so that $\bwp(g) = g.0$.

We first prove the following lemma, which is used in proofs by induction on the size of~$\fg$:

\begin{lemma}
  \label{lem:central-lemma}
  Let $g, m \in \fg\otimes W(F^\alg)$ and $h, n \in Z(\fg)\otimes W(F^\alg)$.
  Then,
  $
    (g + h).(m + n) = g.m + h.n
  $.
\end{lemma}

\begin{proof}
  \begin{align*}
    (g + h).(m + n)
    &=
    \sigma(g+h) \circ (m+n) \circ (-g-h)
    \\
    &=
    \sigma(g) \circ \sigma(h) \circ m \circ n \circ (-g) \circ (-h)
    &
    \textnormal{as $h,n \in Z(\fg)\otimes W(F^\alg)$}
    \\
    &=
    \sigma(g) \circ  m \circ (-g) \circ \sigma(h) \circ n \circ (-h)
    &
    \textnormal{as $h,n \in Z(\fg)\otimes W(F^\alg)$}
    \\
    &=
    (g.m) \circ (h.n)
    \\
    &=
    g.m + h.n
    &
    \textnormal{as $h.n \in Z(\fg)\otimes W(F^\alg)$}.
    &
    \qedhere
  \end{align*}
\end{proof}

In particular, applied to $m=n=0$, \Cref{lem:central-lemma} implies:

\begin{corollary}
  \label{cor:bwp-additive}
  Let $g \in \fg \otimes W(F^\alg)$ and $h \in Z(\fg) \otimes W(F^\alg)$.
  Then, $\bwp(g+h) = \bwp(g) + \bwp(h)$.
\end{corollary}

We now verify the remaining hypotheses of \Cref{thm:param-galrep}:

\begin{lemma}
  \label{lem:lsep-is-good}
  The group $(\fg\otimes W(F^\alg),\circ)$ satisfies the following properties:
  \begin{enumroman}
  \item\label{item:lsep-is-good-i}
    We have $(\fg \otimes W(F^\alg))^{\sigma} \subseteq (\fg \otimes W(F^\alg))^{\Gamma_F}$.
  \item\label{item:lsep-is-good-ii}
    The multiplicative Artin--Schreier map $\bwp:\fg\otimes W(F^\alg)\rightarrow \fg\otimes W(F^\alg)$ is surjective.
  \item\label{item:lsep-is-good-iii}
    The cohomology set $H^1(\Gamma_F, (\fg \otimes W(F^\alg), \circ))$ is a singleton.
  \end{enumroman}
\end{lemma}

\begin{proof}
  Point \ref{item:lsep-is-good-i} directly follows from \Cref{lem:witt-props} (points \ref{lem:witt-props-ii} and \ref{lem:witt-props-iii}), so we focus on points~\ref{item:lsep-is-good-ii} and~\ref{item:lsep-is-good-iii}.

  When $\fg$ is an abelian Lie $\Z_p$-algebra, the group law $\circ$ coincides with $+$, so we have already proved the claims in \Cref{lem:witt-props} (points~\ref{lem:witt-props-0} and \ref{lem:witt-props-iv}).

  We now prove the result for a general Lie $\Z_p$-algebra $\fg$ of nilpotency class $<p$, by induction on the size of $\fg$.
  The case $\fg=0$ is clear, so assume that~$\fg$ is nonzero and that the result holds for Lie $\Z_p$-algebras of nilpotency class $<p$ whose size is smaller than $|\fg|$.
  Pick a subalgebra $\fn \subseteq Z(\fg)$ isomorphic to~$\Z/p\Z$.
  We have already shown the claims for the abelian Lie algebra $\fn$, and the claims hold for the Lie algebra $\fg/\fn$ by the induction hypothesis.
  To combine these two known cases, we use the exact sequence of Lie $\Z_p$-algebras
  \[
    0 \rightarrow
    \fn \rightarrow
    \fg \rightarrow
    \fg/\fn \rightarrow
    0,
  \]
  which induces an exact sequence of Lie $\Z_p$-algebras (see~\Cref{rmk:flat-module})
  \[
    0 \rightarrow
    \fn\otimes W(F^\alg) \rightarrow
    \fg\otimes W(F^\alg) \rightarrow
    (\fg/\fn)\otimes W(F^\alg) \rightarrow
    0,
  \]
  and thus an exact sequence of groups (see~\Cref{subsn:lazard-correspondence})
  \[
    1 \rightarrow
    (\fn\otimes W(F^\alg),\circ)
    \rightarrow (\fg\otimes W(F^\alg),\circ)
    \rightarrow ((\fg/\fn)\otimes W(F^\alg),\circ)
    \rightarrow 1.
  \]
  We prove points~\ref{item:lsep-is-good-ii} and~\ref{item:lsep-is-good-iii}:
  \begin{enumroman}[start=2]
    \item
      Let $x \in \fg \otimes W(F^\alg)$, and let $\bar x \in (\fg/\fn) \otimes W(F^\alg)$ be its projection.
      By the induction hypothesis, there is some $\bar y \in (\fg/\fn) \otimes W(F^\alg)$ such that $\bwp(\bar y) = \bar x$.
      Choose an arbitrary lift $y \in \fg \otimes W(F^\alg)$ of $\bar y$.
      Then, $x - \bwp(y)$ belongs to $\fn \otimes W(F^\alg)$.
      By the case $\fg = \Z/p\Z$, there is a $z \in \fn \otimes W(F^\alg)$ such that $\bwp(z) = x - \bwp(y)$.
      As $z$ is central, \Cref{cor:bwp-additive} implies that $\bwp(y + z) = \bwp(y) + \bwp(z) = x$.
      Thus, $\bwp$ is surjective.
    
    \item
      We have the following exact sequence of pointed sets in non-abelian Galois cohomology:
      \[
          H^1(\Gamma_F, (\fn \otimes W(F^\alg), \circ))
        \longrightarrow
        H^1(\Gamma_F, (\fg \otimes W(F^\alg), \circ))
        \longrightarrow
          H^1(\Gamma_F, ((\fg/\fn) \otimes W(F^\alg), \circ))
      \]
      In this sequence, we already know that the cohomology sets associated to $\fn$ and to $\fg/\fn$ are trivial.
      Hence, so is the cohomology set $H^1(\Gamma_F, (\fg\otimes W(F^\alg),\circ))$ associated to $\fg$.
      \qedhere
  \end{enumroman}
\end{proof}

\Cref{lem:witt-props,lem:lsep-is-good} let us apply \Cref{thm:param-galrep} to parametrize $G$-extensions of $F$, where $G = (\fg,\circ)$, generalizing \Cref{thm:parametrization-abelian} to (some) non-abelian $p$-extensions:

\begin{theorem}
  \label{thm:parametrization}
  There is a bijection
  \[
    \begin{matrix}
      \orb : &
      H^1(\Gamma_F, G)&
      \overset\sim\longrightarrow &
      \Orb{\fg}{W(F^\perf)}{W(F^\perf)}
      \\ &
      [\tau \mapsto (-g) \circ \tau(g)] &
      \longmapsto &
      [\bwp(g)] &
      \textnormal{for any } g \in \bwp^{-1}(\fg \otimes W(F^\perf)).
    \end{matrix}
  \]
  Moreover, if $\orb([\gamma]) = [m]$, then
  $
    \Stab_G(\gamma) \simeq \Stab_{(\fg \otimes W(F^\perf), \circ)}(m)
  $.
\end{theorem}

\begin{remark}
  \label{rk:no-cohen-rings}
  In \cite{abrashkin-ramification-filtration-3}, Abrashkin instead constructed a closely related bijection
  \[
    H^1(\Gamma_F, G) \stackrel\sim\longrightarrow \Orb{\fg}{O(F)}{O(F)},
  \]
  where $O(F)\subseteq W(F)$ is the Cohen ring of $F$, i.e., the $p$-adically complete flat local $\Z_p$-algebra with maximal ideal~$p$ and residue field~$F$ (unique up to isomorphism).
  We refer to \cite[Subsection~1.2.4]{fontaineouyang} or \cite[Subsection~1.1]{bertmess} for more details about Cohen rings.
  
  The main advantage of Cohen rings over Witt vectors is the fact that $\Z/p\Z\otimes O(F) = O(F)/pO(F) = F$, whereas $\Z/p\Z\otimes W(F^\perf)=F^\perf$ is slightly larger.
  (However, this disadvantage disappears modulo the image of the Artin--Schreier map, see \Cref{lem:artsch-perfclos}.)
  On the other hand, the main advantage of Witt vectors over Cohen rings is functoriality: any field homomorphism $F_1\hookrightarrow F_2$ induces a canonical ring homomorphism $W(F_1^\perf)\hookrightarrow W(F_2^\perf)$, but not necessarily a canonical ring homomorphism $O(F_1)\rightarrow O(F_2)$.
  Functoriality is a convenient property for us to apply to the actions of $\Gamma_F$ and $\sigma$ on $F^\sep$ and to embeddings of global fields into their completions.
\end{remark}

\begin{proposition}[Naturality]
  \label{rmk:functoriality}
  The bijection of \Cref{thm:parametrization} is natural in both $\fg$ and~$F$:
  \begin{enumalpha}
  \item
    \label{rmk:functoriality-i}
    Let $f:\fg_1\rightarrow \fg_2$ be a morphism of finite Lie $\Z_p$-algebras of nilpotency class smaller than~$p$.
    We obtain a commutative diagram:
    \[\begin{tikzcd}
      H^1(\Gamma_F, (\fg_1, \circ)) \dar \rar[->]{\orb}[swap]{\sim} & \Orb{\fg_1}{W(F^\perf)}{W(F^\perf)} \dar \\
      H^1(\Gamma_F, (\fg_2, \circ)) \rar[->]{\orb}[swap]{\sim} & \Orb{\fg_2}{W(F^\perf)}{W(F^\perf)}
    \end{tikzcd}\]
    where the vertical maps are the natural maps induced by $f$.
  \item
    \label{rmk:functoriality-ii}
    Let $f : F_1\hookrightarrow F_2$ be a field homomorphism between fields of characteristic~$p$.
    We obtain a commutative diagram:
    \[\begin{tikzcd}
      H^1(\Gamma_{F_1}, (\fg, \circ)) \dar \rar[->]{\orb}[swap]{\sim} & \Orb{\fg}{W(F_1^\perf)}{W(F_1^\perf)} \dar \\
      H^1(\Gamma_{F_2}, (\fg, \circ)) \rar[->]{\orb}[swap]{\sim} & \Orb{\fg}{W(F_2^\perf)}{W(F_2^\perf)}
    \end{tikzcd}\]
    where the vertical maps are the natural maps induced by $f$.
  \end{enumalpha}
\end{proposition}

\begin{remark}
\label{rmk:subfield}
  Let $[m]\in\Orb{\fg}{W(F^\perf)}{W(F^\perf)}$ correspond to a $G$-extension $K|F$ via the bijections of \Cref{thm:parametrization} and \Cref{lem:etext-bij-cohomology}.
  \Cref{rmk:functoriality} implies the following:
  \begin{enumalpha}
    \item
      \label{rk:func-quotient}
      For any ideal $\fn$ of $\fg$, the orbit $[m \bmod \fn]\in \Orb{(\fg/\fn)}{W(F^\perf)}{W(F^\perf)}$ corresponds to the $(\fg/\fn,\,\circ)$-subextension of $K$ fixed by~$(\fn,\,\circ)$.
    \item
      \label{rk:func-completion}
      For any valuation $v$ of $F$, denoting by $F_v$ the completion of $F$ with respect to $v$, the orbit $[m] \in \Orb{\fg}{W(F_v^\perf)}{W(F_v^\perf)}$ corresponds to the $G$-extension $K \otimes_F F_v$ of $F_v$.
  \end{enumalpha}
\end{remark}

By \Cref{thm:parametrization} and \Cref{lem:etext-bij-cohomology}, elements of $\Orb{\fg}{W(F^\perf)}{W(F^\perf)}$ correspond bijectively to $G$-extensions of $F$.
Thus, we extend the definition of the last jump:

\begin{definition}
  \label{def:lastjump-for-orbits}
  Assume that $F$ is either a local or a global function field of characteristic~$p$.
  If $K|F$ is the $G$-extension associated to an orbit $[m] \in \Orb{\fg}{W(F^\perf)}{W(F^\perf)}$, we define both $\lastjump(m)$ and $\lastjump([m])$ to be $\lastjump(K|F)$ (cf.~\Cref{subsn:last-jump}).
\end{definition}

Finally, we describe the effect of the ``twisting operation'' of \Cref{def:twisting} in terms of the parametrization:

\begin{lemma}[Twisting]
  \label{lem:twisting-orb}
  Let $\fn$ be a Lie subalgebra of $Z(\fg)$.
  For any $\gamma\in\Hom(\Gamma_F, (\fg, \circ))$ and $\delta\in\Hom(\Gamma_F, (\fn, \circ))$, let $m \in \fg \otimes W(F^\perf)$ and $n \in \fn \otimes W(F^\perf)$ be such that:
  \[
    \orb([\gamma]) = [m] \in \Orb{\fg}{W(F^\perf)}{W(F^\perf)}
    \andd
    \orb([\delta]) = [n] \in \Orb{\fn}{W(F^\perf)}{W(F^\perf)}.
  \]
  Then,
  $
    \orb([\gamma\cdot\delta]) = [m+n] \in \Orb{\fg}{W(F^\perf)}{W(F^\perf)}
  $.
\end{lemma}

\begin{proof}
  Pick $g\in \fg \otimes W(F^\alg)$ and $h\in \fn \otimes W(F^\alg)$ such that $m = \bwp(g)$ and $n=\bwp(h)$, and such that $\gamma(\tau)=(-g)\circ\tau(g)$ and $\delta(\tau)=(-h)\circ\tau(h)$ for all $\tau \in \Gamma_F$.
  For all $\tau \in \Gamma_F$, we have $(\gamma\cdot\delta)(\tau)=\gamma(\tau)\circ\delta(\tau)=(-g)\circ\tau(g)\circ(-h)\circ\tau(h)=(-g-h)\circ\tau(g+h)$ as $h\in Z(\fg)\otimes W(F^\alg)$.
  Moreover, by \Cref{cor:bwp-additive}, we have $\bwp(g+h)=\bwp(g)+\bwp(h)=m+n$.
  Hence, $\orb([\gamma\cdot\delta])=[m+n]$ by definition.
\end{proof}

\section{Local parametrization}
\label{sn:local}

In this section, we fix a finite field $\kappa$ of characteristic $p > 2$, a local function field $\fF$ with residue field~$\kappa$, and a uniformizer $\pi$ of $\fF$, so that $\fF = \kappa\llpar\pi\rrpar$.
We denote by $\tilde\pi$ the Teichmüller representative $(\pi, 0, 0, \ldots) \in W(\fF) \subseteq W(\fF^\perf)$ of $\pi$.
We also fix a finite Lie $\Z_p$-algebra $\fg \neq 0$ of nilpotency class at most $2$, and we denote by $G \coloneqq (\fg, \circ)$ its associated $p$-group.

Our goal is to study $G$-extensions of the local field~$\fF$.
We refine the parametrization of these extensions by describing an ``approximate fundamental domain'' (\Cref{subsn:fundom,subsn:other-fundom}), and we use Abrashkin's results from \cite{abrashkin-ramification-filtration-3} to describe the last jump in terms of this parametrization (\Cref{subsn:ram-eq,subsn:local-lastjump}).
This will allow us to count extensions in \Cref{sn:local-counting}.

The main results of this section can be summed up as follows: we define $W(\kappa)$-modules $\cD^0, \cD  \subseteq W(\fF^\perf)$ and a surjection $\pr : \cD^0 \twoheadrightarrow \cD$ [cf.~\Cref{def:cD-zero,def:cD}] such that:
\begin{itemize}
  \item
    The action of $(\fg \otimes W(\fF^\perf), \circ)$ on $\fg \otimes W(\fF^\perf)$ restricts to an action of the finite group $(\fg \otimes W(\kappa), \circ)$ on $\fg \otimes \cD^0$.
    We denote by $\Orb{\fg}{\cD^0}{W(\kappa)}$ the set of orbits for the restricted action.
    [cf.~\Cref{prop:kappa-acts}]
  \item
    The inclusion $\cD^0 \subseteq W(\fF^\perf)$ induces a bijection $\Orb{\fg}{\cD^0}{W(\kappa)} \simto \Orb{\fg}{W(\fF^\perf)}{W(\fF^\perf)}$ (the latter is itself in bijection with $\Ext(G,\fF)$ by \Cref{thm:parametrization}).
    [cf.~\Cref{thm:local-fundamental-domain}]
  \item
    The last jump of an element $D \in \fg \otimes \cD^0\subseteq \fg\otimes W(\fF^\perf)$ (in the sense of \Cref{def:lastjump-for-orbits}) is characterized in terms of its ``coordinates'' $D_a \in \fg \otimes W(\kappa)$ by the equations of \Cref{def:property-Jv}.
    [cf.~\Cref{thm:lastjump-Jv}]
  \item
    There is an action of $(\fg \otimes W(\kappa), \circ)$ on $\fg \otimes \cD$ such that the surjection $\pr$ induces a surjection $\pr : \Orb{\fg}{\cD^0}{W(\kappa)} \to \Orb{\fg}{\cD}{W(\kappa)}$ between the respective sets of orbits.
    [cf.~\Cref{prop:action-on-cD}]
  \item
    Each fiber of the surjection $\pr : \fg \otimes \cD^0 \twoheadrightarrow \fg \otimes \cD$ has size $|\fg \otimes W(\kappa)|$ and is formed of elements having the same last jump.
    [cf.~\Cref{rk:property-Jv-for-cD}, \Cref{cor:irrelevance-D0}]
\end{itemize}
This is summarized by the following diagram:
\[\begin{matrix}
  \fg\otimes W(\fF^\perf) &
  \longhookleftarrow &
  \fg\otimes\cD^0 &
  \overset{|\fg\otimes W(\kappa)|:1}{\longtwoheadrightarrow{2}} &
  \fg\otimes\cD \\[0.5em]
  \Orb{\fg}{W(\fF^\perf)}{W(\fF^\perf)} &
  \overset\sim\longleftrightarrow &
  \Orb{\fg}{\cD^0}{W(\kappa)} &
  \longtwoheadrightarrow{2} &
  \Orb{\fg}{\cD}{W(\kappa)} \\[0.5em]
  \lastjump([m]) &
  = &
  \lastjump([D^0]) &
  = &
  \lastjump([D])
\end{matrix}\]

Hence, counting $G$-extensions $K|\fF$ with $\lastjump(K) < v$ essentially amounts to counting elements of $\fg \otimes \cD$ satisfying certain equations (given in \Cref{def:property-Jv}).
This fact, which is made more precise in \Cref{lem:numberofD-numberofexts}, will be used throughout \Cref{sn:local-counting} to count local extensions.

\subsection{Fundamental domain}
\label{subsn:fundom}

In \Cref{thm:parametrization}, we have constructed a bijection between the set $H^1(\Gamma_\fF, G)$ and the set $\Orb{\fg}{W(\fF^\perf)}{W(\fF^\perf)}$ of $(\fg\otimes W(\fF^\perf),\circ)$-orbits of $\fg\otimes W(\fF^\perf)$.
For counting orbits, it is often convenient to work with a fundamental domain consisting of exactly one representative from each orbit.
Here, we do a little less: we define a canonical subset $\fg\otimes\cD^0$ of~$\fg \otimes W(\fF^\perf)$ which is a fundamental domain up to the action of the finite subgroup $(\fg\otimes W(\kappa),\circ)$ of $(\fg \otimes W(\fF^\perf),\circ)$ (see \Cref{thm:local-fundamental-domain}).
This allows us to count orbits using the orbit-stabilizer theorem.
This ``fundamental domain'' is closely related to Abrashkin's notion of ``special elements'' (cf.~\cite[Definition~2.1]{abrashkin-differential}) and to \cite[Lemma~3.2]{imai-wild-ramification-groups}.

\begin{lemma}
  \label{lem:size-tensor}
  The set $\fg\otimes W(\kappa)$ is finite of size $|\kappa|^n$ if $|\fg|=p^n$.
\end{lemma}

\begin{proof}
  As in the proof of \Cref{lem:witt-props}, we may assume that $\fg\simeq\Z/p^n\Z$.
  Then, $\fg\otimes W(\kappa) = W(\kappa)/p^n W(\kappa) = W_n(\kappa)$ since the finite field $\kappa$ is perfect.
  The ring $W_n(\kappa)$ of Witt vectors of length $n$ over $\kappa$ has size $|\kappa|^n$.
\end{proof}

\begin{definition}
  \label{def:cD-zero}
  We define the following free $W(\kappa)$-submodule of $W(\fF) \subseteq W(\fF^\perf)$:
  \[
    \cD^0 \coloneqq
    \bigoplus_{a\in\Onotp}
      W(\kappa)
      \tilde\pi^{-a}.
  \]
\end{definition}

Note that $\cD^0/p\cD^0 = \bigoplus_{a\in\{0\}\cup\N\setminus p\N}\kappa\pi^{-a} \subseteq \fF^\perf$.
The set $\fg \otimes \cD^0$ is the sub-$W(\kappa)$-module of $\fg \otimes W(\fF^\perf)$ consisting of elements $D$ of the form $\sum_{a \in \Onotp} D_a \tilde\pi^{-a}$, where the coordinates~$D_a$ belong to the Lie $W(\kappa)$-algebra $\fg \otimes W(\kappa)$ and are almost all zero.

\begin{lemma}
  \label{lem:artin-schreier-in-fundom}
  ~
  \begin{enumroman}
    \item
      \label{lem:artin-schreier-in-fundom-ii}
      If $x \in \fF^\perf$ is such that $\wp(x) \in \cD^0/p\cD^0$, then $x \in \kappa$.
    \item
      \label{lem:artin-schreier-in-fundom-iii}
      The map $\cD^0/p\cD^0 \to \fF/\wp(\fF)$ induced by the inclusion $\cD^0/p\cD^0\subseteq \fF$ is a surjection.
  \end{enumroman}
\end{lemma}

\begin{proof}
  ~
  \begin{enumroman}
    \item
      Since $\wp(x) = x^p - x$ lies in $\fF$ and $\fF^\perf$ is a purely inseparable extension of $\fF$, we have $x\in \fF$.
      If the ($\pi$-adic) valuation of $x$ is negative, then the valuation of $\wp(x) = x^p - x$ is~$p$ times that of~$x$, hence is a nonzero multiple of~$p$, contradicting $\wp(x) \in \cD^0/p\cD^0$.
      Thus, $x \in \kappa\llbracket\pi\rrbracket$.
      Assume that $x \not\in \kappa$, and write $x = x_0 + x_1$, where $x_0 \in \kappa$ and $x_1$ has positive valuation.
      Then $\wp(x_1) = x_1^p - x_1$ has valuation equal to that of $x_1$, which is positive, contradicting the fact that $\wp(x_1) = \wp(x) - \wp(x_0)$ belongs to $\cD^0/p\cD^0$.
    \item
      This is \cite[Proposition~5.2~(b)]{potthast}.
      \qedhere
  \end{enumroman}
\end{proof}

\begin{lemma}
  \label{lem:local-fundamental-domain-prep}
  ~
  \begin{enumroman}
  \item\label{lem:local-fundamental-domain-prep-inj}
    Let $D \in \fg \otimes \cD^0$ and $g \in \fg \otimes W(\fF^\perf)$.
    Then, there is an equivalence:
    \[
      g.D \in \fg \otimes \cD^0
      \quad\iff\quad
      g \in \fg \otimes W(\kappa).
    \]
  \item\label{lem:local-fundamental-domain-prep-surj}
    The natural map $\fg\otimes\cD^0 \to \Orb{\fg}{W(\fF^\perf)}{W(\fF^\perf)}$ is surjective.
  \end{enumroman}
\end{lemma}

\begin{proof}
  We first prove \ref{lem:local-fundamental-domain-prep-inj}($\Leftarrow$).
  Write $D = \sum_{a \in \Onotp} D_a \tilde\pi^{-a}$ and assume that $g \in \fg \otimes W(\kappa)$.
  Recall that the Baker--Campbell--Hausdorff formula takes the form $x \circ y = x + y + \frac12 [x,y]$.
  We have:
  \begin{align*}
    g.D
    &=
    \sigma(g) \circ D \circ (-g)
    \\
    &=
    \sigma(g) + D + (-g) +
    \frac12 [\sigma(g),D] +
    \frac12 [D, -g] +
    \frac12 [\sigma(g), -g]
    \\
    &=
    \sigma(g)\circ(-g)
    +
    D
    -
    \frac12
    \big[
      D, \,
      \sigma(g)+g
    \big]
    \\
    &=
    \underbrace{
      \sigma(g) \circ (-g) + D_0 - \frac12[D_0,\sigma(g)+g]
    }_{\in \fg \otimes W(\kappa)} 
    +
    \sum_{a \in \notp}
    \underbrace{
      \left(
        D_a
        -
        \frac12
        \big[
          D_a, \,
          \sigma(g)+g
        \big]
      \right)
    }_{\in \fg \otimes W(\kappa)}
      \tilde\pi^{-a}
    \in \fg \otimes \cD^0.
  \end{align*}
  For $\fg=\Z/p\Z$, \ref{lem:local-fundamental-domain-prep-inj}($\Rightarrow$) and \ref{lem:local-fundamental-domain-prep-surj} follow from \Cref{lem:artin-schreier-in-fundom} since $\fg\otimes\cD^0 = \cD^0/p\cD^0$, $\fg\otimes W(\fF^\perf) = \fF^\perf$, $\fg\otimes W(\kappa) = \kappa$, and $g.D = \wp(g) + D$.
  To show \ref{lem:local-fundamental-domain-prep-inj}($\Rightarrow$) and \ref{lem:local-fundamental-domain-prep-surj} for arbitrary $\fg$, we proceed by induction on the size of $\fg$ as in the proof of \Cref{lem:lsep-is-good}.
  The case $\fg=0$ is clear, so we assume that~$\fg$ is nonzero.
  Pick a subalgebra $\fn \subseteq Z(\fg)$ isomorphic to~$\Z/p\Z$.
  We have already shown the claims for the abelian Lie algebra $\fn$, and the claims hold for the Lie algebra $\fg/\fn$ by the induction hypothesis.

  We now prove \ref{lem:local-fundamental-domain-prep-inj}($\Rightarrow$).
  Assume that $g.D \in \fg \otimes \cD^0$.
  Let $\bar g \in (\fg/\fn) \otimes W(\fF^\perf)$ be the projection of $g$.
  By hypothesis, the projections of both $D$ and~$g.D$ belong to $(\fg/\fn) \otimes \cD^0$, so $\bar g$ belongs to $(\fg/\fn) \otimes W(\kappa)$ by the induction hypothesis.
  Pick an arbitrary lift $\gamma \in \fg \otimes W(\kappa)$ of $\bar g$, and let $\delta \coloneqq g - \gamma$, which belongs to $\fn \otimes W(\fF^\perf) \subseteq Z(\fg) \otimes W(\fF^\perf)$.
  We have:
  \begin{align*}
    g.D
    & = (\gamma + \delta).(D+0)
    \\
    & =
    \gamma.D + \wp(\delta)
    &
    \textnormal{
      by \Cref{lem:central-lemma}.
    }
  \end{align*}
  By hypothesis, $g.D \in \fg \otimes \cD^0$.
  By the implication ($\Leftarrow$) proved above, $\gamma.D \in \fg \otimes \cD^0$.
  Therefore, $\wp(\delta) \in \fg \otimes \cD^0$.
  Since $\delta\in\fn\otimes W(\fF^\perf)$, it follows that $\wp(\delta)\in\fn\otimes\cD^0$.
  As $\fn\simeq\Z/p\Z$ satisfies \ref{lem:local-fundamental-domain-prep-inj}($\Rightarrow$), we conclude that $\delta\in\fn\otimes W(\kappa)$.
  Hence, $g = \gamma + \delta \in \fg \otimes W(\kappa)$.

  Finally, we prove \ref{lem:local-fundamental-domain-prep-surj}.
  Consider an element $m \in \fg \otimes W(\fF^\perf)$ and let $\bar m\in(\fg/\fn)\otimes W(\fF^\perf)$ be its projection.
  By the induction hypothesis, there is an element $\bar g \in (\fg/\fn) \otimes W(\fF^\perf)$ such that the element $\bar n \coloneqq \bar g.\bar m$ belongs to $(\fg/\fn) \otimes \cD^0$.
  Choose lifts $g \in \fg \otimes W(\fF^\perf)$ and $n \in \fg \otimes \cD^0$ of $\bar g$ and~$\bar n$, respectively.
  The element $g.m-n$ belongs to $\fn \otimes W(\fF^\perf)$.
  As $\fn\simeq\Z/p\Z$ satisfies~\ref{lem:local-fundamental-domain-prep-surj}, there is an element $h \in \fn \otimes W(\fF^\perf)$ such that $i \coloneqq h.(g.m-n) \in \fn \otimes \cD^0$.
  By \Cref{lem:central-lemma}, we have $(g+h).m = g.m + h.0$ and $i = h.(g.m-n) = g.m-n+h.0$, so $(g+h).m = n + i \in \fg\otimes\cD^0$.
  We have shown that the $(\fg \otimes W(\fF^\perf), \circ)$-orbit of $m$ intersects $\fg \otimes \cD^0$.
\end{proof}

\begin{remark}
\label{prop:kappa-acts}
  By \iref{lem:local-fundamental-domain-prep}{lem:local-fundamental-domain-prep-inj}($\Leftarrow$), the action of $(\fg\otimes W(\fF^\perf), \circ)$ on~$\fg\otimes W(\fF^\perf)$ restricts to an action of $(\fg\otimes W(\kappa), \circ)$ on~$\fg\otimes\cD^0$.
We denote by $\Orb{\fg}{\cD^0}{W(\kappa)}$ the set of orbits of $\fg \otimes \cD^0$ under the action of $(\fg \otimes W(\kappa), \circ)$.
\end{remark}

\begin{theorem}[Local approximate fundamental domain]
  \label{thm:local-fundamental-domain}
  There is a bijection
  \[
    \alpha^0 :
    \Orb{\fg}{W(\fF^\perf)}{W(\fF^\perf)}
    \overset\sim\longrightarrow
    \Orb{\fg}{\cD^0}{W(\kappa)}
  \]
  whose inverse is the map induced by the inclusion $\fg \otimes \cD^0 \to \fg \otimes W(\fF^\perf)$.
  Moreover, if $[D] = \alpha^0([m])$, then
  $
    \Stab_{(\fg\otimes W(\fF^\perf), \circ)}(m)
    \simeq
    \Stab_{(\fg\otimes W(\kappa), \circ)}(D)
  $.
\end{theorem}

\begin{proof}
  \iref{lem:local-fundamental-domain-prep}{lem:local-fundamental-domain-prep-inj} implies that the natural map $\Orb{\fg}{\cD^0}{W(\kappa)} \rightarrow \Orb{\fg}{W(\fF^\perf)}{W(\fF^\perf)}$ is injective and that
  $
    \Stab_{(\fg\otimes W(\kappa),\circ)}(D) =
    \Stab_{(\fg\otimes W(\fF^\perf),\circ)}(D) \simeq
    \Stab_{(\fg\otimes W(\fF^\perf),\circ)}(m)
  $
  if $[D] = \alpha^0([m])$.
  \iref{lem:local-fundamental-domain-prep}{lem:local-fundamental-domain-prep-surj} implies the surjectivity of this map.
\end{proof}

We again extend the definition of the last jump like in \Cref{def:lastjump-for-orbits}:

\begin{definition}
  \label{def:lastjump-in-fundom}
  If $[D] \in \Orb{\fg}{\cD^0}{W(\kappa)}$ is the image of $[m]\in\Orb{\fg}{W(\fF^\perf)}{W(\fF^\perf)}$ under $\alpha^0$, we define $\lastjump([D]) \coloneqq \lastjump([m])$.
  Note that $\lastjump([D]) = \lastjump(D)$.
\end{definition}

\subsection{Ramification equations}
\label{subsn:ram-eq}

In this subsection, we introduce for each $v>0$ a property~$J(v)$ on elements $D \in \fg \otimes \cD^0$.
This property will later be shown to coincide with the condition that $\lastjump(D) < v$ (\Cref{thm:lastjump-Jv}).

For any $b\in\notp$ and $v \in \R_{>0}$, we define
\begin{equation}
\label{eq:def-mu}
  \mu_v(b)
  \coloneqq \max(0,\lceil\log_p(v/b)\rceil)
  = \min\{k\geq 0 \mid b p^k \geq v\}
  = |\{k \geq 0 \mid b p^k < v\}|
\end{equation}
so that $\mu_v(b) = 0 \Leftrightarrow b \geq v$, that $b p^{\mu_v(b)} \geq v$ for all $v$, and that $b p^{\mu_v(b)} < pv$ when $b < pv$.
Note the following property of $\mu_v$:

\begin{lemma}
  \label{lem:sum-mu}
  For any $v > 0$, we have:
  \[
    \sum_{a \in \notp}
      \mu_v(a)
    =
    \lceil v \rceil -1.
  \]
\end{lemma}

\begin{proof}
  We have $\mu_v(a) = |\{k \geq 0 \mid a p^k < v\}|$.
  As every integer $0 < \gamma < v$ can be uniquely written as $a p^k$ for some $a \in \notp$ and some $k \in \{k \geq 0 \mid a p^k < v\}$, we have:
  \[
    \sum_{a \in \notp}
      \mu_v(a)
    =
    \cardsuchthat{\gamma\in\N}{0<\gamma<v}
    =
    \lceil v \rceil - 1.
    \qedhere
  \]
\end{proof}

If $n_1 \geq n_2$ are integers, we define
\begin{equation}
  \label{eqn:def-eta}
  \eta(n_1, n_2)
  \coloneqq
  \begin{cases}
    1 & \textnormal{if } n_1 > n_2 \\
    \frac12 & \textnormal{if } n_1 = n_2.
  \end{cases}
\end{equation}

\begin{definition}
  \label{def:property-Jv}
  Let $v\in\R_{>0}$,
  and let $D \in \fg \otimes \cD^0$.
  Write
  $
    D =
    \sum_{a\in\Onotp}
      D_a
      \widetilde\pi^{-a}
  $
  with $D_a \in \fg\otimes W(\kappa)$.
  We say that $D$ \emph{satisfies property $J(v)$} if the following equalities hold for all $b \in \notp$:
  \begin{equation}
    \label{eq:cF-2-integer-nu}
    p^{\mu_v(b)}\sigma^{\mu_v(b)}(D_b)
    =
    -
    b^{-1}
    \sum_{\substack{
      a_1,a_2\in\notp,\\
      n_1\geq n_2\geq0:\\
      bp^{\mu_v(b)} = a_1 p^{n_1} + a_2 p^{n_2},\\
      a_1 p^{n_1} < v,\ a_2 p^{n_1} < v
    }}
      \eta(n_1,n_2)
      a_1
      p^{n_1}
      \left[
        \sigma^{n_1}(D_{a_1}),
        \sigma^{n_2}(D_{a_2})
      \right],
  \end{equation}
  \begin{equation}
    \label{eq:cF-2-noninteger-nu}
    0 = \sum_{\substack{
      a_1,a_2\in\notp,\\
      n \geq 0:\\
      b = a_1 p^{n+i} + a_2,\\
      a_1 p^n < v,\ a_2 p^n < v
    }}
      a_1
      p^n
      \left[
        \sigma^{n+i}(D_{a_1}),
        D_{a_2}
      \right]
    \quad\quad\quad
    \textnormal{for any $i > 0$ such that $bp^{-i} \geq v$.}
  \end{equation}
\end{definition}

(The condition $a_2p^{n_1}<v$ in \Cref{eq:cF-2-integer-nu} is stronger than $a_2p^{n_2}<v$. This is not a typo!)

\begin{remark}
  The sum in \Cref{eq:cF-2-integer-nu} can be simplified by remarking that either $n_2 = n_1$ (in which case $\eta(n_1, n_2)=\frac12$ and $n_2 \leq \mu_v(b)$), or $n_1 > n_2$, in which case $n_2$ necessarily equals $\mu_v(b)$ (comparing valuations in $b p^{\mu_v(b)} = a_1p^{n_1} + a_2p^{n_2}$), $\eta(n_1, n_2)=1$, and $a_2 \equiv b \bmod {p^{n_1-\mu_v(b)}}$.
\end{remark}

\begin{example}
  Assume that $\fg$ is abelian.
  Then, the right-hand sides of \Cref{eq:cF-2-integer-nu,eq:cF-2-noninteger-nu} vanish, so $J(v)$ means that $p^{\mu_v(b)} D_b = 0$ for all $b\in\notp$.
  We retrieve a fact from class field theory: the $p$-part of the inertia group of the maximal abelian extension of $\fF$ with $\lastjump<v$ is isomorphic to $\prod_{b\in\N\setminus p\N}\Z/p^{\mu_v(b)}\Z$.
  (See for example \cite[Lemma~4.1]{gundab}.)
\end{example}

\begin{remark}
  When $p \fg=0$ (i.e., $\fg$ is a Lie $\F_p$-algebra), the equations of \Cref{def:property-Jv} take a simpler form, given in \Cref{cor:jv-in-exp-p}.
  These equations are easier to analyze, and considering this special case (detailed in \Cref{subsn:nilpas-expp,subsn:lastjump-in-expp}) is recommended for a first reading.
\end{remark}

\begin{proposition}
  \label{prop:consequences-Jv}
  Let $v>0$, let $D \in \fg \otimes \cD^0$, and assume that $D$ satisfies $J(v)$.
  Then:
  \begin{enumroman}
    \item
      \label{prop:consequences-Jv-ii}
      For all $b \in \notp$, $p^{\mu_v(b)} D_b$ belongs to $[\fg,\fg] \otimes W(\kappa)$, and in particular to $Z(\fg) \otimes W(\kappa)$.
    \item
      \label{prop:consequences-Jv-iii}
      For all $b \in \notp$, $p^{\mu_v(b)} D_b$ is a $p$-torsion element.
    \item
      \label{prop:consequences-Jv-iv}
      For all $b \in \notp$ such that $b \geq 2v$, we have $D_b=0$.
  \end{enumroman}
\end{proposition}

\begin{proof}
  ~
  \begin{enumroman}
    \item
      Applying $\sigma^{-\mu_v(b)}$ to \Cref{eq:cF-2-integer-nu} ($\sigma$ is an automorphism of $W(\kappa)$),  we can express $p^{\mu_v(b)} D_b$ as a sum of elements of $[\fg,\fg] \otimes W(\kappa)$.
    \item
      By \Cref{eq:cF-2-integer-nu}, checking that $p^{\mu_v(b)} D_b$ is a $p$-torsion element amounts to the vanishing of the following sum:
      \[
        \sum_{\substack{
          a_1,a_2\in\notp,\\
          n_1\geq n_2\geq0:\\
          b p^{\mu_v(b)} = a_1 p^{n_1} + a_2 p^{n_2},\\
          a_1 p^{n_1} < v,\ a_2 p^{n_1} < v
        }}
          \eta(n_1,n_2) a_1 p^{n_1 + 1} [\sigma^{n_1}(D_{a_1}), \sigma^{n_2}(D_{a_2})].
      \]
      By \ref{prop:consequences-Jv-ii}, the elements $p^{\mu_v(a_1)} D_{a_1}$ and $p^{\mu_v(a_2)} D_{a_2}$ are central, and in particular the commutator $p^{n_1 + 1} [\sigma^{n_1}(D_{a_1}), \sigma^{n_2}(D_{a_2})]$ vanishes as soon as $n_1 + 1 \geq \mu_v(a_1)$ or $n_1 + 1 \geq \mu_v(a_2)$.
      Therefore, non-zero terms can only occur when $\mu_v(a_1)>n_1+1$ and $\mu_v(a_2)>n_1+1$, i.e., when $a_1 p^{n_1} < \frac vp$ and $a_2 p^{n_1} < \frac vp$.
      But then, the two inequalities $a_1 p^{n_1} + a_2 p^{n_2} < 2 \frac vp < v$ and $bp^{\mu_v(b)} \geq v$ (by definition of $\mu_v(b)$) contradict the equality $a_1 p^{n_1} + a_2 p^{n_2} = bp^{\mu_v(b)}$, meaning that there are no non-zero terms in the sum.
    \item
      The inequality $b \geq 2v$ implies $\mu_v(b) = 0$.
      By \Cref{eq:cF-2-integer-nu}, we have:
      \[
        D_b
        =
        - b^{-1}
        \sum_{\substack{
          a_1,a_2\in\notp,\\
          n_1\geq n_2\geq0:\\
          b = a_1 p^{n_1} + a_2 p^{n_2},\\
          a_1 p^{n_1} < v,\ a_2 p^{n_1} < v
        }}
          \eta(n_1,n_2) a_1 p^{n_1} [\sigma^{n_1}(D_{a_1}), \sigma^{n_2}(D_{a_2})].
      \]
      An equality $b = a_1 p^{n_1} + a_2 p^{n_2}$ with $a_1 p^{n_1}, a_2p^{n_1} < v$ would imply $b < 2v$, which is not true.
      Hence, the sum is empty and $D_b = 0$.
      \qedhere
  \end{enumroman}
\end{proof}

\subsection{The second fundamental domain.}
\label{subsn:other-fundom}

We define a second ``fundamental domain'' $\fg \otimes \cD$, which is ``lossy'' (it essentially forgets about the unramified part and thus several $G$-extensions are mapped to the same element).
However, it turns out that it retains enough information to determine the last jump of extensions (cf.~\Cref{cor:irrelevance-D0}).
This second fundamental domain will be useful to establish our local--global principle and for counting.

\begin{definition}
  \label{def:cD}
  We define the following free $W(\kappa)$-submodule of $W(\fF) \subseteq W(\fF^\perf)$:
  \[
    \cD^0
    \coloneqq
    \bigoplus_{a\in\notp}
      W(\kappa)
      \tilde\pi^{-a}
  \]
  and we denote by
  $
  \pr : \cD^0 \twoheadrightarrow \cD
  $
  the natural projection, discarding the summand for $a = 0$.
\end{definition}

The surjection $\pr : \cD^0 \twoheadrightarrow \cD$ induces a surjection $\fg \otimes \cD^0 \twoheadrightarrow \fg \otimes \cD$, which we also denote by~$\pr$.
Concretely, it maps an element $D_0+\sum_{a\in\notp}D_a\widetilde\pi^{-a}$ of $\fg \otimes \cD^0$ to the element $\sum_{a\in\notp}D_a\widetilde\pi^{-a}$.
In particular, each fiber of $\pr$ has (finite) size $|\fg \otimes W(\kappa)|$.

Note that the variable $D_0$ does not appear in \Cref{eq:cF-2-integer-nu,eq:cF-2-noninteger-nu}.
Therefore, whether an element $D \in \fg \otimes \cD^0$ satisfies~$J(v)$ only depends on its projection $\pr(D) \in \fg \otimes \cD$, and it makes sense to extend the definition of property $J(v)$ to $\fg\otimes\cD$ as follows:

\begin{definition}
  \label{rk:property-Jv-for-cD}
  Let $v>0$.
  We say that an element $D = \sum_{a\in\N\setminus p\N}D_a\tilde\pi^{-a}$ of $\fg\otimes\cD$ \emph{satisfies property $J(v)$} if \Cref{eq:cF-2-integer-nu,eq:cF-2-noninteger-nu} hold for all $b\in\N\setminus p\N$.
\end{definition}

\begin{proposition}
  \label{prop:action-on-cD}
  If $g \in \fg \otimes W(\kappa)$ and $D \in \fg \otimes \cD^0$, then the element $\pr(g.D) \in \cD$ depends only on $g$ and $\pr(D)$.
  In other words, there is a (unique) action of $(\fg \otimes W(\kappa), \circ)$ on $\fg \otimes \cD$ such that $\pr(g.D) = g.\pr(D)$ for all $D \in \fg \otimes \cD^0$.
  Moreover:
  \begin{enumroman}
    \item
      This action is $W(\kappa)$-linear, i.e., it is an action on the $W(\kappa)$-module $\fg \otimes \cD$;
    \item
      \label{prop:action-on-cD-ii}
      $\fg \otimes W(\kappa)$ acts trivially on $Z(\fg) \otimes \cD$, so that orbits of central elements are of size $1$;
    \item
      \label{prop:action-on-cD-iii}
      Let $v > 0$.
      For any $g \in \fg \otimes W(\kappa)$ and any $D \in \fg \otimes \cD$ satisfying $J(v)$, the element $g.D$ satisfies~$J(v)$.
  \end{enumroman}
\end{proposition}

\begin{proof}
  The computation in the proof of \iref{lem:local-fundamental-domain-prep}{lem:local-fundamental-domain-prep-inj}~($\Leftarrow$) shows that, for every $D \in \fg \otimes \cD^0$:
  \[
    \pr(g.D)
    =
    \sum_{a \in \notp}
      \left(
        D_a
        -
        \frac12
        \big[
          D_a, \,
          \sigma(g)+g
        \big]
      \right)
      \tilde\pi^{-a}.
  \]
  Since $D_0$ does not appear in this formula, $\pr(g.D)$ only depends on $g$ and $D' \coloneqq \pr(D)$.
  This dependency is given by the following action of $(\fg \otimes W(\kappa), \circ)$ on elements $D' \in \fg \otimes \cD$:
  \begin{equation}
    \label{eqn:action-on-cD}
    g.D'
    \coloneqq D' -
    \frac12
    \big[
      D', \,
      \sigma(g)+g
    \big].
  \end{equation}
  That action is visibly $W(\kappa)$-linear in $D'$, and the action on a central element $D' \in Z(\fg) \otimes \cD$ is indeed trivial.
  With this definition, we have $\pr(g.D)=g.\pr(D)$.

  It remains to check that $g.D$ satisfies $J(v)$ if $g \in \fg\otimes W(\kappa)$ and $D \in \fg \otimes \cD$ satisfies $J(v)$.
  As illustrated by \Cref{eqn:action-on-cD}, $g.D$ and $D$ only differ by a commutator (hence an element of~$Z(\fg) \otimes \cD$).
  Thus, the invariance of \Cref{eq:cF-2-noninteger-nu} is immediate.
  In \Cref{eq:cF-2-integer-nu}, the right-hand side is also left unchanged, so it suffices to prove that
  \[
    p^{\mu_v(b)}D_b
    =
    p^{\mu_v(b)}
    \!\left(
      D_b - \frac12 [D_b, \sigma(g)+g]
    \right).
  \]
  By \iref{prop:consequences-Jv}{prop:consequences-Jv-iii}, the element $p^{\mu_v(b)} D_b$ is central.
  Therefore, $p^{\mu_v(b)}[D_b, \sigma(g)+g] = 0$, which proves the claim.
\end{proof}

\begin{definition}
  \label{def:orb-cD}
  We denote by $\Orb{\fg}{\cD}{W(\kappa)}$ the set of orbits of $\fg \otimes \cD$ under the $(\fg \otimes W(\kappa), \circ)$-action of \Cref{prop:action-on-cD}.
  For any $v > 0$, we say that an orbit $[D] \in \Orb{\fg}{\cD}{W(\kappa)}$ \emph{satisfies property $J(v)$} if $D$ satisfies property $J(v)$ (by \iref{prop:action-on-cD}{prop:action-on-cD-iii}, this is independent of the choice of $D$).
\end{definition}

\begin{definition}
  \label{def:alpha}
  We define the map $\alpha : \Orb{\fg}{W(\fF^\perf)}{W(\fF^\perf)} \to \Orb{\fg}{\cD}{W(\kappa)}$ as the following composition:
  \[\begin{tikzcd}
    \Orb{\fg}{W(\fF^\perf)}{W(\fF^\perf)} \rar{\alpha^0}[swap]{\sim} \ar[swap,->>]{rd}{\alpha} & \Orb{\fg}{\cD^0}{W(\kappa)} \dar[->>]{\pr} \\
    & \Orb{\fg}{\cD}{W(\kappa)}
  \end{tikzcd}\]
\end{definition}

\begin{lemma}[Twisting]
  \label{lem:twisting-alpha}
  Let $\fn$ be a Lie subalgebra of the center of $\fg$.
  For any $m\in \fg\otimes W(\fF^\perf)$ and $n\in \fn\otimes W(\fF^\perf)$, if
  \[
    \alpha([m]) = [D] \in \Orb{\fg}{\cD}{W(\kappa)}
    \andd
    \alpha([n]) = [E] \in \Orb{\fn}{\cD}{W(\kappa)},
  \]
  then
  \[
    \alpha([m+n]) = [D+E] \in \Orb{\fg}{\cD}{W(\kappa)}.
  \]
\end{lemma}

\begin{proof}
  By definition of $\alpha$, there are elements $g\in \fg\otimes W(\fF^\perf)$ and $h\in \fn\otimes W(\fF^\perf)$ such that $\pr(g.m)=D$ and $\pr(h.n)=E$.
  By \Cref{lem:central-lemma}, we have $\pr((g+h).(m+n)) = \pr(g.m+h.n) = \pr(g.m) + \pr(h.n) = D+E$, which belongs to $\fg\otimes\cD$.
  So $\alpha([m+n])=[D+E]$.
\end{proof}

\subsection{Determination of the last jump}
\label{subsn:local-lastjump}

The goal of this subsection is to deduce the following theorem from the results of \cite{abrashkin-ramification-filtration-3}:

\begin{theorem}
  \label{thm:lastjump-Jv}
  Let $m \in \fg \otimes W(\fF^\perf)$ and $v \in \R_{>0}$.
  Then, the following are equivalent:
  \begin{enumalpha}
    \item
      \label{item:lastjump-Jv-a}
      $\lastjump(m)<v$.
    \item
      \label{item:lastjump-Jv-b}
      $\alpha^0([m])$ satisfies property $J(v)$.
    \item
      \label{item:lastjump-Jv-c}
      $\alpha([m])$ satisfies property $J(v)$.
  \end{enumalpha}
\end{theorem}

We briefly recall Abrashkin's results in his own notation.
For a non-increasing list of integers
\[
  \nbar
  = (n_1,\dots,n_s)
  = (\underbrace{m_1,\dots,m_1}_{d_1},\dots,\underbrace{m_k,\dots,m_k}_{d_k}) \in \Z^s
\]
with $m_1>\cdots>m_k$, define the rational number (generalizing \Cref{eqn:def-eta}):
\[
  \eta(\nbar)
  \coloneqq \frac{1}{|\Stab_{\mathfrak S_s}\!(\nbar)|}
  = \frac{1}{d_1! \cdots d_k!}.
\]
For any integer $N\geq0$, any rational number $\gamma>0$, and any element $D = \sum_{a\in\Onotp} D_a\widetilde\pi^{-a}$ of $\fg\otimes\cD^0$ (with $D_a\in \fg\otimes W(\kappa)$), define the following element of $\fg \otimes W(\kappa)$:
\begin{equation}
  \label{eq:def-cF}
  \cF_{\gamma,-N}(D)
  \coloneqq
  \sum_{\substack{
    1 \leq s < p\\
    \abar, \nbar
  }}
    \eta(\nbar) a_1 p^{n_1} [\sigma^{n_1}(D_{a_1}),\dots,\sigma^{n_s}(D_{a_s})],
\end{equation}
where the sum is over all lengths $s \in \{1, \ldots, p-1\}$, all $s$-tuples $\abar = (a_1,\dots,a_s)\in(\Onotp)^s$, and all $s$-tuples $\nbar = (n_1,\dots,n_s)\in\Z^s$ satisfying $n_1\geq0$, $n_1\geq \cdots\geq n_s\geq-N$ and
\begin{equation}
\label{eq:abrashkin-sum}
a_1 p^{n_1} + \cdots + a_s p^{n_s} = \gamma.
\end{equation}

The main result from \cite{abrashkin-ramification-filtration-3} implies the following formula for the last jump (valid for all Lie $\Z_p$-algebras $\fg$ of nilpotency class $<p$):

\begin{theorem}[Abrashkin's theorem]
  \label{thm:abrashkins-theorem}
  Let $D\in \fg\otimes\cD^0$, and assume that the corresponding $G$-extension of $\fF$ is a field.
  Then:
  \[
    \lastjump(D)
    =
    \sup\!\Big(
      \{0\}\cup
      \big\{
        \gamma > 0
        \ \big\vert\ 
        \exists N_0 \geq 0, \
        \forall N \geq N_0, \
        \cF_{\gamma,-N}(D) \neq 0
      \big\}
    \Big).
  \]
\end{theorem}

\begin{proof}
  Consider a surjective continuous group homomorphism $f : \Gamma_\fF \to G$ in the class $\orb^{-1}([D])$.
  For every $v > 0$, let $\fg^{v}$ be the Lie subalgebra of $\fg$ corresponding to the image in $G$ of the ramification subgroup~$\Gamma_{\fF}^{v}$ under~$f$.
  By definition, we have $\lastjump(D) = \inf\{v > 0 \mid \fg^{v} = 0\}  = \sup\left( \{0\} \cup \{v > 0 \mid \fg^{v} \neq 0\} \right)$.
  For each $v > 0$, by \cite[Theorem~3.1]{abrashkin-differential} (which is the ``covariant'' version of \cite[Theorem~B]{abrashkin-ramification-filtration-3}), there is an integer $N$ such that~$\fg^{v}$ is the smallest ideal of~$\fg$ whose extension of scalars $\fg \otimes W(\kappa)$ contains~$\cF_{\gamma,-N}(D)$ for all $\gamma \geq v$.
  Therefore, the condition $\fg^{v}=0$ is equivalent to the vanishing of $\cF_{\gamma,-N}(D)$ for all $\gamma \geq v$ and $N$ large enough.
  The result follows directly.
\end{proof}

We now specialize the definition of $\cF_{\gamma,-N}(D)$ (\Cref{eq:def-cF}) to our situation.
Note that \Cref{eq:abrashkin-sum} can only hold if $\gamma \in \N[\frac1p]$.
Hence, for $\gamma\not\in\N[\frac1p]$, we always have $\cF_{\gamma,-N}(D) = 0$.
We now assume that $\gamma = bp^m$ for some $b\in\notp$ and $m \in \Z$.
In our case, since~$\fg$ has nilpotency class $\leq 2$, we need only to consider lengths $s \leq 2$ as all commutators involving more elements vanish.
\begin{itemize}
  \item
    For $s=1$, the equality $a_1 p^{n_1} = \gamma$ (with $a_1 \in \notp$, $n_1 \geq 0$) can only happen if $\gamma$ is an integer ($m \geq 0$).
    In that case, $a_1 = b$ and $n_1 = m$, and the corresponding term is $bp^m\sigma^m(D_b)$.
  \item
    For $s=2$, the equality $a_1 p^{n_1} + a_2 p^{n_2} = \gamma$ (with $n_1 \geq 0$) implies:
    \begin{itemize}
      \item
        If $\gamma$ is an integer ($m \geq 0$), then either $a_2 = 0$ (and then $a_1 = b$, $n_1 = m$), or $a_2 \neq 0$ and $n_2 \geq 0$.
      \item
        If $\gamma$ is not an integer ($m < 0$), then $n_2 < 0$ and $a_2 \neq 0$.
        In particular, $n_2 < n_1$ and thus $\eta(n_1, n_2) = 1$.
        Comparing valuations shows that $n_2 = m$ (there are no terms if $N \leq -m$).
    \end{itemize}
\end{itemize}
Thus, we have the following expression for $\cF_{bp^m,-N}(D)$ when $b \in \notp$ and $m \geq 0$:
\begin{equation*}
  \begin{aligned}
    \cF_{b p^m,-N}(D)
    ={}&
    b p^m\sigma^m(D_b)
    & \textnormal{(for $s=1$)}
    \\
    & + \sum_{n=-N}^m
      \eta(m,n) b p^m [\sigma^m(D_b),\sigma^n(D_0)]
    & \textnormal{(for $s=2$, $a_2 = 0$)}
    \\
    & + \sum_{\substack{
          a_1,a_2\in\notp,\\
          n_1\geq n_2\geq0:\\
          bp^m = a_1p^{n_1}+a_2p^{n_2}
      }}
      \eta(n_1,n_2)a_1p^{n_1}[\sigma^{n_1}(D_{a_1}),\sigma^{n_2}(D_{a_2})]
    & \textnormal{(for $s=2$, $a_2 \neq 0$)}
  \end{aligned}
\end{equation*}
and the following expression for $\cF_{bp^m,-N}(D)$ when $b \in \notp$ and $m<0$, assuming $N \geq -m$:
\begin{equation*}
  \begin{aligned}
    \cF_{bp^m,-N}(D)
    ={}&
    \sum_{\substack{
      a_1,a_2\in\notp,\\
      n_1\geq0:\\
      bp^m = a_1p^{n_1} + a_2p^m
    }}
      a_1
      p^{n_1}
      [
        \sigma^{n_1}(D_{a_1}),
        \sigma^m(D_{a_2})
      ].
  \end{aligned}
\end{equation*}

\begin{proposition}
  \label{prop:jv-F-vanishes}
  Let $D \in \fg \otimes \cD^0$.
  Then, $D$ satisfies property $J(v)$ if and only if, for all $b \in \notp$ and $m \in \Z$ such that $bp^m \geq v$, and for every sufficiently large $N$, we have $\cF_{bp^m, -N}(D) = 0$.
\end{proposition}

\begin{proof}
  For every $b \in \notp$, both property $J(v)$ and the equality $\cF_{bp^{\mu_v(b)},-N}(D)=0$ independently imply that $p^{\mu_v(b)} D_b \in [\fg,\fg]\otimes W(\kappa)$ (see \iref{prop:consequences-Jv}{prop:consequences-Jv-ii} and the expression for $\cF_{bp^m,-N}(D)$ above).
  We may hence prove the desired equivalence under the assumption that $p^{\mu_v(b)} D_b \in [\fg,\fg] \otimes W(\kappa)$ for all $b\in\notp$.
  In particular, $p^m D_b\in Z(\fg)\otimes W(\kappa)$ for all $m\geq0$ such that $bp^m\geq v$.

  Under this assumption, every term in $\cF_{bp^m,-N}(D)$ with $a_1 p^{n_1} \geq v$ or $a_2 p^{n_1} \geq v$ can be omitted, as the corresponding commutator vanishes once multiplied by $p^{n_1}$.
  Thus, the vanishing of $\cF_{bp^m,-N}(D)$ for every $m < 0$ such that $b p^m \geq v$ is equivalent to \Cref{eq:cF-2-noninteger-nu} (with $i\coloneq -m$).
  Similarly, for all $m \geq \mu_v(b)$, the sum over $n$ in $\cF_{bp^m,-N}(D)$ (the one involving $D_0$) can be omitted, and the vanishing of $\cF_{bp^{\mu_v(b)},-N}(D)$ is equivalent to \Cref{eq:cF-2-integer-nu}.

  All that is left to do is to check that property $J(v)$ implies the vanishing of $\cF_{bp^m,-N}(D)$ when $m > \mu_v(b)$.
  In that case, the inequality $b p^m \geq p v > 2v$ makes it impossible to simultaneously have $a_1 p^{n_1} + a_2 p^{n_2} = b p^m$ and $a_1 p^{n_1}, a_2 p^{n_1} < v$.
  Thus, all terms in $\cF_{bp^m,-N}(D)$ vanish except for $bp^m\sigma^m(D_b)$, but that term also vanishes because $m > \mu_v(b)$ and $p^{\mu_v(b)} D_b$ is $p$-torsion by \iref{prop:consequences-Jv}{prop:consequences-Jv-iii}.
\end{proof}

Finally, we prove \Cref{thm:lastjump-Jv}:

\begin{proof}[Proof of \Cref{thm:lastjump-Jv}.]
  The equivalence \ref{item:lastjump-Jv-b} $\Leftrightarrow$ \ref{item:lastjump-Jv-c} is clear (see~\Cref{rk:property-Jv-for-cD,def:orb-cD})

  Pick a $D \in \fg \otimes \cD^0$ such that $\alpha^0([m]) = [D]$.
  By definition, $m$ and $D$ belong to the same $(\fg \otimes W(\fF^\perf), \circ)$-orbit, and thus $\lastjump(m) = \lastjump(D)$ (they define the same extension!).
  Pick a $g \in \fg \otimes W(\fF^\alg)$ such that $\bwp(g) = D$, and consider the continuous group homomorphism $f : \Gamma_\fF \to G$ defined by $\tau \mapsto (-g) \circ \tau(g)$.

  If $f$ is surjective, the corresponding $G$-extension of $\fF$ is a field, and the result follows from \Cref{thm:abrashkins-theorem} and \Cref{prop:jv-F-vanishes} (recall that $\cF_{\gamma,-N}(D)$ always vanishes if $\gamma$ is not of the form~$bp^m$ for $b \in \notp$ and $m \in \Z$).

  Assume now that $f$ is not surjective.
  The image of $f$ is then a certain subgroup $H \subset G$, corresponding to a Lie subalgebra $\fh \subset \fg$.
  Then, the $G$-extension associated to $f$ is not a field, but the corresponding $H$-extension is a field.
  In that case, by \Cref{thm:parametrization} applied to $\fh$, there is a $D' \in \fh \otimes \cD^0$ and a $g' \in \fh \otimes W(\fF^\alg)$ such that $D' = \bwp(g')$, and $f(\tau) = (-g') \circ \tau(g')$.
  By the surjective case above (applied to the Lie algebra $\fh$), we have the equivalence:
  \[
    \lastjump(D') < v
    \quad\quad\Longleftrightarrow\quad\quad
    D' \textnormal{ satisfies property } J(v).
  \]
  Since the elements $D$ and $D'$ induce the same element $[f]$ of $H^1(\Gamma_\fF, G)$, we have $\lastjump(D) = \lastjump(D')$, and $D$ and $D'$ are in the same $(\fg \otimes W(\fF^\perf), \circ)$-orbit by \Cref{thm:parametrization}.
  By \Cref{prop:kappa-acts}, $D$ and~$D'$ are then in the same $(\fg \otimes W(\kappa), \circ)$-orbit, and by \iref{prop:action-on-cD}{prop:action-on-cD-iii} we have:
  \[
    D \textnormal{ satisfies property } J(v)
    \quad\quad\Longleftrightarrow\quad\quad
    D' \textnormal{ satisfies property } J(v).
  \]
  Assembling everything together, we obtain \ref{item:lastjump-Jv-a} $\Leftrightarrow$ \ref{item:lastjump-Jv-b}.
\end{proof}

\begin{corollary}[Irrelevance of $D_0$]
  \label{cor:irrelevance-D0}
  Let $m \in \fg\otimes\cD$.
  By \Cref{thm:lastjump-Jv}, all elements $m' \in \fg \otimes \cD^0$ such that $\pr(m') = m$ have the same last jump $\lastjump(m')$, namely
  \[
    \inf\!\Big(
      \{0\}
      \cup
      \!\big\{
        v > 0
        \ \big\vert\ 
        m \textnormal{ satisfies property } J(v)
      \big\}
    \Big).
  \]
\end{corollary}

\begin{definition}
  \label{def:lastjump-cD}
  If $m \in \fg\otimes\cD$, we define $\lastjump(m)$ to be the common value of $\lastjump(m')$ for all elements $m' \in \pr^{-1}(m)$ (cf.~\Cref{cor:irrelevance-D0}).
  Similarly, we let $\lastjump([m]) \coloneqq \lastjump(m)$ where $[m] \in \Orb{\fg}{\cD}{W(\kappa)}$ is the $(\fg \otimes W(\kappa),\circ)$-orbit of $m$ (this is well-defined by \iref{prop:action-on-cD}{prop:action-on-cD-iii}).
\end{definition}

\section{Counting local extensions}
\label{sn:local-counting}

The setting and notations are identical to those of \Cref{sn:local}.
This section deals with obtaining estimates of the number
\[
  \sum_{\substack{
    K \in \Ext(G, \fF):\\
    \lastjump(K) < v
  }}
    \frac1{|\Aut(K)|}
\]
for various values of $v$, thus describing the distribution of last jumps of $G$-extensions of $\fF$.
The starting point is the following observation:

\begin{lemma}
  \label{lem:numberofD-numberofexts}
  Let $v > 0$.
  Using the notation introduced in \Cref{def:lastjump-cD}, we have the equality
  \[
    \sum_{\substack{
      K \in \Ext(G, \fF):\\
      \lastjump(K) < v
    }}
      \frac1{|\Aut(K)|}
    =
    \cardsuchthat{
      D \in \fg \otimes \cD
    }{
      \lastjump(D) < v
    },
  \]
  as well as the analogous equality if the condition $<v$ is replaced by $=v$.
\end{lemma}

\begin{proof}
  We focus on the first claim, as the second claim directly follows from it.
  Using the three (stabilizer-preserving) bijections of \Cref{lem:etext-bij-cohomology}, \Cref{thm:parametrization} and \Cref{thm:local-fundamental-domain} and the notation of \Cref{def:lastjump-in-fundom}, the left-hand side rewrites as follows:
  \begin{align*}
    \sum_{\substack{
      K \in \Ext(G, \fF):\\
      \lastjump(K) < v
    }}
      \frac1{|\Aut(K)|}
    &
    =
    \sum_{\substack{
      [D] \in \Orb{\fg}{\cD^0}{W(\kappa)}:\\
      \lastjump([D]) < v
    }}
      \frac1{|\Stab_{(\fg \otimes W(\kappa), \circ)}(m)|}
    \\
    &=
    \frac{
      \cardsuchthat{
        D\in \fg\otimes\cD^0
      }{
        \lastjump(D) < v
      }
    }{
      |\fg\otimes W(\kappa)|
    }
    &
    \textnormal{by the orbit-stabilizer theorem.}
  \end{align*}
  Recalling \Cref{def:lastjump-cD} and the fact that the projection map $\pr : \fg \otimes \cD^0 \twoheadrightarrow \fg \otimes \cD$ is $|\fg \otimes W(\kappa)|$-to-one, we see that this equals
  $
    \cardsuchthat{
      D \in \fg \otimes \cD
    }{
      \lastjump(D) < v
    }
  $
  as claimed.
\end{proof}

Let
\[
  r \coloneqq \dim_{\F_p}(\fg[p])
  \andd
  M \coloneqq
  \begin{cases}
    1& \textnormal{if } \fg[p] \textnormal{ is abelian},\\
    1+p^{-1}& \textnormal{otherwise},
  \end{cases}
\]
as in \Cref{thm:intro-counting} for $G = (\fg, \circ)$.
\Cref{cor:slightly-ramified} below gives a natural interpretation of the number $M$ as the largest possible last jump smaller than $2$ for $G$-extensions of~$\fF$.

The main results of this section are the following estimates for the number of elements of~$\fg \otimes \cD$ with bounded last jump (this is motivated by \Cref{lem:numberofD-numberofexts}):

\begin{theorem}
  \label{thm:local-counting}
  For some (explicit) constant $E \geq 1$ (which does not depend on $\kappa$), the following estimates hold:
  \begin{enumalpha}
  \item
    \label{item:count-unramified}
    $
      |\{D\in \fg\otimes\cD \mid \lastjump(D)=0\}|
      = |\{D\in \fg\otimes\cD \mid \lastjump(D)<1\}|
      = 1
    $.
  \item
    \label{item:count-least-ramified}
    $
      |\{D\in \fg\otimes\cD \mid \lastjump(D)\leq M\}|
      = |\{D\in \fg\otimes\cD \mid \lastjump(D)<2\}|
      = |\kappa|^r
    $.
  \item
    \label{item:count-more-ramified-integer}
    For any $l\in\Z_{\geq 0}$, we have
    $
      |\{D \in \fg\otimes\cD \mid \lastjump(D) < l+1\}|
      \leq
      E^l |\kappa|^{r l}
    $.
  \item
    \label{item:count-more-ramified-noninteger}
    Assume that $\fg[p]$ is non-abelian.
    Then:
    \begin{enumroman}
      \item
        \label{item:count-more-ramified-large}
        for all integers $l \geq p$, we have
        $
          |\{D \in \fg\otimes\cD \mid \lastjump(D) < l+1\}|
          \leq
          E^l
          |\kappa|^{r l - 1}
        $.
      \item
        \label{item:count-more-ramified-small}
        for all $l \in \{1, \ldots, p-1\}$, we have:
        \[
          |\{D \in \fg\otimes\cD \mid \lastjump(D) < l(1+p^{-1})\}|
          =
          \cO\!\left(
            E^l |\kappa|^{r l-1}
          \right)
        \]
        where the implied constant in the $\cO$-estimate depends only on $\fg[p]$.
    \end{enumroman}
  \end{enumalpha}
\end{theorem}

Points \ref{item:count-unramified} and \ref{item:count-least-ramified} follow from \Cref{lem:unramified} and \Cref{cor:slightly-ramified} respectively.
The proof of points~\ref{item:count-more-ramified-integer} and \ref{item:count-more-ramified-noninteger} will be the goal of \Cref{subsn:counting-locext}.

\Cref{thm:local-counting} is later used for the global results of \Cref{sn:global}, where it is applied at each place~$P$ of the global function field $\F_q(T)$ by letting~$\fF$ be its completion at~$P$ and~$\pi$ be a uniformizer.

\subsection{Smallness criteria for the last jump}
\label{subsn:smallness-criteria}

We use \Cref{thm:lastjump-Jv} to characterize elements $D = \sum_{a\in\notp}D_a\tilde\pi^{-a} \in \fg \otimes \cD$ whose last jump is ``small'':

\begin{corollary}[Small $v$]
  \label{cor:small-v}
  Let $v\leq p$.
  We have $\lastjump(D) < v$ if and only if:
  \begin{align}
    p\sigma(D_b)
    &=
    -(2b)^{-1}
    \sum_{\substack{
      1 \leq a_1,a_2 < v: \\
      b p = a_1 + a_2
    }}
      a_1
      [D_{a_1},D_{a_2}]
    &\textnormal{for all $b\in\notp$ with $b<v$},
    \label{eq:small-v-1}
    \\
    D_b
    &=
    -(2b)^{-1}
    \sum_{\substack{
      1 \leq a_1, a_2 < v:\\
      b = a_1+a_2
    }}
      a_1
      [D_{a_1}, D_{a_2}]
    &\textnormal{for all $b\in\notp$ with $b\geq v$},
    \label{eq:small-v-2}
    \\
    0
    &=
    [\sigma^i D_{\lfloor v\rfloor}, D_{a}]
    &\hspace{-1.2cm}\begin{array}{l}\textnormal{for all $1\leq a<v$ and $i>0$ with $a p^{-i} \geq v - \lfloor v\rfloor$}\\\textnormal{if $v$ is not an integer}.\end{array}
    \label{eq:small-v-3}
  \end{align}
\end{corollary}

\begin{proof}
  We apply \Cref{thm:lastjump-Jv}.
  Here,
  $\mu_v(b) = \mathbbm1_{b<v}$.
  In all summands in \Cref{eq:cF-2-integer-nu,eq:cF-2-noninteger-nu}, the conditions of the form $a p^n < v$ are equivalent to $n = 0$, $a < v$.
  \Cref{eq:small-v-1,eq:small-v-2} thus correspond to \Cref{eq:cF-2-integer-nu} (note that $\eta(n_1, n_2) = \frac12$ as $n_1=n_2=0$), and \Cref{eq:small-v-3} corresponds to \Cref{eq:cF-2-noninteger-nu}.
  (For $i>0$, any $b\geq vp^i$ can be written in at most one way as $b=a_1p^i+a_2$ with $1\leq a_1,a_2<v\leq p$, namely $a_2$ must be the remainder of $b$ modulo $p$, and $a_1$ must be $\lfloor v\rfloor$.
  Hence, at most one summand appears in each instance of \Cref{eq:cF-2-noninteger-nu}.
  If $v$ is an integer, then $a_1 = \lfloor v\rfloor= v$ is not strictly smaller than $v$.)
\end{proof}

\begin{corollary}
  \label{cor:large-Da-determined}
  Let $v\leq p$.
  Given elements $D_a \in \fg \otimes W(\kappa)$ for all $a<v$, there is at most one choice of the remaining elements $D_a$ for $a\geq v$ such that the element $D \coloneqq \sum D_a \tilde\pi^{-a}$ satisfies $\lastjump(D)<v$.
\end{corollary}

\begin{proof}
  This follows immediately from \Cref{eq:small-v-2}.
\end{proof}

We obtain the following characterization of unramified extensions, which directly implies \iref{thm:local-counting}{item:count-unramified}:

\begin{lemma}[Unramified extensions]
  \label{lem:unramified}
  The following are equivalent:
  \begin{enumroman}
    \item
      \label{lem:unramified-i}
      $\lastjump(D)<1$.
    \item
      \label{lem:unramified-ii}
      $D_a=0$ for all $a\in\notp$.
    \item
      \label{lem:unramified-iii}
      $\lastjump(D)=0$.
  \end{enumroman}
\end{lemma}

\begin{proof}
  The implication \ref{lem:unramified-iii}$\Rightarrow$\ref{lem:unramified-i} is obvious.
  The implications \ref{lem:unramified-i}$\Rightarrow$\ref{lem:unramified-ii} and \ref{lem:unramified-ii}$\Rightarrow$\ref{lem:unramified-iii} follow from \Cref{cor:small-v} with $v=1$ and $v \rightarrow 0^+$, respectively.
\end{proof}

\begin{proposition}
  \label{prop:precise-slightly-ramified}
  For any $m \geq 0$, the following are equivalent:
  \begin{enumroman}
    \item
      $\lastjump(D) < 1 + p^{-m}$
    \item
      $\lastjump(D) \leq 1 + p^{-(m+1)}$
    \item
      \label{prop:small-v-iii}
      $D_a=0$ for all $a\geq2$, $pD_1=0$, and $[\sigma^i D_1, D_1] = 0$ for all $0 < i \leq m$.
  \end{enumroman}
\end{proposition}

\begin{proof}
  Apply \Cref{cor:small-v}.
  For any $v \in (1,2]$, the inequalities $1 \leq a_1, a_2 < v$ imply that $a_1=a_2=1$, and in particular commutators $[D_{a_1}, D_{a_2}]$ vanish.
  Therefore, \Cref{eq:small-v-1} is equivalent to $p D_1 = 0$, and \Cref{eq:small-v-2} is equivalent to $D_b = 0$ for $b \geq 2$.
  Finally, \Cref{eq:small-v-3} amounts to $[\sigma^i D_1, D_1]=0$ for all $i > 0$ such that $p^{-i} \geq v-1$.
  Thus, for any $v \in (1+p^{-(m+1)}, 1+p^{-m}]$, \ref{prop:small-v-iii} is equivalent to $\lastjump(D) < v$, yielding the result.
\end{proof}

\begin{remark}
  \label{rk:amk}
  Let $m \geq 0$.
  For an element $D \in \fg \otimes \cD$, condition \ref{prop:small-v-iii} of \Cref{prop:precise-slightly-ramified} means that~$D$ is of the form $D_1 \tilde\pi^{-1}$ where $D_1$ is an element of $\fg[p]\otimes W(\kappa) = \fg[p] \otimes_{\F_p} \kappa$ whose projection~$x$ to $(\fg[p]/Z(\fg[p])) \otimes_{\F_p} \kappa$ belongs to the set
  \[
    A_m(\kappa)
    \coloneqq
    \suchthat{
      x \in (\fg[p]/Z(\fg[p])) \otimes \kappa
    }{
      [\sigma^i(x), x] = 0 \textnormal{ for all } 0 < i \leq m
    }.
  \]
  In particular, \Cref{prop:precise-slightly-ramified} implies the following equality:
  \[
    \cardsuchthat{
      D \in \fg \otimes \cD
    }{
      \lastjump(D) < 1+p^{-m}
    }
    =
    \underbrace{
      |A_m(\kappa)|
    }_{\textnormal{choices of } x}
    \ 
    \cdot
    \underbrace{
      |Z(\fg[p]) \otimes \kappa|.
    }_{\textnormal{choices of } D_1 \textnormal{ once } x \textnormal{ is fixed}}
  \]
\end{remark}

From \Cref{prop:precise-slightly-ramified}, we deduce the following characterization, which directly implies \iref{thm:local-counting}{item:count-least-ramified} (we have $|\fg[p] \otimes_{\Z_p} W(\kappa)| = |\fg[p] \otimes_{\F_p} \kappa| = |\kappa|^r$).

\begin{corollary}[Slightly ramified extensions]
  \label{cor:slightly-ramified}
  The following are equivalent:
  \begin{enumroman}
    \item
      \label{prop:slightly-ramified-i}
      $\lastjump(D)<2$.
    \item
      $\lastjump(D)\leq M$.
    \item
      \label{prop:slightly-ramified-iii}
      $D_a=0$ for all $a\geq2$, and $pD_1=0$.
  \end{enumroman}
\end{corollary}

\begin{proof}
  If $\fg[p]$ is non-abelian, this is simply the case $m=0$ of \Cref{prop:precise-slightly-ramified}.
  Assume now that~$\fg[p]$ is abelian.
  Then, when $pD_1=0$, all commutators $[\sigma^i D_1, D_1]$ vanish, making condition~\ref{prop:small-v-iii} of \Cref{prop:precise-slightly-ramified} independent of the value of $m$ (it matches condition~\ref{prop:slightly-ramified-iii} here).
  It remains only to observe that $\lastjump(D) < 2$ implies~\ref{prop:small-v-iii} and hence implies $\lastjump(D) < 1+p^{-m}$ for all $m \geq 0$, meaning that $\lastjump(D) \leq 1$.
\end{proof}

Finally, we generalize the implication \ref{prop:slightly-ramified-i} $\Rightarrow$ \ref{prop:slightly-ramified-iii} of \Cref{prop:precise-slightly-ramified} to slightly larger values of the last jump:

\begin{proposition}
  \label{prop:cor-l-plus}
  Let $1\leq l\leq p-1$ and $m\geq1$.
  \begin{enumroman}
    \item\label{item:cor-l-plus-1}
      If $\lastjump(D) < l + p^{-m}$, then $[\sigma^i D_l,D_a] = 0$ for all $1\leq i\leq m$ and $1\leq a\leq l$.
    \item\label{item:cor-l-plus-l}
      If $\lastjump(D) < l + lp^{-m}$, then $[\sigma^i D_l,D_a] = 0$ for all $1\leq i\leq m-1$ and $1\leq a\leq l-1$ and $[\sigma^i D_l,D_l] = 0$ for all $1\leq i\leq m$.
  \end{enumroman}
\end{proposition}

\begin{proof}
  ~
  \begin{enumroman}
    \item
      For any $1\leq i\leq m$ and $1\leq a\leq l$, we have $ap^{-i} \geq p^{-m}$.
      Hence, \Cref{eq:small-v-3} implies $[\sigma^i D_l,D_a] = 0$.
    \item
      This follows in the same way since $ap^{-i}\geq lp^{-m}$ for $1\leq i\leq m-1$ and $1\leq a\leq l-1$ and since $lp^{-i}\geq lp^{-m}$ for $1\leq i\leq m$.
    \qedhere
  \end{enumroman}
\end{proof}

\subsection{Equations in Witt vectors}
\label{subsn:eq-witt-vec}

Before we start proving points \ref{item:count-more-ramified-integer} and \ref{item:count-more-ramified-noninteger} of \Cref{thm:local-counting}, we advise the reader to first read the proof in the case where~$\fg$ has exponent~$p$, which is given in \Cref{prop:large-v-expp}: this special case is much easier to deal with, and gives a simple illustration of the rough idea of our general proof --- namely that fixing $D_b$ for the smallest few values of $b$ essentially determines $D_b$ for all $b$.

Let $r \coloneqq \dim_{\F_p}(\fg[p])$.
The finite $\Z_p$-module $\fg$ decomposes into a product:
\[
  \fg
  \simeq
  \prod_{i=1}^r
    \Z/p^{n_i}\Z
  \quad\quad
  \textnormal{as $\Z_p$-modules.}
\]
We have $\fg\otimes W(\kappa)\simeq\prod_i W_{n_i}(\kappa)$ as $\Z_p$-modules (see also \Cref{lem:size-tensor}).
If $X$ is an element of $\fg \otimes W(\kappa)$, $1\leq i\leq r$, and $0 \leq j < n_i$, we denote by $X^{(ij)} \in \kappa$ the $j$-th coordinate of the $i$-th Witt vector (in $W_{n_i}(\kappa)$) associated to $X$.

Let $v \in \R_{>0}$.
According to \Cref{thm:lastjump-Jv}, equations characterizing elements $D \in \fg \otimes \cD$ such that $\lastjump(D) < v$ are given in \Cref{def:property-Jv}.
We show below that these equations are polynomial equations%
\footnote{
  As equations over the Witt vectors, these equations are not polynomial as they involve the absolute Frobenius endomorphism~$\sigma$, which is why we rephrase them as equations over $\kappa$.
  They are, however, ``algebraic difference equations'' (in the sense of \cite{belair,langweiltordu,hils2024langweiltypeestimatesfinite}) over the difference ring $(W(\kappa), \sigma)$.
}
in the indeterminates $D_a^{(ij)} \in \kappa$, where the triple $(a,i,j)$ belongs to the countable set
\[
  \Omega
  \coloneqq
  \suchthat{
    (a,i,j)
  }{
    a\in\notp,\ 1\leq i\leq r,\ 0\leq j<n_i
  }.
\]
Moreover, these equations do not depend on anything besides the Lie algebra $\fg$ and the real number $v$ (they depend neither on $\fF$ nor on $\kappa$, and have coefficients in $\F_p$).

Consider the following polynomial ring in countably many variables:
\[
  \cR \coloneq \F_p[(D_a^{(ij)})_{(a,i,j) \in \Omega}]
\]
and its perfect closure
\[
  \cR^\perf = \F_p[((D_a^{(ij)})^{p^{-\infty}})_{(a,i,j)\in\Omega}].
\]
Since $\prod_i W_{n_i}(\cR)$ is the image of the Lie algebra $\fg\otimes W(\cR)$ in $\fg\otimes W(\cR^\perf) = \prod_i W_{n_i}(\cR^\perf)$, it naturally inherits a Lie algebra structure.
Define $D_a$ to be the element of $\prod_i W_{n_i}(\cR)$ for which the $j$-th coordinate of the $i$-th Witt vector is the indeterminate $D_a^{(ij)}$.
Then, each side of the \Cref{eq:cF-2-integer-nu,eq:cF-2-noninteger-nu} defining property $J(v)$ (for given $b$ and $i$) can be interpreted as an element of~$\prod_i W_{n_i}(\cR)$.
Equating the coordinates of these elements for all admissible values of~$b$ and~$i$, we obtain (polynomial) equations in the indeterminates $D_a^{(ij)}$.
We let $\cI_v$ be the ideal of~$\cR$ associated to these equations.

By construction, elements $D \in \fg\otimes\cD$ such that $\lastjump(D)<v$ are in one-to-one correspondence with the solutions to these polynomial equations in $\kappa^\Omega$, i.e., with the $\kappa$-points of the affine $\F_p$-scheme $\J_v \coloneq \Spec(\cR/\cI_v)$.
(That all but finitely many coordinates vanish actually follows from the equations, cf.~\iref{prop:consequences-Jv}{prop:consequences-Jv-iv}.)

The following remark is not used in this paper, but it suggests a moduli space interpretation:

\begin{remark}
  \label{rk:mod-space}
  Let $\J^0_v = \mathbb A^{\log_p\!|\fg|}_{\F_p} \times \J_v$ (reincorporating the variables $D_0^{(ij)}$).
  By \Cref{thm:lastjump-Jv}, there is a bijection~$\J^0_v(\kappa) \simeq \suchthat{D \in \fg \otimes \cD^0}{\lastjump(D) < v}$.
  Up to some action, the space~$\J^0_v$
  thus seems to admit an interpretation as a moduli space parametrizing ramified $G$-covers of the one-dimensional infinitesimal neighborhood of a single point at which the ``depth'' of wild ramification is $<v$.
  This bears some analogy to Hurwitz spaces, which parametrize tamely ramified covers of the line with a fixed number of branch points.

\end{remark}

\noindent\textbf{Grading.}~
We define a $\Q_{\geq0}$-grading on the $\F_p$-algebra $\cR^\perf$ by assigning the degree $p^{-\mu_v(a)+j}$ to the generator~$D_a^{(ij)}$.
Note that~$\mu_v(a)$ is bounded, so the degrees of monomials in $\cR$ form a discrete (hence well-founded) subset of $\R_{\geq0}$.
The following general remark shows how the grading on~$\cR^\perf$ induces a partial grading on $\fg\otimes W(\cR^\perf)$, for which $D_a\in\fg\otimes W(\cR^\perf)$ is homogeneous of degree $p^{-\mu_v(a)}$:

\begin{remark}
\label{rmk:witt-grading}
  Let $R = \bigoplus_{d\in\Q_{\geq0}} R_d$ be a $\Q_{\geq0}$-graded ring of characteristic~$p$.
  For any $d\in\Q_{\geq0}$, let $V_d\subseteq W(R)$ be the set of Witt vectors over $R$ whose $j$-th coordinate is homogeneous of degree $dp^j$ for all $j$:
  \[
    V_d
    \coloneqq
    \suchthat{
      (x_0,x_1,\dots) \in W(R)
     }{
      x_j \in R_{dp^j} \textnormal{ for all }j\geq0
    }.
  \]
  Looking at the definition of addition and multiplication of Witt vectors, one verifies that each~$V_d$ is a $\Z_p$-submodule of $W(R)$, that if $(x_0,x_1,\dots)\in\bigoplus_{d'\leq d}V_{d'}$, then $x_j\in \bigoplus_{d'\leq d} R_{d'p^j}$ for all $j\geq0$, and that $V_d\cdot V_e\subseteq V_{d+e}$.
  Moreover, if $\sigma$ is the Frobenius endomorphism of $R$, then $\sigma(V_d)\subseteq V_{dp}$.
  If $\fg$ is a Lie $\Z_p$-algebra, we obtain $\Z_p$-submodules $\fg\otimes V_d$ of $\fg\otimes W(R)$ satisfying $[\fg\otimes V_d, \fg\otimes V_e] \subseteq \fg\otimes V_{d+e}$.
  We call the elements of $V_d$ or of $\fg\otimes V_d$ \emph{homogeneous of degree $d$}.
\end{remark}

\subsection{Bounding the number of extensions with large last jump.}
\label{subsn:counting-locext}
~

\npar{Proving \iref{thm:local-counting}{item:count-more-ramified-integer}.}
Fix an integer $l \geq 0$, and define the following finite subset of $\Omega$:
\[
  \Omega^\bullet_l
  \coloneqq
  \suchthat{
    (a,i,j) \in \Omega
  }{
    n_i - \mu_{l+1}(a) \leq j
  },
\]
whose cardinality satisfies:
\begin{equation}
  \label{eqn:size-omega-bullet}
  |\Omega^\bullet_l|
  =
  \sum_{a\in\notp}
    \sum_{i=1}^r
      \min(n_i, \, \mu_{l+1}(a))
  \leq
  r \sum_{a\in\notp}\mu_{l+1}(a)
  \underset{\textnormal{\Cref{lem:sum-mu}}}{=\joinrel=}
  r l.
\end{equation}
We shall show that, if $D \in \fg \otimes \cD$ is such that $\lastjump(D) < l+1$, then the coordinates $D_a^{(ij)}$ for triples $(a,i,j)\in\Omega^\bullet_l$ determine $D$ up to at most $|G|^{2l}$ choices.
To this end, we introduce the following polynomial ring in $|\Omega^\bullet_l|$ variables:
\[
\cR^\bullet_l \coloneq \F_p[(D_a^{(ij)})_{(a,i,j) \in \Omega^\bullet_l}].
\]

\begin{proposition}
  \label{prop:bound-rank}
  The $k$-algebra $\cR/\cI_{l+1}$ is a finite $\cR^\bullet_l$-module of rank at most $|G|^{2l}$.
\end{proposition}

\begin{proof}
  We are going to show that every polynomial in $\cR$ is congruent, modulo $\cI_{l+1}$, to a linear combination with coefficients in $\cR^\bullet_l$ of monomials of the form
  \begin{equation}
    \label{eqn:form-of-final-monomials}
    \prod_{(a,i,j)\in\Omega\setminus\Omega^\bullet_l}
      \left(D_a^{(i j)}\right)^{e_{a i j}}
    \qquad\qquad
    \textnormal{where }
    0\leq e_{a i j} < p^{2 \mu_{l+1}(a)}.
  \end{equation}
  As the number of such monomials is
  $
    \prod_{(a,i,j)\in\Omega\setminus\Omega^\bullet_l}
      p^{2 \mu_{l+1}(a)}
  $,
  which by \Cref{lem:sum-mu} is at most
  $
    p^{2l\sum_i n_i}
    =
    |G|^{2l}
  $,
  this implies the result.

  Recall that $D_a$ is homogeneous of degree $p^{-\mu_{l+1}(a)}$ for any $a$, and recall the properties listed in \Cref{rmk:witt-grading}.
  We see that the left-hand side $p^{\mu_{l+1}(b)}\sigma^{\mu_{l+1}(b)}(D_b)$ of \Cref{eq:cF-2-integer-nu} is homogeneous of degree $p^{\mu_{l+1}(b)}\cdot p^{-\mu_{l+1}(b)}=1$.
  (The multiplication by $p^{\mu_{l+1}(b)}$ does not change the degree. The Frobenius homomorphism $\sigma^{\mu_{l+1}(b)}$ multiplies the degree by $p^{\mu_{l+1}(b)}$.)
  The right-hand side is a sum of terms which are homogeneous of degree at most~$\frac2p<1$.
  ($\sigma^{n_1}(D_{a_1})$ is homogeneous of degree $p^{n_1}\cdot p^{-\mu_{l+1}(a_1)}$, which is $\leq\frac1p$ as $a_1 p^{n_1} < l+1$, and similarly for $\sigma^{n_2}(D_{a_2})$, so each summand is homogeneous of degree at most~$\frac2p$.)


  Assume now that $(a,i,j)$ belongs to $\Omega \setminus \Omega^\bullet_l$, i.e., that $j + \mu_{l+1}(a) < n_i$, and consider the $(i, \, j + \mu_{l+1}(a))$-coordinate of \Cref{eq:cF-2-integer-nu} for $b=a$.
  Recall that multiplication by $p$ of Witt vectors coincides with ${\Ver} \circ \sigma$.
  The left-hand side
  \begin{align*}
    \Big(
      p^{\mu_{l+1}(a)}
      \sigma^{\mu_{l+1}(a)}D_a
    \Big)^{(i, \, j + \mu_{l+1}(a))}
    &=
    \Big(
      \Ver^{\mu_{l+1}(a)}
      \sigma^{2\mu_{l+1}(a)}
      D_a
    \Big)^{(i, \, j + \mu_{l+1}(a))} \\
    &=
    \Big(
      \sigma^{2\mu_{l+1}(a)}
      D_a
    \Big)^{(i j)}
    =
    \left(
      D_a^{(i j)}
    \right)^{p^{2\mu_{l+1}(a)}}
  \end{align*}
  has degree $p^{2 \mu_{l+1}(a)} \cdot p^{-\mu_{l+1}(a)+j} = p^{j+\mu_{l+1}(a)}$, and the right-hand side has strictly smaller degree by the previous paragraph.
  For all triples $(a,i,j) \in \Omega \setminus \Omega^\bullet_l$, interpret the $(i, \, j+\mu_{l+1}(a))$-coordinate of \Cref{eq:cF-2-integer-nu} as a rewriting rule (from left to right): any monomial in $\cR$ which is divisible by
  $
    \left(
      D_a^{(i j)}
    \right)^{p^{2\mu_{l+1}(a)}}
  $
  can be rewritten by replacing this factor with the right-hand side.
  All the rules obtained for all triples $(a,i,j) \in \Omega \setminus \Omega^\bullet_l$ define a rewriting system on $\cR$.
  The strict inequality of degrees shown above together with the well-foundedness of the set of possible degrees of monomials implies that this rewriting system is strongly normalizing.

  Consider an arbitrary $P \in \cR$, and let $\tilde P \in \cR$ be a normal form of $P$, obtained by rewriting~$P$ arbitrarily until it cannot be rewritten any further.
  We have $P \equiv \tilde P \bmod {\cI_{l+1}}$ by definition of~$\cI_{l+1}$ and of the rewriting system.
  By definition, $\tilde P$ cannot be rewritten and thus does not contain any monomial divisible by $\left(D_a^{(ij)}\right)^{p^{2 \mu_{l+1}(a)}}$ for any $(a,i,j) \in \Omega\setminus\Omega^\bullet_l$.
  This means that $\tilde P$ belongs to the $\cR^\bullet_l$-module spanned by monomials of the form given in \Cref{eqn:form-of-final-monomials}, concluding the proof.
\end{proof}

\Cref{prop:bound-rank} yields the following corollary, which implies \iref{thm:local-counting}{item:count-more-ramified-integer}:

\begin{corollary}
  \label{cor:counting-bounded-last-jump}
  For every tuple $(d_{a,i,j})_{(a,i,j)\in\Omega^\bullet_l} \in \kappa^{|\Omega^\bullet_l|}$, we have:
  \[
    \cardsuchthat{
      D \in \fg \otimes \cD
    }{
      \begin{matrix}
        D_a^{(ij)} = d_{a,i,j} \textnormal{ for all } (a,i,j) \in \Omega^\bullet_l \\
        \lastjump(D) < l+1
      \end{matrix}
    }
    \leq
    |G|^{2l}.
  \]
  In particular, we have \iref{thm:local-counting}{item:count-more-ramified-integer} with $E = |G|^2$:
  \[
    \cardsuchthat{
      D \in \fg \otimes \cD
    }{
      \lastjump(D) < l+1
    }
    \leq
    |G|^{2l}
    |\kappa|^{|\Omega^\bullet_l|}
    \leq
    |G|^{2l}
    |\kappa|^{r l}.
  \]
\end{corollary}

\begin{proof}
  Denote by $f_l : \cR^\bullet_l \to \cR/\cI_{l+1}$ the ring homomorphism obtained as the composition $\cR^\bullet_l \hookrightarrow \cR \twoheadrightarrow \cR/\cI_{l+1}$.
  By \Cref{prop:bound-rank}, the map $f_l$ is finite of degree at most $|G|^{2l}$.
  The tuple~$d_{a,i,j}$ corresponds to a ring homomorphism $d : \cR^\bullet_l \to \kappa$ (induced by $D_a^{(ij)} \mapsto d_{a,i,j}$), and thus to a maximal ideal $\fm \subseteq \cR^\bullet_l$ (namely $\fm \coloneqq \ker(d)$) whose residue field is a subfield of $\kappa$ (the image of~$d$).
  We have:
  \begin{align*}
    &
    \cardsuchthat{
      D \in \fg \otimes \cD
    }{
      \begin{matrix}
        D_a^{(ij)} = d_{a,i,j}
        \textnormal{ for all } (a,i,j) \in \Omega^\bullet_l \\
        \lastjump(D) < l+1
      \end{matrix}
    }
    \\
    = \quad
    &
    \cardsuchthat{
      D : \cR/\cI_{l+1} \to \kappa
    }{
      D \circ f_l = d
    }
    &
    \textnormal{by definition of } \cI_{l+1}
    \\
    = \quad
    &
    \cardsuchthat{
      \fn \textnormal{ maximal ideal of } \cR/\cI_{l+1}
    }{
      \begin{matrix}
        f_l^{-1}(\fn) = \fm \\
        (\cR/\cI_{l+1})/\fn \subseteq \kappa
      \end{matrix}
    }
    \\
    \leq \quad
    &
    \rank_{\cR^\bullet_l} \
      \cR/\cI_{l+1}
      &
      \textnormal{by \Cref{lem:bound-size-fiber} below}
    \\
    \leq \quad
    &
    |G|^{2l}
    &
    \textnormal{by \Cref{prop:bound-rank}.}
    &
    \qedhere
  \end{align*}
\end{proof}

In the proof, we have used the following lemma:

\begin{lemma}
  \label{lem:bound-size-fiber}
  Let $f : A \to B$ be a finite ring homomorphism, and let $\fm$ be a maximal ideal of~$A$.
  Then, the number of maximal ideals of $B$ containing~$f(\fm)$ is at most $\rank_A B$.
\end{lemma}

\begin{proof}
  Replacing $A$ by $A/\fm$ and $B$ by $B \otimes_A A/\fm$ (this may only decrease the rank of $B$), we may assume that $A$ is a field and that $\fm = 0$.
  Since $B$ is finite over the field $A$, it is Artinian and hence has finitely many maximal ideals.
  We want to show that~$B$ has at most $\dim_A B$ maximal ideals.
  The Chinese remainder theorem gives an isomorphism:
  \[
    \bigoplus_{\fn \subseteq B \textnormal{ maximal}}
      B/\fn
    \simeq
    B/J(B)
  \]
  where $J(B)$ is the Jacobson radical of $B$.
  We have:
  \[
    |\{\fn \subseteq B \textnormal{ maximal}\}|
    \leq
    \sum_{\fn \subseteq B \textnormal{ maximal}}
      \dim_A (B/\fn)
    =
    \dim_A (B/J(B))
    \leq
    \dim_A B.
    \qedhere
  \]
\end{proof}

\npar{Proving \iref{thm:local-counting}{item:count-more-ramified-noninteger}.}
We now assume that $\fg[p]$ is non-abelian.

\begin{proof}[Proof of \customref{\Cref*{thm:local-counting}~\ref*{item:count-more-ramified-noninteger}\ref*{item:count-more-ramified-large}}{item:count-more-ramified-large}]
  Let $l \geq p$.
  For any $x,y \in \fg[p^2]$, we have $[px,py]=[p^2x,y]=0$, so $p \cdot \fg[p^2]$ is an abelian subalgebra of $\fg[p]$.
  Since $\fg[p]$ is non-abelian, the subalgebra $p \cdot \fg[p^2]$ is strictly contained in $\fg[p]$, which implies that some exponent $n_i$ in the decomposition $\fg\simeq\prod_i\Z/p^{n_i}\Z$ equals $1$.
  Since $l \geq p$, we have $\mu_{l+1}(1) \geq 2$.
  Therefore, the inequality in \Cref{eqn:size-omega-bullet} is strict: we have $|\Omega^\bullet_{l+1}| \leq rl-1$, and \Cref{cor:counting-bounded-last-jump} gives the desired bound.
\end{proof}

\begin{proof}[Proof of \customref{\Cref*{thm:local-counting}~\ref*{item:count-more-ramified-noninteger}\ref*{item:count-more-ramified-small}}{item:count-more-ramified-small}]
  Let $l \in \{1, \ldots, p-1\}$.
  We apply \Cref{cor:small-v} with $v = l+lp^{-1}$.
  For $b \neq l$, \Cref{eq:small-v-1,eq:small-v-2} are identical to the equations obtained for $v = l+1$, which were analyzed above.
  For $b = l$, \Cref{eq:small-v-1} says that $p \sigma(D_l)=0$ (the sum is empty because $a_1, a_2 < v$ implies that $a_1+a_2 \leq 2 l < l p$), so $D_l \in \fg[p]\otimes\kappa$.
  \Cref{eq:small-v-3} says that $[\sigma(D_l),D_l]=0$ (only $i=1$ is possible).

  To restrict the possibilities for $D_l$, we look for an upper bound on the number of elements of $x \in \fg[p] \otimes \kappa$ satisfying $[\sigma(x), x]=0$.
  For $x,y\in \fg[p]\otimes\kappa$, we can interpret $[x,y]=0$ as a system of bilinear polynomial equations in the $r$ coordinates $x_i, y_i \in \kappa$ of $x$ and $y$.
  By assumption, the Lie $\F_p$-algebra~$\fg[p]$ is non-abelian, so one of these polynomial equations is non-trivial.
  Consider any non-zero monomial $c x_i y_j$ in this equation.
  Substituting $x = \sigma(y)$, the resulting polynomial equation (in the $r$ variables $y_i$) has a non-zero monomial $c y_i^p y_j$, so in particular the closed subvariety (of $r$-dimensional affine space) cut out by the equation $[\sigma(y),y]$ has dimension at most $r-1$.
  By \cite[Lemma~1]{langweilig}, the number of solutions $D_l\in \fg[p]\otimes \kappa$ to $[\sigma(D_l),D_l]=0$ is then $\cO(|\kappa|^{r-1})$ (where the implied constant depends only on $\fg[p]$).
  
  Consider one such solution $D_l \in \fg[p]\otimes \kappa$.
  For the other elements $D_b\in \fg\otimes W(\kappa)$ (with $b \neq l$), we reason exactly like in the proof of \iref{thm:local-counting}{item:count-more-ramified-integer}, using \Cref{cor:counting-bounded-last-jump}.
  The analogue of~$\Omega^\bullet_l$ is now given by the set
  $
    \suchthat{
      (a,i,j) \in \Omega
    }{
      a \neq l
      \textnormal{ and }
      n_i - \mu_{l+1}(a) \leq j
    }
  $
  which has cardinality:
  \[
    \sum_{\substack{
      a \neq l\\
      a \in \notp
    }}
      \sum_{i=1}^r
        \min(n_i, \, \mu_{l+1}(a))
    \leq
    r
    \sum_{\substack{
      a \neq l\\
      a \in \notp
    }}
      \mu_{l+1}(a)
    =
    r
    \Big(
      l
      -
      \underbrace{
        \mu_{l+1}(l)
      }_{=1}
    \Big)
    = r (l-1).
  \]
  Finally, we obtain the desired bound on the total number of solutions:
  \begin{align*}
    \cardsuchthat{
      D \in \fg \otimes \cD
    }{
      \lastjump(D) < l(1+p^{-1})
    }
    &=
    \cO\Big(
      \underbrace{
        |\kappa|^{r-1}
      }_{\textnormal{choices of $D_l$}}
      \cdot
      \underbrace{
        E^{l+lp^{-1}-1}
        |\kappa|^{r(l-1)}
      }_{\textnormal{choices of $(D_b)_{b \neq l}$ once $D_l$ is fixed}}
    \Big)
    \\
    &=
    \cO\!\left(
      E^l|\kappa|^{r l-1}
    \right).
    \qedhere
  \end{align*}
\end{proof}

\npar{Refinement of the value of $E$.}
Finally, we note that the bound of \Cref{thm:local-counting} (points~\ref{item:count-more-ramified-integer} and~\ref{item:count-more-ramified-noninteger}) can sometimes%
\footnote{
  We did not manage to find any Lie algebra for which the inequality does not hold with $E=1$.
}
be refined:

\begin{proposition}
  \label{prop:better-bound}
  Assume that there is a Lie subalgebra $\fn \subseteq Z(\fg)$ containing $[\fg,\fg]$ such that, for all~$n \geq 1$:
  \begin{equation}
    \label{eqn:weird-hyp}
    \fn \cap p^n \fg = p^n \fn.
  \end{equation}
  Then, for all integers $l \geq 0$, we have \iref{thm:local-counting}{item:count-more-ramified-integer} with $E=1$:
  \[
    \cardsuchthat{
      D \in \fg \otimes \cD
    }{
      \lastjump(D) < l+1
    }
    \leq
    |\kappa|^{r l}.
  \]
\end{proposition}

\Cref{eqn:weird-hyp} is automatically satisfied for $n \geq \max_i n_i$.
In particular, the proposition applies automatically (for example, taking $\fn = Z(\fg)$) when $\fg = \fg[p]$ is of exponent~$p$.
(This case will be discussed again in \Cref{subsn:lastjump-in-expp}.)
It also applies if $\fg$ is abelian, taking $\fn = \fg$.

\begin{proof}
  Modulo $\fn \otimes W(\kappa)$, \Cref{eq:cF-2-integer-nu} becomes $p^{\mu_{l+1}(b)} \bar D_b = 0$ for all $b$ (as $\fn$ contains $[\fg,\fg]$).
  First choose $\bar D_b \in (\fg/\fn)[p^{\mu_{l+1}(b)}] \otimes W(\kappa)$ for all $b$.
  An upper bound for the number of choices for~$\bar D_b$ is:
  \[
    \prod_{b \in \notp}
      \Big|
        (\fg/\fn)[p^{\mu_{l+1}(b)}] \otimes W(\kappa)
      \Big|.
  \]
  Having determined $D_b$ modulo $\fn \otimes W(\kappa)$ for all $b$ (and hence modulo $Z(\fg) \otimes W(\kappa)$), it follows from \Cref{eq:cF-2-integer-nu} that $p^{\mu_{l+1}(b)} D_b$ is determined for all $b$.
  Hence, $D_b$ is determined modulo $(\fg[p^{\mu_{l+1}(b)}] \cap \fn) \otimes W(\kappa) = \fn[p^{\mu_{l+1}(b)}]\otimes W(\kappa)$.
  An upper bound for the number of remaining choices for $D_b$ (once $\bar D_b$ is fixed for all $b$) is:
  \[
    \prod_{b \in \notp}
      \Big|
        \fn[p^{\mu_{l+1}(b)}] \otimes W(\kappa)
      \Big|.
  \]
  Note the exact sequence (for any $n \geq 0$):
  \[
    0
    \to
    \fg[p^n]/\fn[p^n]
    \to
    (\fg/\fn)[p^n]
    \overset{\times p^n}\to
    (\fn \cap p^n \fg)/p^n \fn
    \to
    0.
  \]
  By hypothesis, the quotient $(\fn \cap p^n \fg)/p^n \fn$ is trivial, so that $(\fg/\fn)[p^n] = \fg[p^n]/\fn[p^n]$.
  Therefore, the total number of choices for~$D_b$ is bounded above by:
  \begin{align*}
    \prod_{b \in \notp}
      \Big|
        (\fg[p^{\mu_{l+1}(b)}]/\fn[p^{\mu_{l+1}(b)}])
        \otimes W(\kappa)
      \Big|
      \cdot
      \Big|
        \fn[p^{\mu_{l+1}(b)}] \otimes W(\kappa)
      \Big|
    &
    =
    \prod_{b \in \notp}
      \Big|
        \fg[p^{\mu_{l+1}(b)}] \otimes W(\kappa)
      \Big|
    \\
    &=
    |\kappa|^{\sum \log_p |\fg[p^{\mu_{l+1}(b)}]|}.
  \end{align*}
  It remains only to estimate the exponent, using the inequalities $|\fg[p^k]|\leq|\fg[p]|^k$ for $k\geq0$:
  \begin{align*}
    \sum_{b \in \notp}
      \log_p |\fg[p^{\mu_{l+1}(b)}]|
    &\leq
    \sum_{b \in \notp}
      r \mu_{l+1}(b)
  \underset{\textnormal{\Cref{lem:sum-mu}}}{=\joinrel=}
    r l.
    \qedhere
  \end{align*}
\end{proof}

\section{Global asymptotics}
\label{sn:global}

We fix a rational function field $F \coloneqq \F_q(T)$ of characteristic~$p$ and a finite Lie $\Z_p$-algebra $\fg \neq 0$ of nilpotency class at most $2$.
Let $G \coloneqq (\fg, \circ)$ be the corresponding $p$-group.
For each place~$P$ of $F = \F_q(T)$, we make the following definitions:
\begin{itemize}
  \item
    $F_P$ is the completion of $F$ at $P$, $\kappa_P$ is its residue field, and $\pi_P$ is a uniformizer of $F_P$, so that $F_P = \kappa_P\llpar\pi_P\rrpar$.
    Note that $|\kappa_P| = q^{\deg(P)}$ is the absolute norm of the prime $P$;
  \item
    $\tilde\pi_P = (\pi_P, 0, 0, \ldots)$ is the Teichmüller representative of~$\pi_P$ in $W(F_P) \subseteq W(F_P^\perf)$.
  \item
    The following objects are defined exactly as in \Cref{sn:local} in the case $(\fF, \pi) = (F_P, \pi_P)$, adding a subscript $P$ to the notation:
    \begin{itemize}
      \item
        The $W(\kappa_P)$-submodules $\cD^0_P$ and $\cD_P$ of $W(F_P) \subseteq W(F_P^\perf)$
        (cf.~\Cref{def:cD-zero,def:cD})
      \item
        The maps $\alpha^0_P : \Orb{\fg}{W(F_P^\perf)}{W(F_P^\perf)} \simto \Orb{\fg}{\cD^0_P}{W(\kappa_P)}$
        (cf.~\Cref{thm:local-fundamental-domain}) and $\alpha_P : \Orb{\fg}{W(F_P^\perf)}{W(F_P^\perf)} \twoheadrightarrow \Orb{\fg}{\cD_P}{W(\kappa_P)}$ (cf.~\Cref{def:alpha}).
    \end{itemize}
    We also use the notation $\lastjump(D)$ if $D \in \fg \otimes \cD^0_P$ (cf.~\Cref{def:lastjump-in-fundom}) or if $D \in \fg \otimes \cD_P$ (cf.~\Cref{def:lastjump-cD}).
\end{itemize}

\medskip

This section is organized as follows:
In \Cref{subsn:locglob}, we prove the local--global principle (\Cref{thm:intro-local--global}).
Then, in \Cref{subsn:analytic}, we establish a general analytic lemma (\Cref{lem:analytic-lemma}), which lets one deduce global asymptotics from local estimates when combined with the local--global principle.
Finally, we prove our main counting result (\Cref{thm:intro-counting}) in \Cref{subsn:proof-main-thm}.
Note that \Cref{thm:intro-local--global} and \Cref{lem:analytic-lemma} are also used in \Cref{sn:more-groups} to prove \Cref{thm:heisenberg-count}.

\subsection{Local--global principle}
\label{subsn:locglob}

Before proving \Cref{thm:intro-local--global}, we introduce some notation.
At each place $P$, the inclusion $F \hookrightarrow F_P$ induces a map
$
  \Orb{\fg}{W(F^\perf)}{W(F^\perf)}
  \to
  \Orb{\fg}{W(F_P^\perf)}{W(F_P^\perf)}
$, corresponding to the natural map $H^1(\Gamma_F, G) \to H^1(\Gamma_{F_P}, G)$ (see \iref{rmk:subfield}{rk:func-completion}).
Combining this map with the maps $\alpha_P : \Orb{\fg}{W(F_P^\perf)}{W(F_P^\perf)} \to \Orb{\fg}{\cD_P}{W(\kappa_P)}$ for all places $P$, we define a map
$
   \Orb{\fg}{W(F^\perf)}{W(F^\perf)}
   \to
   \prod_P
    \Orb{\fg}{\cD_P}{W(\kappa_P)}
$.

Consider an element $[m] \in \Orb{\fg}{W(F^\perf)}{W(F^\perf)}$, and let $K$ be the associated $G$-extension of~$F$.
For all but finitely many places $P$ of $F$, the extension $K|F$ is unramified, so the corresponding element of $\Orb{\fg}{W(F_P^\perf)}{W(F_P^\perf)}$ (the $(\fg \otimes W(F_P^\perf), \circ)$-orbit of $m$)
is mapped by $\alpha_P : \Orb{\fg}{W(F_P^\perf)}{W(F_P^\perf)} \to \Orb{\fg}{\cD_P}{W(\kappa_P)}$ to the trivial orbit $\{0\}$ (see \Cref{lem:unramified}).
We have thus described a map
\begin{equation}
  \label{eqn:def-globalpha}
  \alpha:
  \Orb{\fg}{W(F^\perf)}{W(F^\perf)}
  \longrightarrow
  \rprod
  \Orb{\fg}{\cD_P}{W(\kappa_P)}
  ,
\end{equation}
where the restricted product $\rprod
\Orb{\fg}{\cD_P}{W(\kappa_P)}$ is the following subset of $\prod_P
\Orb{\fg}{\cD_P}{W(\kappa_P)}$:
\[
  \rprod
    \Orb{\fg}{\cD_P}{W(\kappa_P)}
  \coloneqq
  \Big\{
    \big(
      [D_P]
    \big)_{P \textnormal{ place of } F}
    \ \Big\vert\ 
    [D_P] = \{0\} \textnormal{ for all but finitely many places } P
  \Big\}.
\]

\begin{lemma}
  \label{lem:direct-local--global-principle}
  For every $([D_P]) \in \rprod
  \Orb{\fg}{\cD_P}{W(\kappa_P)}$, we have:
  \begin{equation}
    \label{eqn:locglob}
    \sum_{\substack{
      [\gamma]\in H^1(\Gamma_F, G):\\
      \alpha(\orb([\gamma])) = ([D_P])
    }}
      \frac1{|\Stab_G(\gamma)|}
    \ =\ 
    \prod_P |[D_P]|,
  \end{equation}
  where $|[D_P]|$ denotes the size of the $(\fg \otimes W(\kappa_P), \circ)$-orbit of $D_P$.
\end{lemma}

\begin{proof}
  We first prove the claim in the case $\fg=\Z/p\Z$, where $G = \Z/p\Z$.
  In this case, the orbits~$[D_P]$ all have size $1$ by \iref{prop:action-on-cD}{prop:action-on-cD-ii}.
  Let $\cO_P = \kappa_P \llbracket \pi_P \rrbracket$ be the ring of integers of~$F_P$.
  We have the following exact sequence, which is established in \cite[Corollary~5.3]{potthast}:
  \begin{equation}
    \label{artin-schreier-locglob}
    0
    \;\longrightarrow\;
    \F_q/\wp(\F_q)
    \;\longrightarrow\;
    F/\wp(F)
    \;\overset{\alpha}{\longrightarrow}\;
    \bigoplus_P F_P/(\cO_P+\wp(F_P))
    \;\longrightarrow\;
    0.
  \end{equation}
  Here, we have identified the surjection $F/\wp(F) \to \bigoplus_P F_P/(\cO_P+\wp(F_P))$ with the map $\alpha$ of \Cref{eqn:def-globalpha} as follows:
  \begin{itemize}
    \item
      On the one hand, we have an identification
      \[
        F/\wp(F)
        \underset{\textnormal{\Cref{lem:artsch-perfclos}}}\simeq
        F^\perf/\wp(F^\perf)
        =
        \Orb{\fg}{W(F^\perf)}{W(F^\perf)}.
      \]
    \item
      On the other hand, we have identifications:
      \[
        F_P/\wp(F_P)
        \;\underset{\textnormal{\Cref{lem:artsch-perfclos}}}{\simeq}\;
        F_P^\perf/\wp(F_P^\perf)
        \;=\;
        \Orb{\fg}{W(F_P^\perf)}{W(F_P^\perf)}
        \;\underset{\textnormal{\Cref{thm:local-fundamental-domain}}}{\simeq}\;
        \Orb{\fg}{\cD^0_P}{W(\kappa_P)}.
      \]
      Recall that $\fg\otimes\cD^0_P = \cD^0_P/p\cD^0_P = \bigoplus_{a\in\{0\}\cup\N\setminus p\N}\kappa_P\pi_P^{-a}$, so $\Orb{\fg}{\cD^0_P}{W(\kappa_P)} = (\cD^0_P/p\cD^0_P)/\wp(\kappa_P)$.
      Similarly, $\fg\otimes\cD_P = \bigoplus_{a\in\N\setminus p\N}\kappa_P\pi_P^{-a} = \fg\otimes\cD^0_P/\kappa_P$.
      As $\cO_P\cap(\cD^0_P/p\cD^0_P) = \kappa_P$, we obtain an identification
      \[
        F_P/(\cO_P+\wp(F_P)) \simeq \Orb{\fg}{\cD_P}{W(\kappa_P)}.
      \]
      Assembling the identifications $F_P/(\cO_P + \wp(F_P)) \simeq \Orb{\fg}{\cD_P}{W(\kappa_P)}$ for all places $P$ yields the desired identification $\bigoplus_P F_P/(\cO_P + \wp(F_P)) \simeq \prod'_P \Orb{\fg}{\cD_P}{W(\kappa_P)}$.
  \end{itemize}
  Using the exact sequence (\ref{artin-schreier-locglob}), one sees that there are exactly $|\F_q/\wp(\F_q)|=p$ cosets $[m]\in F/\wp(F)$ with $\alpha([m]) = ([D_P])$.
  Since $\orb:H^1(\Gamma_F,G)\rightarrow F/\wp(F)$ is a bijection, there are exactly~$p$ classes $[\gamma]\in H^1(\Gamma_F,G)$ with $\alpha(\orb([\gamma]))=([D_P])$ for all places~$P$.
  Since $G$ is abelian, each of them has stabilizer $\Stab_G(\gamma)=G$ of size~$p$.
  The claim then follows:
  \[
    \sum_{\substack{
      [\gamma]\in H^1(\Gamma_F, G):\\
      \alpha(\orb([\gamma])) = ([D_P])
    }}
      \frac1{|\Stab_G(\gamma)|}
    =
    p \cdot \frac1p
    = 1
    = \prod_P |[D_P]|.
  \]

  We now treat the case of a general Lie algebra $\fg$.
  Let $A$ be the left-hand side of \Cref{eqn:locglob}.
  By the orbit-stabilizer theorem, we have:
  \[
    A
    \;=\;
    \frac1{|G|} \,
    \Big|\Big\{
      \gamma\in \Hom(\Gamma_F, G)
      \ \Big\vert\ 
      \alpha(\orb([\gamma])) = ([D_P])
    \Big\}\Big|.
  \]
  We show that $A=\prod_P|[D_P]|$ by induction over the size of the nonzero finite Lie algebra $\fg$.
  Pick any Lie subalgebra $\fn \subseteq Z(\fg)$ isomorphic to~$\Z/p\Z$, and let $N \coloneqq (\fn, \circ)$ be the corresponding subgroup of $Z(G)$.
  We denote the corresponding projection maps (modulo $\fn \otimes \cD_P$ or $N$) by $\boldsymbol\pi$.
  Since $\cD_P$ is a flat $\Z_p$-module, we have an exact sequence
  \[
    0 \rightarrow \fn\otimes\cD_P \rightarrow \fg\otimes\cD_P \stackrel{\boldsymbol\pi}\rightarrow (\fg/\fn)\otimes\cD_P \rightarrow 0.
  \]
  We can write
  \[
    A = \frac1{|G/N|}
      \sum_{\substack{
        \varepsilon\in\Hom(\Gamma_F,G/N):\\
        \alpha(\orb([\varepsilon])) = ([\boldsymbol\pi(D_P)])
      }}
        a(\varepsilon)
  \]
  where we have defined
  \[
    a(\varepsilon)
    \;\coloneqq\;
    \frac{1}{|N|} \,
    \Big|
      \Big\{
        \gamma\in\Hom(\Gamma_F,G)
        \ \Big\vert\ 
        \boldsymbol\pi\circ\gamma = \varepsilon
        \textnormal{ and }
        \alpha(\orb([\gamma])) = ([D_P])
      \Big\}
    \Big|.
  \]

  Let $\varepsilon \in \Hom(\Gamma_F, G/N)$ be such that $\alpha(\orb([\varepsilon]))=([\boldsymbol\pi(D_P)])$.
  The surjectivity of the natural map $\Orb{\fg}{W(F^\perf)}{W(F^\perf)} \rightarrow \Orb{(\fg/\fn)}{W(F^\perf)}{W(F^\perf)}$ implies the surjectivity of the corresponding map $H^1(\Gamma_F,G)\rightarrow H^1(\Gamma_F,G/N)$ (cf.~\iref{rmk:functoriality}{rmk:functoriality-i}), and therefore of $\Hom(\Gamma_F,G)\rightarrow\Hom(\Gamma_F,G/N)$.
  (This reflects the fact that there is no obstruction to the embedding problem for $G\twoheadrightarrow G/N$.)
  Thus, we can pick an arbitrary preimage $\gamma_0 \in\Hom(\Gamma_F,G)$ such that $\boldsymbol\pi \circ \gamma_0 = \varepsilon$.
  The entire fiber
  $\suchthat{
    \gamma \in \Hom(\Gamma_F,G)
  }{
    \boldsymbol\pi\circ\gamma = \varepsilon
  }$
  can then be described as the set of twists $\gamma_0\cdot\delta$ with $\delta\in\Hom(\Gamma_F,N)$ (cf.~\Cref{rmk:twisting}).
  We therefore have:
  \[
    a(\varepsilon) =
    \frac{1}{|N|}
    \Big|
      \Big\{
        \delta\in\Hom(\Gamma_F,N)
        \ \Big\vert\ 
        \alpha(\orb([\gamma_0\cdot\delta])) = ([D_P])
      \Big\}
    \Big|.
  \]
  Since $\alpha(\orb([\varepsilon]))=([\boldsymbol\pi(D_P)])$ and $\gamma_0$ lifts $\varepsilon$, we have $\alpha(\orb([\gamma_0]))=([U_P])\in\rprod\Orb{\fg}{\cD_P}{W(\kappa_P)}$ for some $U_P\in \fg\otimes\cD_P$ with $\boldsymbol\pi(U_P)=\boldsymbol\pi(D_P)$.
  If, for a given $\delta\in\Hom(\Gamma_F,N)$, we write $\alpha(\orb([\delta])) = ([V_P]) \in \rprod \Orb{\fn}{\cD_P}{W(\kappa_P)}$, then $\alpha(\orb([\gamma_0\cdot\delta])) = ([U_P+V_P]) \in \rprod \Orb{\fg}{\cD_P}{W(\kappa_P)}$ by \Cref{lem:twisting-orb} and \Cref{lem:twisting-alpha}.
  We obtain
  \bgroup\allowdisplaybreaks
  \begin{align*}
    a(\varepsilon)
    &=
    \sum_{\substack{
      ([V_P]) \in \rprod \Orb{\fn}{\cD_P}{W(\kappa_P)}:\\
      ([U_P + V_P]) = ([D_P])
    }}
      \underbrace{
        \frac{1}{|N|}
        \Big|
          \Big\{
            \delta\in\Hom(\Gamma_F,N)
            \ \Big\vert\ 
            \alpha(\orb([\delta])) = ([V_P])
          \Big\}
        \Big|
      }_{= 1 \textnormal{ by the base case, as } \fn \simeq \Z/p\Z}
    \\
    &=
    \cardsuchthat{
      ([V_P]) \in \rprod \Orb{\fn}{\cD_P}{W(\kappa_P)}
    }{
      ([U_P + V_P]) = ([D_P])
    }
    \\
    &=
    \prod_P
      \Big|\Big\{
        V_P \in \fn\otimes\cD_P
        \;\Big|\;
        U_P+V_P \in [D_P]
      \Big\}\Big|
    \tag*{by \iref{prop:action-on-cD}{prop:action-on-cD-ii}}
    \\
    &=
    \prod_P
      \cardsuchthat{
        x \in [D_P]
      }{
        x - U_P \in \fn \otimes \cD_P
      }
    \tag*{\textnormal{by the change of variables $x = U_P+V_P$}}
    \\
    &=
    \prod_P
      \cardsuchthat{
        x \in [D_P]
      }{
        \boldsymbol\pi(x) = \boldsymbol\pi(U_P)
      }
  \end{align*}
  \egroup
  The number of such $x$ is the size of the fiber of the map $\boldsymbol\pi: [D_P]\rightarrow[\boldsymbol\pi(D_P)]$, $x \mapsto \boldsymbol\pi(x)$ above $\boldsymbol\pi(U_P) = \boldsymbol\pi(D_P)$.
  All the fibers of that map have the same size since $[D_P]$ is a single $(\fg\otimes W(\kappa_P), \circ)$-orbit.
  Hence:
  \[
    a(\varepsilon) = \prod_P \frac{|[D_P]|}{|[\boldsymbol\pi(D_P)]|}.
  \]
  In particular, $a(\varepsilon)$ does not on depend on $\varepsilon$.
  We can finally conclude:
  \[
    A =
    \underbrace{
      \frac{1}{|G/N|}
      \Big|\Big\{
        \varepsilon\in\Hom(\Gamma_F,G/N)
        \;\Big|\;
        \alpha(\orb([\varepsilon])) = ([\boldsymbol\pi(D_P)])
      \Big\}\Big|
    }_{ = \prod_P |[\boldsymbol\pi(D_P)]| \textnormal{ by the induction hypothesis}}
    \cdot
    \prod_P
      \frac
        {|[D_P]|}
        {|[\boldsymbol\pi(D_P)]|}
    =
    \prod_P
      |[D_P]|.
    \qedhere
  \]
\end{proof}

The local--global principle for the last jump (\Cref{thm:intro-local--global}) follows readily from \Cref{lem:direct-local--global-principle}:

\begin{theorem}[Local--global principle]
  \label{thm:local--global-principle}
  Pick rational numbers $N_P\in\Q_{\geq0}$ for every place $P$ of $F=\F_q(T)$, so that $N_P=0$ for all but finitely many places $P$.
  Then,
  \[
    \sum_{\substack{
      K \in \Ext(G,F):\\
      \forall P,\ \lastjump_P(K) = N_P
    }}
      \frac{1}{|\Aut(K)|}
    = \prod_P \sum_{\substack{
      K_P \in \Ext(G,F_P):\\
      \lastjump(K_P) = N_P
    }}
      \frac{1}{|\Aut(K_P)|}.
  \]
\end{theorem}

\begin{proof}
  Using the bijection $\Ext(G,F)\simto H^1(\Gamma_F,G)$ from \Cref{lem:etext-bij-cohomology} and the fact that $\lastjump_P(K)=\lastjump(D_P)$ if $K$ corresponds to $[\gamma]$ and $\alpha_P(\orb([\gamma]))=[D_P]$ (cf.~\Cref{def:lastjump-for-orbits,def:lastjump-in-fundom,def:lastjump-cD}), we have:
  \[
    \sum_{\substack{
      K \in \Ext(G,F):\\
      \forall P,\ \lastjump_P(K) = N_P
    }}
      \frac{1}{|\Aut(K)|}
    =
    \sum_{\substack{
      ([D_P]) \in \rprod \Orb{\fg}{\cD_P}{W(\kappa_P)}:\\
      \forall P,\, \lastjump([D_P]) = N_P
    }}
      \sum_{\substack{
        [\gamma] \in H^1(\Gamma_F, G):\\
        \alpha(\orb([\gamma])) = ([D_P])
      }}
        \frac{1}{|\Stab_G(\gamma)|}
  \]
  By \Cref{lem:direct-local--global-principle}, the inner sum equals $\prod_P |[D_P]|$.
  Therefore:
  \begin{align*}
    \sum_{\substack{
      K \in \Ext(G,F):\\
      \forall P:\ \lastjump_P(K) = N_P
    }}
      \frac{1}{|\Aut(K)|}
    & =
    \prod_P
      \cardsuchthat{
        D_P \in \fg\otimes\cD_P
      }{
        \lastjump(D_P) = N_P
      }
    \\
    & =
    \prod_P
      \sum_{\substack{
        K_P \in \Ext(G,F_P):\\
        \lastjump(K_P) = N_P
      }}
        \frac{1}{|\Aut(K_P)|}
    \qquad\qquad
    \textnormal{by \Cref{lem:numberofD-numberofexts}.}
    \qedhere
  \end{align*}
\end{proof}

\begin{remark}
  In the proof of the local--global principle, we crucially made use of the fact that, for any $D \in \fg \otimes \cD^0$, the number $\lastjump(D)$ is completely determined by the projection $\pr(D)  \in \fg \otimes \cD$.
  (See \Cref{cor:irrelevance-D0}.)
  Our method of proof does not imply a similar local--global principle for the discriminant, as the discriminant in general cannot be determined from $\pr(D)$.
  For the same reason, our method of proof fails for groups of nilpotency class larger than $2$.
\end{remark}

\subsection{Analytic lemma}
\label{subsn:analytic}

We now prove a general analytic lemma that allows us to combine local estimates into global asymptotics when there is a local--global principle.
We write the lemma and its proof in such a way that it is valid for any function field $F$ of characteristic~$p$ with field of constants $\F_q$, without having to assume that $F = \F_q(T)$.

\begin{lemma}
  \label{lem:analytic-lemma}
  Let $K\geq1$ be an integer.
  For every place $P$ of $F$ and every $n\in\frac1K\Z_{\geq0}$, let $a_{P,n}\in\R_{\geq 0}$, with $a_{P,0}=1$.
  For every $n \in \frac1K \Z_{>0}$, let $b_n\in\Z_{\geq0}$ and $e_n,k_n\in\R$ be such that for all places~$P$, we have
  \[
    a_{P,n} = b_n|\kappa_P|^{e_n} + \cO(|\kappa_P|^{k_n}),
  \]
  where the implied constant factor in the $\cO$-estimate depends neither on the place $P$ nor on the number $n$.
  Assume that $b_n=0$ for all but finitely many numbers $n$, but $b_n\neq0$ for at least one number $n$.
  Now, define
  \[
    A \coloneq \max\left\{\tfrac{e_n+1}{n} \mid n\in\tfrac1K\Z_{>0}\textnormal{ with }b_n\neq0\right\}
  \]
  and assume that
  \begin{equation}
    \label{eqn:hypothesis-on-sup}
    \sup\{\tfrac{k_n+1}{n}\mid n\in\tfrac1K\Z_{>0}\} < A.
  \end{equation}
  Let
  \[
    S \coloneq
    \{ n\in\tfrac1K\Z_{>0} \mid b_n\neq0\textnormal{ and } A = \tfrac{e_n+1}{n} \},
  \]
  define
  \begin{equation}
    \label{eqn:def-B}
    B \coloneq \sum_{n\in S} b_n,
  \end{equation}
  and let $M\in\frac1K\Z_{>0}$ be the least common integer multiple of the numbers $n\in S$, so that $M\Z = \bigcap_{n \in S} n\Z$.
  Finally, define:
  \[
    a_N
    \coloneqq
    \sum_{\substack{(n_P)_P\in\prod_P\frac1K\Z_{\geq0}:\\\sum_P n_P\deg(P) = N}}
      \prod_P a_{P,n_P}.
  \]
  Then, there is a function $C:\Q/M\Z\rightarrow\R_{\geq0}$ with $C(0)>0$, such that for rational $N$ going to infinity, we have
  \[
    a_N = C(N\bmod M)\cdot q^{AN} N^{B-1} + o(q^{AN}N^{B-1}).
  \]
\end{lemma}

\begin{proof}
  Rescaling by a factor of $K$, we can assume without loss of generality that $K=1$.
  We then only need to consider integers $N$ since $a_N=0$ for $N\notin\Z_{\geq0}$.
  Moreover, we can assume without loss of generality that $b_n=0$ for all $n\notin S$.
  (For the finitely many $n \not\in S$ with $b_n\neq0$, replace $b_n$ by $0$ and $k_n$ by $\max(e_n,k_n)$, and observe that Inequality~(\ref{eqn:hypothesis-on-sup}) still holds.)
  Let $\delta\coloneq A - \sup\{\frac{k_n+1}{n}\mid n\in\Z_{>0}\} > 0$.
  By definition of the numbers $a_N$, the generating function $f(X) \coloneqq \sum_{N \geq 0} a_N X^N$ factors as an Euler product
  \[
    f(X) = \prod_P f_P(X^{\deg(P)}),
    \qquad\qquad
    \textnormal{where}
    \qquad
    f_P(X) \coloneq \sum_{n\geq0} a_{P,n} X^n.
  \]
  We estimate the local factor at a place $P$, for a complex number $X$:
  \begin{align*}
    f_P(|\kappa_P|^{-A}X)
    &= 1 + \sum_{n\geq1} a_{P,n} |\kappa_P|^{-An}X^n
    \\
    &= 1 + \sum_{n\in S} b_n|\kappa_P|^{e_n-An}X^n
    + \sum_{n\geq1} \cO(|\kappa_P|^{k_n-An}X^n) \\
    &= 1 + \sum_{n\in S} b_n|\kappa_P|^{-1}X^n
    + \sum_{n\geq1} \cO(|\kappa_P|^{-1-\delta n}X^n).
  \end{align*}
  Recall that the constant in the $\cO$-estimate is independent of both $n$ and $P$.
  We obtain:
  \begin{align*}
    f_P(|\kappa_P|^{-A}X)
    &= 1 + \sum_{n\in S} b_n|\kappa_P|^{-1}X^n
    + \cO\!\left(
      \sum_{n\geq1} |\kappa_P|^{-1-\delta n}X^n
    \right)
    \\
    &= 1 + |\kappa_P|^{-1} \sum_{n\in S} b_nX^n
    + |\kappa_P|^{-1-\delta} X \cdot
    \cO\!\left(
      \sum_{n\geq0}
        \left(|\kappa_P|^{-\delta}X\right)^n
    \right)
  \end{align*}
  For $|X|<|\kappa_P|^{\delta/2}$, we have $|\kappa_P|^{-\delta}|X| \leq |\kappa_P|^{-\delta/2} \leq q^{-\delta/2}$, and then
  \[
    \left|
      \sum_{n\geq0}
        \left(|\kappa_P|^{-\delta}X\right)^n
    \right|
    \leq
    \sum_{n \geq 0}
      q^{-n\delta/2}
    =
    \frac1{1-q^{-\delta/2}},
  \]
  which means that the $\cO$-factor is bounded by a constant independent of $P$.
  Note also that $|\kappa_P|^{-1-\delta} |X| \leq |\kappa_P|^{-1-\delta/2}$.
  We have obtained the estimate
  \[
    f_P(|\kappa_P|^{-A}X)
    = 1 + |\kappa_P|^{-1}\sum_{n\in S}b_nX^n
    + \cO(|\kappa_P|^{-1-\delta/2}),
  \]
  which implies, under the additional assumption that $|X|<|\kappa_P|^{1/3n}$ for all $n\in S$, that
  \[
    f_P(|\kappa_P|^{-A}X)
    \prod_{n\in S}\left(1-|\kappa_P|^{-1}X^n\right)^{b_n}
    = 1 + \cO(|\kappa_P|^{-1-\varepsilon})
  \]
  for $\varepsilon \coloneq \min(\frac13, \frac\delta2)$.
  Let $\delta' \coloneq \min\!\Big(\big\{\frac\delta2\big\}\cup\big\{\frac1{3n}\mid n\in S\big\}\Big)$.
  For $|X|<q^{\delta'}$, we obtain
  \begin{equation}
    \label{estimate-product}
    f(q^{-A}X)
    =
    \prod_P
      f_P\!\left(
        |\kappa_P|^{-A}X^{\deg(P)}
      \right)
    = \prod_P \frac{
      1 + \cO(|\kappa_P|^{-1-\varepsilon})
    }{
      \prod_{n\in S} \left(1 - (q^{-1}X^n)^{\deg(P)}\right)^{b_n}
    }.
  \end{equation}
  The product of the numerators is absolutely convergent and hence defines a holomorphic function for $|X|<q^{\delta'}$.
  Moreover, that product does not vanish for $X=1$ as none of the factors $f_P(|\kappa_P|^{-A}) = 1 + \sum_{n\geq1} a_{P,n} |\kappa_P|^{-An} \geq 1$ vanish (the coefficients $a_{P,n}$ are nonnegative).
  The Hasse--Weil zeta function
  \[
    Z_F(X) = \prod_P \frac1{1-X^{\deg(P)}}
  \]
  is holomorphic for $|X|\leq\frac1q$ except for a simple pole at $X=\frac1q$.
  Using this function, \Cref{estimate-product}, rewrites as
  \[
    f(q^{-A}X)
    \;=\;
    \underbrace{
      \prod_P
        \Big(
          1 + \cO(|\kappa_P|^{-1-\varepsilon})
        \Big)
    }_{\substack{
      \textnormal{holomorphic for $|X|<q^{\delta'}$}\\
      \textnormal{non-vanishing at $X=1$}
    }}
    \,\cdot\,
    \prod_{n \in S}
      \,
      \underbrace{
        Z_F(q^{-1}X^n)^{b_n}.
      }_{\substack{
        \textnormal{holomorphic for $|X| \leq 1$, except for}\\
        \textnormal{poles of order $b_n$ at $n$-th roots of $1$}\\
      }}
  \]
  Thus, the function $f(q^{-A}X)\prod_{n\in S}(1-X^n)^{b_n}$ is holomorphic for $|X|\leq1$ and non-vanishing at $X=1$.
  It follows that $f(X)$ is holomorphic for $|X|\leq q^{-A}$ except for poles of order at most $B$ at $q^{-A},q^{-A}\cdot\zeta,\dots,q^{-A}\cdot\zeta^{M-1}$ for the $M$-th roots of unity $1,\zeta,\dots,\zeta^{M-1}$, and the pole at $q^{-A}$ has order exactly $B$.
  By \cite[Theorem~IV.10]{flajolet-sedgewick-analytic-combinatorics}, this implies that
  \[
    a_N
    = \sum_{i=0}^{M-1} D_i (q^{-A}\zeta^i)^{-N}N^{B-1} + o(q^{AN}N^{B-1})
  \]
  for some constants $D_0,\dots,D_{M-1}$, with $D_0\neq0$.
  We conclude that
  \[
    a_N = C(N) q^{AN} N^{B-1} + o(q^{AN}N^{B-1}),
  \]
  where
  $
    C(N) \coloneq \sum_{i=0}^{M-1} D_i \zeta^{-iN}
  $
  only depends on the remainder of $N$ modulo $M$.
  Since $D_0\neq0$, we have $C(N)\neq0$ for some $N$.
  Since $a_N\geq0$ for all $N$, we have $C(N)\geq0$ for all $N$.
  
  \medskip

  We show $C(0)>0$ as follows:
  Let $\nu$ be the greatest common divisor of the numbers $n\in S$.
  Let $L\geq0$ be large enough so that every residue class modulo $M$ which is divisible by $\nu$ can be written as the sum of at most $L$ (not necessarily distinct) elements of $S$.
  
  Inequality (\ref{eqn:hypothesis-on-sup}) implies $e_n > k_n$ for $n\in S$, so $a_{P,n}>0$ for all but finitely many places $P$.
  By the Weil bound, for any sufficiently large integer $d$, there is a place of degree $d$.
  In particular, there are infinitely many places with $\deg(P)\equiv1\mod M$.
  Pick any $L$ places $Q_1,\dots,Q_L$ with $\deg(Q_i)\equiv1\mod M$ and $a_{Q_i,n}>0$ for all $n\in S$, and let $\Omega$ be the set of remaining places $P\neq Q_1,\dots,Q_L$.
  Restricting to places in $\Omega$ and elements $n_P$ of $S$, define
  \[
    a_N'
    \coloneqq
    \sum_{\substack{(n_P)_P\in\prod_{P\in\Omega}S:\\\sum_{P\in\Omega} n_P\deg(P) = N}}
      \prod_{P\in\Omega} a_{P,n_P}.
  \]
  We have $a_N' = C'(N\bmod M)\cdot q^{AN}N^{B-1}+o(q^{AN}N^{B-1})$ by the same argument as above, again with $C'(r)>0$ for some residue class $r$ modulo $M$.
  By definition, $a_N'=0$ unless $N$ is a sum of elements of $S$, and hence divisible by $\nu$.
  Thus, $C'(r)>0$ implies that $r$ is divisible by~$\nu$.
  Now, write $-r\equiv n_{Q_1}+\cdots+n_{Q_L}\mod M$ with $n_{Q_1},\dots,n_{Q_L}\in\{0\}\cup S$.
  The inequality
  \[
    a_{n_{Q_1}\deg(Q_1)+\cdots+n_{Q_L}\deg(Q_L) + N}
    \geq a_{Q_1,n_1}\cdots a_{Q_L,n_L} a_N'
  \]
  for $N\equiv r \mod M$ implies that
  $
    C(0) \geq a_{Q_1,n_1}\cdots a_{Q_L,n_L} C'(r) > 0
  $.
\end{proof}

\begin{remark}
  Note the similarities between our \Cref{eqn:hypothesis-on-sup,eqn:def-B} and \cite[Equations~(8.1) and (8.2)]{darda-yasuda-wild}.
  In that regard, our local--global principle (\Cref{thm:local--global-principle}) seems to be compatible with Darda and Yasuda's ``Main Speculation''.
\end{remark}

\subsection{Main counting result}
\label{subsn:proof-main-thm}

\begin{theorem}
  \label{thm:proof-counting}
  Let
  \[
    r \coloneqq \log_p|G[p]|
    \andd
    M \coloneqq
    \begin{cases}
      1 & \textnormal{if } G[p]\textnormal{ is abelian},\\
      1+p^{-1} & \textnormal{otherwise}.\\
    \end{cases}
  \]
  If $G[p]$ is non-abelian, assume that $|G[p]|\leq p^{p-1}$.
  Assume moreover that $q$ is a large enough power of~$p$ (depending on the group $G$).
  Then, there is a function
  $
    C:\Q/M\Z\rightarrow\R_{\geq0}
  $
  with $C(0) \neq 0$, such that for rational $N\rightarrow\infty$, we have
  \[
    \sum_{\substack{
      K\in\Ext(G,F):\\
      \lastjump(K)=N
    }}
      \frac1{|\Aut(K)|}
    \quad=\quad
    C\!\left(N\bmod M\right)\cdot
    q^{\frac{r+1}{M}\cdot N}
    +
    o\!\left(
      q^{\frac{r+1}{M}\cdot N}
    \right).
  \]
\end{theorem}

\begin{proof}
For every $N \in \Q_{\geq0}$, define:
\[
  a_N
  \coloneqq
  \sum_{\substack{
    K\in\Ext(G,F):\\
    \lastjump(K)=N
  }}
    \frac1{\Aut(K)}.
\]
Combining \Cref{eq:def-global-lastjump}, \Cref{thm:local--global-principle}, and \Cref{lem:numberofD-numberofexts}, we see that
\[
  a_N
  =
  \sum_{\substack{
    (n_P)_P\in\prod_P\frac1{|\fg|}\Z_{\geq0}:\\
    \sum_P n_P\deg(P) = N
  }}
    \prod_P
      a_{P,n_P}
  \qquad\qquad
  \textnormal{where}
  \qquad
  a_{P,n}
  \coloneqq
  \cardsuchthat{
    D\in\fg\otimes\cD_P
  }{
    \lastjump(D)=n
  }.
\]
Our goal is now to use \Cref{lem:analytic-lemma} (with $K = |\fg|$) in order to obtain estimates for $a_N$.
For this, we use the following estimates arising from \Cref{thm:local-counting}:

\bgroup
\allowdisplaybreaks
\textbf{If $\fg[p]$ is non-abelian, then:}
\begin{align*}
  a_{P,n} &= 1 && \textnormal{for }n=0,\\
  a_{P,n} &= 0 && \textnormal{for }0<n<1,\\
  a_{P,n} &= |\kappa_P|^r + \cO(|\kappa_P|^{r-1}) && \textnormal{for }n=1+p^{-1},\\
  a_{P,n} &= 0 && \textnormal{for }1+p^{-1}<n<2, \\
  a_{P,n} &= \cO(E^l|\kappa_P|^{l r-1}) && \textnormal{for }l\leq n<l+l p^{-1}\textnormal{ with }l=1,\dots,p-1, \\
  a_{P,n} &= \cO(E^l|\kappa_P|^{l r-1}) && \textnormal{for }l\leq n<l+1\textnormal{ with }l=p,p+1,\dots, \\
  a_{P,n} &= \cO(E^l|\kappa_P|^{l r}) && \textnormal{for }l+l p^{-1}\leq n<l+1\textnormal{ with }l=2,\dots,p-1.
\end{align*}

\textbf{If $\fg[p]$ is abelian, then:}
\begin{align*}
  a_{P,n} &= 1 && \textnormal{for }n=0, \\
  a_{P,n} &= 0 && \textnormal{for }0<n<1, \\
  a_{P,n} &= |\kappa_P|^r + \cO(1) && \textnormal{for }n=1, \\
  a_{P,n} &= 0 && \textnormal{for }1<n<2, \\
  a_{P,n} &= \cO(E^l|\kappa_P|^{l r}) && \textnormal{for }l\leq n<l+1\textnormal{ with }l=2,3,\dots.
\end{align*}
\egroup
We now fix a real $\varepsilon>0$ satisfying $\varepsilon < \frac12$ and, if $\fg[p]$ is non-abelian (in which case we have assumed $r \leq p-1$), also satisfying $\varepsilon < \frac{p-r}{p+1}$.
We assume that $q \geq E^{1/\varepsilon}$, so that  $E \leq |\kappa_P|^\varepsilon$ for all places $P$ (note that any $q$ works if $E=1$).

\textbf{If $\fg[p]$ is non-abelian:}
The estimates above show that we can take the following values of~$b_n, e_n, k_n$ in order to apply \Cref{lem:analytic-lemma}:
\[\begin{tabular}{c|*{5}{@{\hspace{.3cm}}c}}
  $n$ & \begin{tabular}{c} $0<n<1$ \\ or $1+ p^{-1} < n <2$ \end{tabular} & $1+ p^{-1}$ & \begin{tabular}{c}$l+lp^{-1} \leq n < l+1$\\ with $2 \leq l < p$ \end{tabular} & \begin{tabular}{c} $l \leq n < l+lp^{-1}$ with $1 \leq l < p$ \\ or $l \leq n < l+1$ with $l \geq p$ \end{tabular} \\
  \hline
  $b_n$ & $0$       & $1$   & $0$                 & $0$ \\
  $e_n$ & $0$       & $r$   & $0$                 & $0$ \\
  $k_n$ & $-\infty$ & $r-1$ & $l\varepsilon + lr$ & $l\varepsilon + lr - 1$
\end{tabular}\]
In the notation of \Cref{lem:analytic-lemma}, we have $S = \{1+p^{-1}\}$, $A = \frac{r+1}{1+p^{-1}}$, $B = 1$ and $M = 1+p^{-1}$.
We verify Inequality~(\ref{eqn:hypothesis-on-sup}):
\begin{align*}
  &\frac{l\varepsilon+lr+1}{l+lp^{-1}} \leq \frac{\varepsilon+r+1/2}{1+p^{-1}} < \frac{r+1}{1+p^{-1}} = A
  &&\textnormal{for }2\leq l < p, \\
  &\frac{l\varepsilon+lr}{l} = \varepsilon+r < \frac{p-r}{p+1} + r = \frac{r+1}{1+p^{-1}} = A
  &&\textnormal{for }l\geq1.
\end{align*}

\textbf{If $\fg[p]$ is abelian:}
The estimates above show that we can take the following values of~$b_n, e_n, k_n$ in order to apply \Cref{lem:analytic-lemma}:
\[\begin{tabular}{c|*{4}{@{\hspace{.8cm}}c}}
  $n$ & $0<n<1$ & $1$ & $1<n<2$ & $l \leq n < l+1$ with $l \geq 2$\\
  \hline
  $b_n$ & $0$ & $1$ & $0$ & $0$  \\
  $e_n$ & $0$ & $r$ & $0$ & $0$  \\
  $k_n$ & $-\infty$ & $0$ & $-\infty$ & $l\varepsilon + lr$
\end{tabular}\]
We have $S = \{1\}$, $A = r+1$, $B = 1$ and $M = 1$.
We verify Inequality~(\ref{eqn:hypothesis-on-sup}):
\begin{align*}
  &\frac{l\varepsilon+lr+1}{l} \leq \varepsilon+r+\frac12 < r+1 = A
  &&\textnormal{for }l\geq2.
  \qedhere
\end{align*}
\end{proof}

\section{Groups of exponent~$p$: the example of Heisenberg groups}
\label{sn:more-groups}

When $\fg$ has exponent~$p$, our main result (\Cref{thm:intro-counting}) only applies to finitely many $p$-groups due to the assumption that $|\fg|=|\fg[p]|\leq p^{p-1}$.
To illustrate how this restriction may be overcome, we are going to deal with the infinite family of Heisenberg groups $H_k(\F_p)$ of exponent~$p$.

The section is organized as follows:
in \Cref{subsn:nilpas-expp,subsn:lastjump-in-expp}, we review the simpler form taken by nilpotent Artin--Schreier theory and by the equations of \Cref{def:property-Jv} when focusing on groups of exponent~$p$.
In particular, we show how a stronger version of \Cref{cor:counting-bounded-last-jump} follows immediately from the equations (\Cref{prop:large-v-expp}).
We then introduce generalized Heisenberg groups in \Cref{subsn:heisenberg-groups}, we obtain estimates for the number of elements of $H_k(\F_q)$ which commute with (part of) their Frobenius orbit in \Cref{subsn:elements-commuting-with-frob}, and we use these estimates to obtain asymptotics for the number of $H_k(\F_p)$-extensions of function fields in \Cref{subsn:local-counting-heisenberg,subsn:heisenberg-global-asymptotics}, proving our main theorem (\Cref{thm:heisenberg-count}, which is \Cref{thm:intro-heisenberg}).

\subsection{Nilpotent Artin--Schreier theory in exponent~$p$}
\label{subsn:nilpas-expp}

When we focus only on $p$-groups~$G$ of exponent~$p$ and of nilpotency class at most $2$ (corresponding to a Lie $\F_p$-algebra $\fg$), we do not need Witt vectors at all, as we shall briefly explain.

Let $\fg \neq 0$ be a finite Lie $\F_p$-algebra of nilpotency class $\leq 2$.
For every field of characteristic~$p$, we have a bijection $\orb: H^1(\Gamma_F, G) \simto \Orb{\fg}{F}{F}$, where $\Orb{\fg}{F}{F}$ is the set of orbits of $\fg \otimes_{\F_p} F$ under the $(\fg \otimes F, \circ)$-action given by $g.m \coloneqq \sigma(g) \circ m \circ (-g)$.
In other words, compared to previous sections, we can ignore every coordinate of every Witt vector besides the first, and we do not have to take the perfect closure of the base field (see \Cref{lem:artsch-perfclos}).
All proofs are similar to the proofs of \Cref{sn:preliminaries}, but are largely simplified.

When we have fixed a finite field $\kappa$, we let $\cD^0$ (resp.~$\cD$) be the $\kappa$-linear subspace of~$\kappa\llpar\pi\rrpar$ spanned by the elements $\pi^{-a}$ for $a \in \Onotp$ (resp.~for $a \in \notp$).
Then, the obvious analogues of all results from \Cref{sn:local} hold, with simplified proofs.
In particular, we have a bijection $\alpha^0 : \Orb{\fg}{F}{F} \simto \Orb{\fg}{\cD^0}{\kappa}$ and a $|\fg \otimes \kappa|$-to-one surjection $\alpha : \Orb{\fg}{F}{F} \simto \Orb{\fg}{\cD}{\kappa}$.

\subsection{The distribution of last jumps in exponent~$p$}
\label{subsn:lastjump-in-expp}

Let as above $\fg$ be a nonzero finite Lie $\F_p$-algebra of nilpotency class $\leq 2$ and let $r\coloneq\dim_{\F_p}\fg$.
By \Cref{thm:lastjump-Jv}, the condition that an element $D \in \fg \otimes \cD$ satisfies $\lastjump(D) < v$ is equivalent to the equations given in \Cref{def:property-Jv}.
Moreover, in this setting, these equations take a much simpler form:

\begin{corollary}
  \label{cor:jv-in-exp-p}
  Consider an element $D \in \fg \otimes \cD$, written as $D = \sum_{b \in \notp} D_b \pi^{-b}$ with $D_b \in \fg \otimes \kappa$.
  Then, $D$ satisfies $\lastjump(D) < v$ if and only if the following equations hold, for all $b \in \N\setminus p \N$:
  \begin{align}
    \label{ex:jv-in-exp-p-i}
    D_b
    &=
    -(2b)^{-1}\sum_{\substack{
      a_1,a_2\in\notp:\\
      b = a_1 + a_2,\\
      a_1<v,\ a_2 < v
    }}
    a_1 [D_{a_1}, D_{a_2}]
    &
    \textnormal{
      if $b\geq v$
    }
    \\
    \label{ex:jv-in-exp-p-ii}
    0
    &=
    \sum_{\substack{
      a_1,a_2\in\notp:\\
      b = a_1 + a_2
    }}
    a_1 [D_{a_1}, D_{a_2}]
    &
    \textnormal{
      if $b<v$
    }
    \\
    \label{ex:jv-in-exp-p-iii}
    0
    &=
    \sum_{\substack{
      a_1,a_2\in\notp:\\
      b=a_1 p^i + a_2,\\
      a_1<v,\ a_2<v
    }}
      a_1
      [\sigma^i(D_{a_1}),D_{a_2}]
    &
    \textnormal{
      for all $i > 0$ such that $bp^{-i} \geq v$.
    }
  \end{align}
\end{corollary}

This simpler form of the equations can be used to directly obtain a stronger form of \Cref{cor:counting-bounded-last-jump}/\Cref{prop:better-bound} when the exponent is~$p$ (with a much simpler proof):

\begin{proposition}
  \label{prop:large-v-expp}
  For all $v > 0$, we have the upper bound:
  \begin{align*}
    \cardsuchthat{
      D \in \fg \otimes \cD
    }{
      \lastjump(D) < v
    }
    &
    =
    \cardsuchthat{
      (D_a)_{p \nmid a < v} \in (\fg \otimes \kappa)^{\lceil v \rceil - \lceil \frac vp \rceil}
    }{
      \textnormal{\Cref{ex:jv-in-exp-p-ii,ex:jv-in-exp-p-iii} hold}
    }
    \\
    & \leq
    |\fg \otimes \kappa|^{
      \lceil v \rceil - \lceil \frac v p \rceil
    }
    =
    |\kappa|^{
      r
      \left(
        \lceil v \rceil - \lceil \frac v p \rceil
      \right)
    }.
  \end{align*}
\end{proposition}

\begin{proof}
  The elements $D_b$ for $b \geq v$ are uniquely expressed in terms of the elements~$D_b$ for~$b < v$ using \Cref{ex:jv-in-exp-p-i}, and they never appear in \Cref{ex:jv-in-exp-p-ii,ex:jv-in-exp-p-iii}.
  This implies the first equality.
  The following inequalities are clear.
\end{proof}

In order to apply \Cref{lem:analytic-lemma} without assuming that $r$ is small, we need to better understand the distribution of extensions with small last jump, and hence study the sets $A_m(\kappa)$ appearing in \Cref{rk:amk} (as $\fg$ has exponent~$p$, we have $\fg[p] = \fg$ and $\fg \otimes_{\Z_p} W(\kappa) = \fg \otimes_{\F_p} \kappa$).
Since the sizes of these sets heavily depend on the Lie algebra (we are essentially counting elements commuting with their Frobenii), a unified statement for the asymptotics seems out of reach.%
\footnote{
  The main dependence in the Lie algebra $\fg$ certainly comes from its ``commuting variety'', i.e., the subvariety $C \subseteq (\mathbb A^r_{\F_p})^2$ cut out by the equations corresponding to $[x,y]=0$.
  Indeed, for example, in the case of $A_1(\F_q)$, we are counting $\F_q$-points of the intersection of $C$ with the graph of the Frobenius map.
  Interpreting this as an intersection number, and using the Grothendieck--Lefschetz trace formula (following the methods of \cite{langweiltordu} or of \cite{hils2024langweiltypeestimatesfinite}) might relate this number to the geometry of $C$.
}

\subsection{Higher Heisenberg groups and their Lie algebras}
\label{subsn:heisenberg-groups}

Let $p\neq2$ and $k \geq 1$.
For us, the Heisenberg group $H_k(\F_p)$ is the group defined by the following matrix representation%
\footnote{
  Be aware that this generalization of the Heisenberg group is distinct from the generalization considered in \cite{muller-thesis} (where the rank of the center is not always $1$).
}:
\begin{equation}
  \label{def:heisenberg-group}
  H_k(\F_p)
  \coloneqq
  \suchthat{
    \begin{pmatrix}
      1 & \vec a^t & c \\
      0 & I_k      & \vec b \\
      0 & 0        & 1
    \end{pmatrix}
    \in
    \GL_{k+2}(\F_p)
  }{
    \vec a, \vec b \in \F_p^k, \,
    c \in \F_p
  }.
\end{equation}
Equivalently, it is the unique group (up to isomorphism) with exponent~$p$, center $\Z/p\Z$, nilpotency class $2$, and order $p^{2k+1}$.
We denote the corresponding Lie $\F_p$-algebra by $\fh_k$, which as an $\F_p$-vector space can be decomposed as $(\F_p^k)^2 \oplus \F_p$ (the $\F_p$-factor is the center, corresponding to the coefficient $c$, and the two $\F_p^k$-factors correspond respectively to $\vec a$ and $\vec b$).
The Lie bracket of~$\fh_k$ is determined by the induced map $(\fh_k/Z(\fh_k))^2 \to Z(\fh_k)$, which is a nondegenerate $\F_p$-bilinear alternating form $f_k : (\F_p^k)^2 \to \F_p$.
Concretely, if $(\vec a, \vec b)$ and $(\vec \alpha, \vec \beta)$ are two elements of~$(\F_p^k)^2$, one can check that
\[
  f_k\!\left(
    (\vec a, \vec b), \,
    (\vec {a'}, \vec {b'})
  \right)
  \quad=\quad
  \vec a \cdot \vec {b'} - \vec b \cdot \vec {a'}
  \quad=\quad
  \sum_{i=1}^k
    (a_i b'_i - b_i a'_i).
\]
($f_k$ is the only non-degenerate bilinear alternating form up to automorphisms of $(\F_p^k)^2$.)

\subsection{Elements commuting with their Frobenii.}
\label{subsn:elements-commuting-with-frob}

For any finite field $\F_q$ of characteristic~$p$ and any $m\geq0$, extend the bilinear form $f_k$ to $\F_q^{2k}$ and define the set
\begin{equation}
  \label{eqn:def-anm}
  A_{k,m}(\F_q)
  \coloneqq
  \suchthat{
    x \in \F_q^{2k}
  }{
    f_k(\sigma^i(x),x) = 0
    \textnormal{ for } i\in\{1,\dots,m\}
  }.
\end{equation}
Clearly, $A_{k,0}(\F_q) \supseteq A_{k,1}(\F_q) \supseteq A_{k,2}(\F_q) \supseteq \cdots$.
The definition of $A_{k,m}(\F_q)$ is motivated by its appearance in \Cref{rk:amk}, which implies the following equality when the base field is $\F_q\llpar\pi\rrpar$:
\begin{align}
  \cardsuchthat{
    D \in \fh_k \otimes \cD
  }{
    \lastjump(D) < 1+p^{-m}
  }
  & =
  |Z(\fh_k) \otimes \F_q|
  \cdot
  |A_{k,m}(\F_q)|
  \nonumber
  \\
  & =
  q
  \cdot
  |A_{k,m}(\F_q)|.
  \label{eqn:count-anm}
\end{align}

We say that a subspace $W$ of $\F_q^{2k}$ is \emph{isotropic} if $f_k(x,y)=0$ for all $x,y\in W$.
Note that all maximal isotropic subspaces of $\F_q^{2k}$ are $k$-dimensional.

\begin{lemma}[Large $m$]
  \label{lem:Anm-large-m}
  We have
  \[
    A_{k,m}(\F_q)
    = \bigcup_{\substack{
      W\subseteq\F_p^{2k}\\
      \textnormal{isotropic}
    }}
      (W \otimes_{\F_p} \F_q)
    \qquad\textnormal{for }m \geq k.
  \]
\end{lemma}

\begin{proof}
  ~
  \begin{itemize}
    \item[$(\supseteq)$]
      Let $W\subseteq\F_p^{2k}$ be isotropic and let $x\in W\otimes\F_q$.
      Then, $\sigma^i(x) \in W\otimes\F_q$ for all $i\in\Z$.
      Hence, $f_k(\sigma^i(x),x)=0$ for all $i\in\Z$.
    \item[$(\subseteq)$]
      Let $x\in A_{k,m}(\F_q)$.
      We consider the vector space $W'$ over $\F_q$ generated by $x,\sigma(x),\dots,\sigma^k(x)$.
      By assumption, $f_k(\sigma^j(x),\sigma^i(x)) = \sigma^i(f_k(\sigma^{j-i}(x),x)) = 0$ for all $0\leq i<j\leq k$
      , so $W'$ is isotropic.
      The maximal isotropic subspaces of $\F_q^{2k}$ are $k$-dimensional, so $\dim_{\F_q}(W') \leq k$.
      Hence, the $k+1$ vectors $x,\sigma(x),\dots,\sigma^k(x)$ are linearly dependent over $\F_q$, so $\sigma(W') = W'$.
      By Galois descent for vector spaces, $W'=W\otimes\F_q$ for some isotropic $\F_p$-subspace~$W = (W')^\sigma$ of $\F_p^{2k}$.
      The claim follows as $x$ lies in $W' = W\otimes\F_q$.
      \qedhere
  \end{itemize}
\end{proof}

\begin{lemma}[Counting maximal isotropic subspaces]
  \label{lem:count-maximal-isotropic}
  There are exactly $\prod_{i=1}^k (p^i + 1)$ maximal (i.e., $k$-dimensional) isotropic subspaces of $\F_p^{2k}$.
\end{lemma}

\begin{proof}
  We count tuples $(b_1,\dots,b_k)$ of linearly independent vectors such that $f_k(b_i,b_j) = 0$ for all $1\leq i,j\leq k$.
  After choosing linearly independent vectors $b_1,\dots,b_i$, there are exactly~$p^{2k-i}$ vectors $b_{i+1}$ such that $f_k(b_{i+1}, b_j)$ for $j=1,\dots,i$, as the orthogonal complement of the $i$-dimensional space $\langle b_1,\dots,b_i\rangle$ has dimension $2k-i$.
  Exactly $p^i$ of these vectors $b_{i+1}$ lie in $\langle b_1,\dots,b_i\rangle$, so that there are $p^{2k-i}-p^i$ choices for $b_{i+1}$.
  Hence, the number of tuples $(b_1,\dots,b_k)$ as above is
  \[
    \prod_{i=0}^{k-1}
      (p^{2k-i}-p^i)
    =
    \prod_{i=0}^{k-1}
      p^i(p^{2(k-i)}-1).
  \]
  Each maximal isotropic subspace $W$, being $k$-dimensional, has exactly
  \[
    |\GL_k(\F_p)|
    =
    \prod_{i=0}^{k-1}
      (p^k - p^i)
    =
    \prod_{i=0}^{k-1}
      p^i (p^{k-i} - 1)
  \]
  bases $(b_1,\dots,b_k)$.
  Hence, the number of maximal isotropic subspaces is
  \[
    \prod_{i=0}^{k-1}
      \frac
        {p^i(p^{2(k-i)}-1)}
        {p^i(p^{k-i}-1)}
    =
    \prod_{i=0}^{k-1}
      (p^{k-i}+1)
    =
    \prod_{i=1}^k
      (p^i+1).
    \qedhere
  \]
\end{proof}

\begin{lemma}[Small $m$]
  \label{lem:small-m-counting}
  We have $|A_{k,m}(\F_q)| = q^{2k-m} + \cO_k(q^k)$ for $0\leq m<k$.
\end{lemma}

\begin{proof}
  We handle the conditions $f_k(\sigma^i(x),x)=0$ for $i=1,\dots,m$ using the following sum over the $q^m$ characters $\chi$ of the (additive) group $\F_q^m$:
  \[
    |A_{k,m}(\F_q)|
    =
    \frac{1}{q^m}
    \sum_\chi
    \sum_{x\in(\F_q^{k})^2}
      \chi\Big(
        f_k(\sigma(x),x), \,
        \dots, \,
        f_k(\sigma^m(x),x)
      \Big).
  \]
  Writing $x = (\vec a, \vec b)$ with $\vec a, \vec b \in \F_q^k$, we have by definition $f_k(\sigma^i(x),x) = \sigma^i(\vec a) \cdot \vec b - \sigma^i(\vec b) \cdot \vec a = \sum_{j=1}^k \big( \sigma^i(a_j) b_j - \sigma^i(b_j) a_j \big)$.
  We can then factor the inner sum to obtain the following expression:
  \[
    |A_{k,m}(\F_q)|
    =
    \frac{1}{q^m}
    \sum_\chi
      \left(
        \sum_{a,b\in\F_q}
          \chi\Big(
            \sigma(a)b-\sigma(b)a, \,
            \dots, \,
            \sigma^m(a)b-\sigma^m(b)a
          \Big)
      \right)^k.
  \]
  Using the non-degenerate trace form $(x,y)\mapsto\Tr_{\F_q|\F_p}(xy)$, we can identify the characters $\chi$ of~$\F_q^m$ with vectors $t\in\F_q^m$, so that
  \[
    |A_{k,m}(\F_q)|
    = \frac{1}{q^m} \sum_{t\in\F_q^m} F(t)^k,
  \]
  where
  \[
    F(t)
    \coloneqq
    \sum_{a,b\in\F_q}
      e_p\!\left(
        \sum_{i=1}^m
        \Tr\!\Big(
          t_i
          \big(
            \sigma^i(a)b-\sigma^i(b)a
          \big)
        \Big)
      \right)
  \]
  with $e_p(x \bmod p) \coloneqq e^{2\pi i x/p}$.
  For $t=0$ (corresponding to the trivial character), we clearly have $F(0)=q^2$.
  For $t\neq0$, we split up the sum $F(t)$ into sub-sums $F(t,b)$ according to the choice of $b\in\F_q$.
  Clearly, $F(t,0)=q$.
  For $b\neq0$, the substitution $r \coloneqq a/b$ shows
  \begin{align*}
    F(t,b)
    &=
    \sum_{r\in\F_q}
      e_p\!\left(
        \sum_{i=1}^m
        \Tr\!\Big(
          t_i\sigma^i(b)b\cdot\big(\sigma^i(r)-r\big)
        \Big)
      \right) \\
    &=
    \sum_{r\in\F_q}
      e_p\!\left(
        \sum_{i=1}^m\Big(
          \Tr\!\big(
            t_i\sigma^i(b)b\sigma^i(r)
          \big) -
          \Tr\!\big(
            t_i\sigma^i(b) b r
          \big)
        \Big)
      \right) \\
    &=
    \sum_{r\in\F_q}
      e_p\!\left(
      \sum_{i=1}^m\Big(
        \Tr\!\left(\sigma^{-i}[t_i\sigma^i(b)b]r\right) -
        \Tr\!\left( t_i\sigma^i(b) b r \right)
      \Big)
    \right)
    & \textnormal{as the trace is $\sigma$-invariant} \\
    &=
    \sum_{r\in\F_q}
      e_p\!\left(
        \Tr\!\Big(
          \sum_{i=1}^m
            \big(
              \sigma^{-i}[t_i\sigma^i(b)b]-t_i\sigma^i(b)b
            \big)
          \cdot r
        \Big)
      \right).
  \end{align*}
  By orthogonality of characters of $\F_q$ and non-degeneracy of the trace form, we obtain
  \begin{align*}
    F(t,b) &=
    \begin{cases}
      q &\textnormal{if } \sum_{i=1}^m \big(\sigma^{-i}[t_i\sigma^i(b)b]-t_i\sigma^i(b)b\big) = 0,\\
      0 &\textnormal{otherwise}.
    \end{cases}
  \end{align*}
  Applying the bijection $\sigma^m$ to the condition $\sum_{i=1}^m\big(\sigma^{-i}[t_i\sigma^i(b)b]-t_i\sigma^i(b)b \big) = 0$ turns it into the following polynomial equation in the variable $b$:
  \[
    \sum_{i=1}^m \left(\sigma^{m-i}(t_i) b^{p^m + p^{m-i}} - \sigma^m(t_i) b^{p^{m+i}+p^m}\right) = 0.
  \]
  For any fixed $t\neq0$, the left-hand side is a non-zero polynomial in $b$ of degree at most $p^{2m} + p^m$.
  Hence, for any $t\neq0$, there are at most $p^{2m} + p^m=\cO(1)$ values $b\in\F_q^\times$ with $F(t,b)\neq0$; as seen above, we have $F(t,b) = q$ in this case.
  We conclude that $F(t) = \sum_{b\in\F_q} F(t,b) = q + \sum_{b\in\F_q^\times} F(t,b) = \cO(q)$ for all $t\neq0$.
  Combining this with the fact that $F(0) = q^2$, we obtain
  \[
    |A_{k,m}(\F_q)|
    = \frac{1}{q^m} \left(q^{2k} + \cO(q^{m+k})\right)
    = q^{2k-m} + \cO(q^k).
    \qedhere
  \]
\end{proof}

\subsection{Local counting}
\label{subsn:local-counting-heisenberg}

In this subsection, we fix a finite field $\kappa$, and we use the estimates of the sizes of the sets $A_{k,m}(\kappa)$ obtained in the previous subsection in order to estimate the distribution of $H_k(\F_p)$-extensions of the local function field $\kappa\llpar\pi\rrpar$ (cf.~\Cref{rk:amk} and \Cref{lem:numberofD-numberofexts}).
More precisely, we prove the following lemma, in the spirit of \Cref{thm:local-counting}:

\begin{lemma}
  \label{lem:heisenberg-local-counting}
  Consider the local field $\fF = \kappa\llpar\pi\rrpar$.
  For any $v\geq0$, let $N(=v)$, $N(<v)$, $N(\leq v)$ be the number of $D\in\fg\otimes\cD$ such that $\lastjump(D)=v$, $\lastjump(D)<v$, or $\lastjump(D)\leq v$, respectively.
  We have:
  \begin{enumalpha}
  \item
    $
      N(=0)
      = N(<1)
      = 1
    $.
  \item
    $
      N(\leq1)
      = N(<1+p^{-k})
      = \prod_{i=1}^k(p^i+1)\cdot|\kappa|^{k+1}\cdot(1+\cO_k(|\kappa|^{-1}))
    $.
  \item
    $
      N(\leq1+p^{-m-1})
      = N(<1+p^{-m})
      = |\kappa|^{2k+1-m}\cdot(1+\cO_k(|\kappa|^{-1}))
    $
    \hfill
    for $0\leq m < k$.
  \item
    $
      N(<l+1)
      =
      \cO_k\!\left(
        |\kappa|^{(l-\lfloor l/p\rfloor)(2k+1)}
      \right)
    $
    \hfill
    for $l \geq 0$.
  \item
    $
      N(<l+p^{-m})
      =
      \cO_k\!\left(
        |\kappa|^{l(2k+1-m)} + |\kappa|^{(l-1)(2k+1)+1}
      \right)
    $
    \hfill
    for $2\leq l < p$ and $1\leq m\leq k$.
  \item
    $
      N(< l + l p^{-m})
      =
      \cO_k\!\left(
        |\kappa|^{l(2k+2-m)-1} + |\kappa|^{(l-1)(2k+1)+1}
      \right)
    $
    \hfill
    for $2\leq l < p$ and $1\leq m\leq k$.
  \end{enumalpha}
\end{lemma}

\begin{proof}
  Let $\boldsymbol{\pi}$ be the projection $\fg=\F_p^{2k}\oplus\F_p\rightarrow\F_p^{2k}$.
  \begin{enumalpha}
  \item
    See \iref{thm:local-counting}{item:count-unramified}.
  \item
    \Cref{eqn:count-anm} implies that $N(<1+p^{-m}) = |\kappa| \cdot |A_{k,m}(\kappa)|$, which by \Cref{lem:Anm-large-m} does not depend on $m$ as soon as $m \geq k$.
    Therefore, $N(\leq 1) = N(<1+p^{-k}) = |\kappa| \cdot |A_{k,k}(\kappa)|$.
    The claim follows using \Cref{lem:Anm-large-m,lem:count-maximal-isotropic}.
  \item
    By \Cref{prop:precise-slightly-ramified} and \Cref{eqn:count-anm}, we have $N(\leq 1+p^{-m-1}) = N(< 1+p^{-m}) = |\kappa| \cdot |A_{k,m}(\kappa)|$.
    The claim follows using \Cref{lem:small-m-counting}.
  \item
    See \Cref{prop:large-v-expp} (we have $l+1 - \lceil\frac{l+1}p\rceil = l - \lfloor \frac lp \rfloor$, and $r = 2k+1$ in this case).
  \item
    By \Cref{cor:large-Da-determined} and \iref{prop:cor-l-plus}{item:cor-l-plus-1}, $N(<l+p^{-m})$ is at most the number of tuples $(D_1,\dots,D_l)$ of elements of $\fg\otimes\kappa$ such that $[\sigma^i(D_l),D_a]=0$ for $i=1,\dots,m$ and $a=1,\dots,l$.
    We will first pick $D_l$ and then $D_1,\dots,D_{l-1}$.

    Let $V_i(D_l)$ be the $\kappa$-span of $\sigma(\boldsymbol{\pi}(D_l)),\dots,\sigma^i(\boldsymbol{\pi}(D_l))$.
    If $V_i(D_l)=V_{i+1}(D_l)$ for some~$i$, then $V_i(D_l)$ is stable under $\sigma$ and hence defined over $\F_p$ by Galois descent for vector spaces; this implies that $V_j(D_l) = V_i(D_l)$ for all $j\geq i$.
    Let $V(D_l) \coloneq \bigcup_{i\geq1} V_i(D_l)$, which is defined over $\F_p$ for the same reason.
    Let $d(D_l) \coloneq \dim(V(D_l))$.
    We then have $\dim(V_i(D_l)) = \min(i, d(D_l))$.

    The conditions $[\sigma^i(D_l),D_a]=0$ for $i=1,\dots,m$ mean that~$\boldsymbol{\pi}(D_a)$ has to lie in the orthogonal complement of $V_m(D_l)$ with respect to the alternating bilinear form~$f_k$.
    This orthogonal complement has dimension $2k-\dim(V_m(D_l)) = 2k-\min(m,d(D_l))$, so the number of valid $D_a$ is $|\kappa|^{2k+1-\min(m,d(D_l))}$ for $a=1,\dots,l-1$.

    Focusing first on tuples $(D_1,\dots,D_l)$ with $d(D_l)\geq m$, we bound the number of $D_l$ such that $[\sigma^i(D_l),D_l]=0$ for $i=1,\dots,m$ using \Cref{lem:Anm-large-m,lem:count-maximal-isotropic,lem:small-m-counting}: the number of such~$D_l$ is $\cO_k(|\kappa|^{2k+1-m})$, so the total number of valid tuples $(D_1,\dots,D_l)$ with $d(D_l)\geq m$ is $\cO_k(|\kappa|^{(l-1)(2k+1-m)+(2k+1-m)}) = \cO_k(|\kappa|^{l(2k+1-m)})$.

    Now, we fix some $d<m$ and focus on the case $d(D_l)=d$.
    In this case, we will have more choices for $D_1,\dots,D_{l-1}$.
    However, we will have fewer choices for $D_l$, as $\boldsymbol\pi(D_l)$ must lie in the $d$-dimensional subspace $V(D_l)$, which is defined over $\F_p$.
    For any of the $\cO_k(1)$ $d$-dimensional subspaces $V$ of $\kappa^{2k}$ defined over $\F_p$, there are $|\kappa|^{d+1}$ choices of $D_l$ such that $\boldsymbol\pi(D_l)\in V$, so there are at most $|\kappa|^{d+1}$ choices of $D_l$ such that $V(D_l) = V$.
    Hence, for any given dimension $d$, we have $\cO_k(|\kappa|^{(l-1)(2k+1-d)+(d+1)})$ valid tuples $(D_1,\dots,D_l)$ with $d(D_l) = d$.
    The exponent $(l-1)(2k+1-d)+(d+1)$ is maximal for $d=0$.
    Thus, the number of valid tuples $(D_1,\dots,D_l)$ with $d(D_l)<m$ is $\cO_k(|\kappa|^{(l-1)(2k+1)+1})$.
  \item
    By \Cref{cor:large-Da-determined} and \iref{prop:cor-l-plus}{item:cor-l-plus-l}, $N(<l+lp^{-m})$ is at most the number of tuples $(D_1,\dots,D_l)$ of elements of $\fg\otimes\kappa$ such that $[\sigma^i(D_l),D_a]=0$ for $i=1,\dots,m-1$ and $a=1,\dots,l-1$ and such that $[\sigma^i(D_l),D_l]=0$ for $i=1,\dots,m$.
    We distinguish the same two types of tuples as before:

    If $d(D_l)\geq m$, then $V_{m-1}(D_l)$ has dimension $m-1$, and thus the number of valid tuples with $d(D_l)\geq m$ is $\cO_k(|\kappa|^{(l-1)(2k+1-(m-1))+(2k+1-m)}) = \cO_k(|\kappa|^{l(2k+2-m)-1})$.

    If $d(D_l)<m$, then $V_{m-1}(D_l)=V_m(D_l)$, so the number of valid tuples with $d(D_l)<m$ is $\cO_k(|\kappa|^{(l-1)(2k+1)+1})$ as before.
  \qedhere
  \end{enumalpha}
\end{proof}

\subsection{Global asymptotics}
\label{subsn:heisenberg-global-asymptotics}

Define the following numbers:
\begin{align*}
  n(m) &\coloneq 1+p^{-m-1}, &
  e'_m &\coloneq 2k+1-m, &
  b'_m &\coloneq 1 &
  \textnormal{for }0\leq m\leq k-1,\\
  n(k) &\coloneq 1, &
  e'_k &\coloneq k+1, &
  b'_k &\coloneq \prod_{i=1}^k(p^i+1).
\end{align*}
Define $A\coloneq\max\{\frac{e'_m + 1}{n(m)}\mid 0\leq m\leq k\}$ and $S' = \suchthat{0\leq m\leq k}{\frac{e'_m + 1}{n(m)} = A}$.
Let $B\coloneq\sum_{m \in S'} b'_m$, and let $M$ be the least common integer multiple of the rational numbers $b'_m$ for $m \in S'$.

\begin{theorem}[cf.~\Cref{thm:intro-heisenberg}]
  \label{thm:heisenberg-count}
  There is a function $C:\Q/M\Z\rightarrow\R_{\geq0}$ with $C(0)\neq0$, such that for rational $N\rightarrow\infty$, we have
  \[
    \sum_{\substack{
      K\in\Ext(H_k(\F_p), \, F):\\
      \lastjump(K)=N
    }}
      \frac1{|\Aut(K)|}
    \quad=\quad
    C\!\left(N\bmod M\right)\cdot
    q^{AN}
    N^{B-1}
    +
    o\!\left(
      q^{AN}
      N^{B-1}
    \right).
  \]
\end{theorem}

\begin{proof}
  We apply \Cref{lem:analytic-lemma} just as in \Cref{subsn:proof-main-thm} (in the proof of \Cref{thm:intro-counting}), and with the same notation, but using the following estimates arising from \Cref{lem:heisenberg-local-counting}:
  \bgroup
  \allowdisplaybreaks
  \begin{align*}
    a_{P,n} &= 1
    && \textnormal{for }n=0, \\
    a_{P,n} &= 0
    && \textnormal{for }0<n<1, \\
    a_{P,n} &= \prod_{i=1}^k (p^i+1)\cdot|\kappa_P|^{k+1} + \cO(|\kappa_P|^k)
    && \textnormal{for }n=1, \\
    a_{P,n} &= 0
    && \textnormal{for }1<n<1+p^{-k}, \\
    a_{P,n} &= |\kappa_P|^{2k+1-m} + \cO(|\kappa_P|^{2k-m})
    && \textnormal{for }n=1+p^{-m-1}\textnormal{ with }0\leq m\leq k-1, \\
    a_{P,n} &= 0
    && \textnormal{for }1+p^{-m-1}<n<1+p^{-m}\\&&&\textnormal{ with }0\leq m\leq k-1, \\
    a_{P,n} &= \cO\!\left(|\kappa_P|^{\max\!\big(l(k+1),\, (l-1)(2k+1)+1\big)}\right) && \textnormal{for }l\leq n<l+p^{-k}\textnormal{ with }2\leq l\leq p-1,
    \tag*{[Case I]}\label{case-1}\\
    a_{P,n} &= \cO\!\left(|\kappa_P|^{\max\!\big(l(2k+2-m)-1,\, (l-1)(2k+1)+1\big)}\right) && \textnormal{for }l+p^{-m}\leq n<l+l p^{-m}\\&&&\textnormal{ with }2\leq l\leq p-1\textnormal{ and }1\leq m\leq k,
    \tag*{[Case II]}\label{case-2} \\
    a_{P,n} &= \cO\!\left(|\kappa_P|^{\max\!\big(l(2k+1-m),\, (l-1)(2k+1)+1\big)}\right) && \textnormal{for }l+l p^{-m-1}\leq n<l+p^{-m}\\&&&\textnormal{ with }2\leq l\leq p-1\textnormal{ and }0\leq m\leq k-1,
    \tag*{[Case III]}\label{case-3} \\
    a_{P,n} &= \cO\!\left(|\kappa_P|^{(l-\lfloor l/p\rfloor)(2k+1)}\right) && \textnormal{for }l\leq n<l+1\textnormal{ with }p\leq l.
    \tag*{[Case IV]}\label{case-4}
  \end{align*}
  The numbers $A, B, M$ defined above are the same as the numbers $A,B,M$ defined in \Cref{lem:analytic-lemma}.
  Inequality~(\ref{eqn:hypothesis-on-sup}) is verified as follows:
  \begin{align*}
    &\textbf{For \ref{case-1} (first argument of max):}\\
    &\frac{l(k+1)+1}{l} < k+2 \leq A &&\textnormal{for }2\leq l\leq p-1, \\[8pt]
    &\textbf{For \ref{case-2} (first argument of max):}\\
    &\frac{l(2k+2-m)}{l+p^{-m}} < \left\{\begin{matrix}\dfrac{2k+2-m}{1+p^{-m-1}}&\textnormal{if }m\leq k-1,\\k+2&\textnormal{if }m=k\end{matrix}\right\} \leq A &&\textnormal{for }2\leq l\leq p-1\textnormal{ and }1\leq m\leq k, \\[8pt]
    &\textbf{For \ref{case-3} (first argument of max):}\\
    &\frac{l(2k+1-m)+1}{l+lp^{-m-1}} < \frac{2k+2-m}{1+p^{-m-1}} \leq A&&\textnormal{for }2\leq l\leq p-1\textnormal{ and }0\leq m\leq k-1, \\[8pt]
    &\textbf{For [Cases I--III] (second argument of max):}\\
    &\frac{(l-1)(2k+1)+2}{l} < k+2 \leq A &&\textnormal{for }l=2, \\[5pt]
    &\frac{(l-1)(2k+1)+2}{l} \leq \left(1-\frac1l\right)(2k+2) \\&\qquad < \left(1-\frac1p\right)(2k+2) < \frac{2k+2}{1+\frac1p} \leq A &&\textnormal{for }3\leq l\leq p-1, \\[8pt]
    &\textbf{For \ref{case-4}:}\\
    &\frac{(l-\lfloor l/p\rfloor)(2k+1)+1}{l}\leq \frac{l-\lfloor l/p\rfloor+1/3}{l}(2k+1)\\&\qquad \leq \frac{2p-2+1/3}{2p-1}(2k+1) < \left(1-\frac1{p^2}\right)(2k+1) \\&\qquad < \left\{\begin{matrix}\dfrac{2k+1}{1+p^{-2}}&\textnormal{if }k\geq2\\k+2&\textnormal{if }k=1\end{matrix}\right\} \leq A&&\textnormal{for }p\leq l.
    \qedhere
  \end{align*}
  \egroup
\end{proof}

Finally, we describe the integer $B$ more concretely:

\begin{proposition}
  \label{prop:B-is-often-1}
  If $(p,k)=(3,1)$, then $A=3$ and $B=5$.
  If $(p,k)=(3,\frac{3^{m+2}+2m-1}{4})$ for some $m\geq0$, then $A=\frac{e'_m + 1}{n(m)} = \frac{e'_{m+1} + 1}{n(m+1)}$ and $B=2$.
  In all other cases, $B=1$.
\end{proposition}

\begin{proof}
  The case $(p,k)=(3,1)$ is clear.
  Otherwise, at least one of the inequalities in
  \[
    \frac{e'_0 + 1}{n(0)} = \frac{2k+2}{1+p^{-1}} \geq \frac{3}{4}(2k+2) = \frac32(k+1) \geq k+2 = \frac{e'_k + 1}{n(k)},
  \]
  is strict, so that $k \not\in S'$, and then:
  \[
    A =
    \max\suchthat{
      \frac{e'_m + 1}{n(m)}
    }{
      0\leq m\leq k-1
    }
    =
    \max\suchthat{
      \frac{2k+2-m}{1+p^{-m-1}}
    }{
      0\leq m\leq k-1
    }.
  \]

  Consider the derivative
  \[
    \frac{\dd}{\dd m}\frac{2k+2-m}{1+p^{-m-1}}
    = \frac{-p^{m+1}-1+(2k+2-m)\log p}{p^{m+1}(1+p^{-m-1})^2}.
  \]
  The denominator is positive, and the numerator is a strictly decreasing function of $m$.
  Hence, there is a single threshold $0\leq r\leq k-1$ such that $\frac{2k+2-m}{1+p^{-m-1}}$ is strictly increasing for real $0\leq m\leq r$ and strictly decreasing for real $r\leq m\leq k-1$.

  This implies that $\frac{e'_m+1}{n(m)}$ can be maximal for at most two integers $0\leq m\leq k-1$, and these integers would need to be consecutive.
  Conversely, if two consecutive integers produce the same value $\frac{e'_m+1}{n(m)}$, then this value must be $A$.

  If $\frac{e'_m+1}{n(m)}$ is maximal for only one integer $0\leq m\leq k-1$, then $B=1$ as claimed, so assume $\frac{e'_m+1}{n(m)}=\frac{e'_{m+1}+1}{n(m+1)}$ with $0\leq m\leq k-2$.
  This is equivalent to
  \[
    \frac{2k+2-m}{1+p^{-m-1}} = \frac{2k+1-m}{1+p^{-m-2}},
  \]
  so $1 + p^{-m-2} = (2k+1-m)\left(p^{-m-1}-p^{-m-2}\right)$, and hence $p^{m+2} + 1 = (2k+1-m)(p-1)$.
  In particular, $p-1$ has to divide $p^{m+2}+1$, but $p^{m+2}+1\equiv2\mod p-1$, so $p=3$.
  Plugging $p=3$ back in and solving for $k$, we obtain $k = \frac{3^{m+2}+2m-1}{4}$.
\end{proof}

\clearpage
\appendix
\section*{Chart of notations}
\label{notation-chart}

{
  \renewcommand{\arraystretch}{1.2}
	\centering\hskip-1.5cm
	\begin{tabularx}{1.2\textwidth}{
		|>{\centering\arraybackslash}p{0.4cm}
		|>{\centering\arraybackslash}p{3.5cm}
		|>{\centering\arraybackslash}p{2.6cm}
		|>{\centering\arraybackslash}X|}
		\hline
    \bf &
		\bf Notation &
		\bf Reference &
		\bf Short description \\
		\hline
    \parbox[t]{2mm}{\multirow{6}{*}{\rotatebox[origin=c]{90}{\small\bf Everywhere}}}
    & $p, \sigma, \wp$ & & An odd prime, the absolute Frobenius, the map $\sigma - \mathrm{id}$ \\
    & $\Ext(G,F)$ & \Cref{subsn:extensions-and-cohomology} & The set of isomorphism classes of (étale) $G$-extensions of $F$ \\
    & $\lastjump(K)$ & \Cref{subsn:last-jump} & Last jump of an extension $K|F$ \\
    & $R^\perf$ & \Cref{subsn:perfect-closure} & Perfect closure of a ring $R$ \\
    & $W(R)$, $\Ver$ & \Cref{subsn:witt-vectors} & Witt vectors over $R$, the Verschiebung map \\
    & $W_n(R)$ & \Cref{subsn:witt-vectors} & Witt vectors of length $n$ over $R$ \\
    \hline
    \parbox[t]{2mm}{\multirow{9}{*}{\rotatebox[origin=c]{90}{\centering\small\bf Sections \ref{sn:local}--\ref{sn:more-groups}}}}
    & $G$ & & A finite $p$-group of nilpotency class $\leq2$ \\
		& $\fg$ & & The finite Lie $\Z_p$-algebra associated to $G$ \\
    & $\circ$ & & The group law on $\fg$ such that $G = (\fg, \circ)$ \\
    & $\bwp$ & & The multiplicative Artin--Schreier map $x \mapsto \sigma(x) \circ (-x)$ \\
    & $g.m$ & & $\sigma(g) \circ m \circ (-g)$, for $g,m \in \fg \otimes W(F^\perf)$ \\
    & $G[p], \fg[p]$ & \Cref{lem:p-torsion-is-group} & The $p$-torsion subgroup of $G$ (resp.~ideal of $\fg$) \\
    & $\Orb{\fg}{W(F^\perf)}{W(F^\perf)}$ & & The set of $(\fg \otimes W(F^\perf), \circ)$-orbits of $\fg \otimes W(F^\perf)$ \\
    & $\orb$ & & The bijection $H^1(\Gamma_F, G) \simto \Orb{\fg}{W(F^\perf)}{W(F^\perf)}$ \\
    & $\lastjump(D)$ & & Last jump of the $G$-extension of $F$ associated to $D \in \fg \otimes W(F^\perf)$ \\
    \hline
    \parbox[t]{2mm}{\multirow{12}{*}{\rotatebox[origin=c]{90}{\centering\small\bf \Cref{sn:local} and \Cref{sn:local-counting}}}}
		& $\fF, \kappa, \pi$ & & A local function field of char.~$p$, its residue field, a uniformizer \\
    & $\tilde\pi$ & & The Teichmüller representative of $\pi$ in $W(\fF)$ \\
    & $\cD^0$ & \Cref{def:cD-zero} & $W(\kappa)$-module spanned by $\tilde\pi^{-a}$ for $a \in \Onotp$ \\
    & $\Orb{\fg}{\cD^0}{W(\kappa)}$ & \Cref{prop:kappa-acts} & Set of $(\fg \otimes W(\kappa), \circ)$-orbits of $\fg \otimes \cD^0$ \\
    & $\alpha^0$ & \Cref{thm:local-fundamental-domain} & The bijection $\Orb{\fg}{W(\fF^\perf)}{W(\fF^\perf)} \simto \Orb{\fg}{\cD^0}{W(\kappa)}$ \\
    & $\mu_v(b)$ & \Cref{eq:def-mu} & Smallest integer $k\geq0$ such that $bp^k \geq v$ \\
    & $\cD$ & \Cref{def:cD} & $W(\kappa)$-module spanned by $\tilde\pi^{-a}$ for $a \in \notp$ \\
    & $\pr$ & \Cref{def:cD} & The projection~$\cD^0 \twoheadrightarrow \cD$ \\
    & $g.D$ & \Cref{prop:action-on-cD} & $D - \frac12[D, \sigma(g)+g]$, for $g \in \fg \otimes W(\kappa)$ and $D \in \fg \otimes \cD$ \\
    & $\Orb{\fg}{\cD}{W(\kappa)}$ & \Cref{def:orb-cD} & Set of $(\fg \otimes W(\kappa), \circ)$-orbits of $\fg \otimes \cD$ \\
    & $\alpha$ & \Cref{def:alpha} & The surjection $\Orb{\fg}{W(\fF^\perf)}{W(\fF^\perf)} \twoheadrightarrow \Orb{\fg}{\cD}{W(\kappa)}$ \\
    & $\lastjump(D)$ & \Cref{def:lastjump-cD} & Common last jump of all $G$-extensions associated to $D \in \fg \otimes \cD$ \\
		\hline
    \parbox[t]{2mm}{\multirow{4}{*}{\rotatebox[origin=c]{90}{\centering\small\bf \Cref{sn:global}}}}
		& $q, F$ & & A power of~$p$, the global function field $\F_q(T)$ \\
		& $P, F_P$ & & A place of $F$, the completion of $F$ at $P$ \\
		& $\pi_P, \cD_P, \alpha_P, \ldots$ & & The associated local objects at $P$ (cf.~\Cref{sn:local})\\
		& $\alpha$ & \Cref{eqn:def-globalpha} & The map $\Orb{\fg}{W(F^\perf)}{W(F^\perf)}
    \to
    \rprod
     \Orb{\fg}{\cD_P}{W(\kappa_P)}$  \\
		\hline
    \parbox[t]{2mm}{\multirow{2}{*}{\rotatebox[origin=c]{90}{\centering\small\bf \Cref{sn:more-groups}}}}
		& $H_k(\F_p), \fh_k$ & \Cref{subsn:heisenberg-groups} & Higher Heisenberg group, the corresponding Lie $\F_p$-algebra \\
		& $f_k$ & \Cref{subsn:heisenberg-groups} & The bilinear form $(\F_p^k)^2 \to \F_p$ inducing the Lie bracket of $\fh_k$ \\
		& $A_{k,m}(\F_q)$ & \Cref{eqn:def-anm} & Set of $x\in(\F_q^k)^2$ such that $f_k(\sigma^i(x),x)=0$ for $i=1,\dots,m$ \\
		\hline
	\end{tabularx}
}

We also sum up in a diagram the three main bijections which we construct in this paper:
\[
  \Ext(G,F)
  \underset{\textnormal{
    \Cref{lem:etext-bij-cohomology}
  }}{
    \overset\sim\longleftrightarrow
  }
  H^1(\Gamma_F, G)
  \underset{\textnormal{
    \Cref{thm:parametrization}
  }}{
    \overset{\orb}{
      \overset\sim\longrightarrow
    }
  }
  \Orb{\fg}{W(F^\perf)}{W(F^\perf)}
  \underbrace{
    \underset{\textnormal{
      \Cref{thm:local-fundamental-domain}
    }}{
      \overset{\alpha^0}{\overset\sim\longrightarrow}
    }
    \Orb{\fg}{\cD^0}{W(\kappa)}
  }_{\textnormal{case $F = \kappa\llpar\pi\rrpar$ local}}.
\]
When $F = \kappa\llpar\pi\rrpar$ is a local field, there is also the surjection $\alpha : \Orb{\fg}{W(F^\perf)}{W(F^\perf)} \twoheadrightarrow \Orb{\fg}{\cD}{W(\kappa)}$ from \Cref{def:alpha}, which is finite-to-one (fibers have size $|\fg \otimes W(\kappa)|$), and all elements in a given fiber correspond to extensions with the same last jump (cf.~\Cref{cor:irrelevance-D0}).

\renewcommand{\addcontentsline}[3]{}
\emergencystretch=1em
\bibliographystyle{alphaurl}
\bibliography{paper.bib}

\end{document}